\documentclass{article}

\PassOptionsToPackage{dvipsnames}{xcolor}
\PassOptionsToPackage{numbers,sort&compress}{natbib}

\usepackage[preprint]{neurips_2023}
\renewcommand{\citet}[1]{\cite{#1}}
\usepackage{microtype}
\usepackage{graphicx}
\usepackage{booktabs}
\usepackage{balance}
\usepackage{hyperref}

\usepackage{algorithm}
\usepackage{algorithmic}
\usepackage{amsmath}
\usepackage{amssymb}
\usepackage{mathtools}
\usepackage{mathrsfs}
\usepackage{amsthm}
\usepackage{comment}
\usepackage{enumitem}
\usepackage{subcaption}
\usepackage{nicefrac}
\usepackage{tikz}
\usetikzlibrary{arrows.meta,backgrounds,patterns,positioning,shapes.geometric,quotes}
\usepackage[capitalize,noabbrev]{cleveref}

\theoremstyle{plain}
\newtheorem{theorem}{Theorem}[section]

\newtheorem{lemma}[theorem]{Lemma}
\newtheorem{corollary}[theorem]{Corollary}
\newtheorem{assumption}{Assumption}[section]
\newtheorem{definition}[theorem]{Definition}
\newtheorem{remark}[theorem]{Remark}

\newtheorem{example}[theorem]{Example}

\let\originalleft\left
\let\originalright\right
\renewcommand{\left}{\mathopen{}\mathclose\bgroup\originalleft}
\renewcommand{\right}{\aftergroup\egroup\originalright}
\newcommand{\norm}[1]{\left\lVert#1\right\rVert}
\newcommand{\abs}[1]{\left\lvert#1\right\rvert}
\newcommand{\quadratic}{q}

\newcommand{\bR}{\mathbb{R}}
\newcommand{\decayfactor}{\rho}
\newcommand{\bP}[2][]{\Pr\ifthenelse{\isempty{#1}}{}{_{#1}}\left[#2\right]}
\newcommand{\bE}[2][]{\mathop\mathbb{E}\ifthenelse{\isempty{#1}}{}{_{#1}}\left[#2\right]}
\newcommand{\bI}[2][]{\mathop\mathbb{I}\ifthenelse{\isempty{#1}}{}{_{#1}}\left[#2\right]}
\newcommand{\Var}[2][]{\mathbf{Var}\ifthenelse{\isempty{#1}}{}{_{#1}}\left[#2\right]}

\DeclareMathOperator*{\argmin}{arg\,min}

\newcommand{\zero}{\mathbf{0}}
\newcommand{\one}{\mathbf{1}}

\newcommand{\MPC}{\mathsf{MPC}}

\newcommand{\dist}{\mathsf{dist}}
\newcommand{\diam}{\mathsf{diam}}
\newcommand{\ALG}{\pi}

\newcommand{\cts}{\mathrm{cts}}
\newcommand{\dis}{\mathrm{dis}}

\newcommand{\controlnum}{k}
\newcommand{\controlidx}{j}

\allowdisplaybreaks

\usepackage[disable,textsize=tiny]{todonotes}

\title{Online Adaptive Policy Selection in Time-Varying Systems: No-Regret via Contractive Perturbations}

\author{
    Yiheng Lin, James A. Preiss, Emile Anand, Yingying Li, Yisong Yue, Adam Wierman \\
    Department of Computing + Mathematical Sciences \\
    California Institute of Technology \\
    Pasadena, California, USA \\
  \texttt{$\{$yihengl, japreiss, eanand, yingli2, yyue, adamw$\}$@caltech.edu} \\}

\begin{document}
\maketitle

\begin{abstract}
We study online adaptive policy selection in systems with time-varying costs and dynamics. We develop the Gradient-based Adaptive Policy Selection (GAPS) algorithm together with a general analytical framework for online policy selection via online optimization. Under our proposed notion of contractive policy classes, we show that GAPS approximates the behavior of an ideal online gradient descent algorithm on the policy parameters while requiring less information and computation. When convexity holds, our algorithm is the first to achieve optimal policy regret. When convexity does not hold, we provide the first local regret bound for online policy selection. Our numerical experiments show that GAPS can adapt to changing environments more quickly than existing benchmarks.

\end{abstract}

\section{Introduction}\label{sec:Intro}
We study the problem of online adaptive policy selection for nonlinear time-varying discrete-time dynamical systems. The dynamics are given by $x_{t+1} = g_t(x_t, u_t)$, where $x_t$ is the state and $u_t$ is the control input at time $t$. The policy class is a time-varying mapping $\ALG_t$ from the state $x_t$ and a policy parameter $\theta_t$ to a control input $u_t$.
At every time step $t$, the online policy incurs a stage cost $c_t = f_t(x_t, u_t)$ that depends on the current state and control input.
The goal of policy selection is to pick the parameter $\theta_t$ online to minimize the total stage costs over a finite horizon $T$.

Online adaptive policy selection and general online control have received significant attention recently \cite{agarwal2019online,agarwal2019logarithmic,hazan2020nonstochastic,ho2021online,baby2022optimal, chen2021black, gradu2020adaptive, simchowitz2020improper} because many control tasks require running the policy on a single trajectory, as opposed to restarting the episode to evaluate a different policy from the same initial state.
Adaptivity is also important when the dynamics and cost functions are time-varying.
For example, in robotics, time-varying dynamics arise when we control a drone under changing wind conditions \cite{o2022neural}.

In this paper, we are interested in developing a unified framework that can leverage a broad suite of theoretical results from online optimization and efficiently translate them to online policy selection, where efficiency includes both preserving the tightness of the guarantees and computational considerations. A central issue is that, in online policy selection, the stage cost $c_t$ depends on all previously selected parameters $(\theta_0, \dots, \theta_{t-1})$ via the state $x_t$. Many prior works along this direction have addressed this issue by finite-memory reductions. This approach leads to the first regret bound on online policy selection, but the bounds are not tight, the computational cost can be large, and the dynamics and policy classes studied are restrictive \cite{agarwal2019online, hazan2020nonstochastic, chen2021black,simchowitz2020improper, gradu2020adaptive}. 

\textbf{Contributions.}
We propose and analyze the algorithm Gradient-based Adaptive Policy Selection (GAPS, Algorithm \ref{alg:OCO-with-parameter-update}) to address three limitations of existing results on online policy selection. First, under the assumption that $c_t$ is a convex function of $(\theta_0, \dots, \theta_t)$, prior work left a $\log T$ regret gap between OCO and online policy selection. We close this gap by showing that GAPS achieves the optimal regret of $O(\sqrt{T})$ (\Cref{thm:main-regret-bound-convex}).
Second, many previous approaches require oracle access to the dynamics/costs and expensive resimulation from imaginary previous states. In contrast, GAPS only requires partial derivatives of the dynamics and costs along the visited trajectory, and computes $O(\log T)$ matrix multiplications at each step.
Third, the application of previous results is limited to specific policy classes and systems because they require $c_t$ to be convex in $(\theta_0, \dots, \theta_t)$. We address this limitation by showing the first local regret bound for online policy selection when the convexity does not hold. Specifically, GAPS achieves the local regret of $O(\sqrt{(1 + V) T})$, where $V$ is a measure of how much $(g_t, f_t, \pi_t)$ changes over the entire horizon.

To derive these performance guarantees, we develop a novel proof framework based on a general exponentially decaying, or ``contractive'', perturbation property (\Cref{def:epsilon-exp-decay-perturbation-property}) on the policy-induced closed-loop dynamics.
This generalizes a key property of disturbance-action controllers \citep[e.g.][]{agarwal2019online, simchowitz2020improper} and includes other important policy classes such as model predictive control (MPC) \citep[e.g.][]{li2021robustness} and linear feedback controllers \citep[e.g.][]{qu2021exploiting}. Under this property, we prove an approximation error bound (\Cref{thm:bridge-GAPS-and-OGD}), which shows that GAPS can mimic the update of an ideal online gradient descent (OGD) algorithm \cite{zinkevich2003online} that has oracle knowledge of how the current policy parameter $\theta_t$ would have performed if used exclusively over the whole trajectory. This error bound bridges online policy selection and online optimization, which means regret guarantees on OGD for online optimization can be transferred to GAPS for online policy selection.

In numerical experiments, we demonstrate that GAPS can adapt faster than an existing follow-the-leader-type baseline in MPC with imperfect disturbance predictions,
and outperforms a strong optimal control baseline in a nonlinear system with non-i.i.d. disturbances.
We include the source code in the supplementary material.

\textbf{Related Work.} Our work is related to online control and adaptive-learning-based control \cite{agarwal2019logarithmic,li2021online,kakade2020information,dean2018regret,cohen2019learning,li2021safe,yu2022online}, especially online control with adversarial disturbances and regret guarantees \cite{agarwal2019online,chen2021black,simchowitz2020improper,pmlr-v119-foster20b,gradu2020adaptive,minasyan2021online}.
For example, there is a rich literature on policy regret bounds for time-invariant dynamics \cite{agarwal2019online,chen2021black,simchowitz2020improper,pmlr-v119-foster20b,wabersich2020performance,kakade2020information}.
There is also a growing interest in algorithms for time-varying systems with small adaptive regret \cite{gradu2020adaptive,minasyan2021online}, dynamic regret \cite{yu2020power,li2019online,lin2021perturbation}, and competitive ratio \cite{sabag2022optimal,yu2022competitive,goel2021competitive,shi2020online}. Many prior works study a specific policy class called disturbance-action controller (DAC) \cite{agarwal2019online, hazan2020nonstochastic, chen2021black,simchowitz2020improper, gradu2020adaptive}. When applied to linear dynamics $g_t$ with convex cost functions $f_t$, DAC renders the stage cost $c_t$ a convex function in past policy parameters $(\theta_0, \dots, \theta_t)$. Our work contributes to the literature by proposing a general contractive perturbation property that includes DAC as a special case, and showing local regret bounds that do not require $c_t$ to be convex in $(\theta_0, \dots, \theta_t)$. A recent work also handles nonconvex $c_t$, but it studies an episodic setting and requires $c_t$ to be ``nearly convex'', which holds under its policy class \cite{chen2023regret}.

In addition to online control,
this work is also related to online learning/optimization \cite{auer2002nonstochastic,zinkevich2003online,hazan2016introduction}, especially online optimization with memory and/or switching costs, where the cost at each time step depends on past decisions. Specifically, our online adaptive policy selection problem is related to online optimization with memory    \cite{anava2015online,shi2020online,rouyer2021algorithm,goel2019beyond,lin2020online,chen2018smoothed,li2020online,arora2012online}. Our analysis for GAPS provides insight on how to handle indefinite memory when the impact of a past decision decays exponentially with time.

Our contractive perturbation property and the analytical framework based on this property are closely related to prior works on discrete-time incremental stability and contraction theory in nonlinear systems \cite{boffi2021regret,shi2021meta,lohmiller1998contraction,angeli2002lyapunov,ruffer2013convergent,bayer2013discrete,tsukamoto2021contraction}, as well as works that leverage such properties to derive guarantees for (online) controllers \cite{li2022certifying,tsukamoto2021learning,tu2022sample}.  In complicated systems, it may be hard to design policies that provably satisfy these properties. This motivates some recent works to study neural-based approaches that can learn a controller together with its certificate for contraction properties simultaneously \cite{dawson2023safe,sun2021learning}. Our work contributes to this field by showing that, when the system satisfies the contractive perturbation property, one can leverage this property to bridge online policy selection with online optimization.

\textbf{Notation.} We use $[t_1\mathbin{:}t_2]$ to denote the sequence $(t_1, \ldots, t_2)$, $a_{t_1:t_2}$  to denote  $(a_{t_1}, a_{t_1+1}, \ldots, a_{t_2})$ for $t_1 \leq t_2$, and $a_{\times \tau}$  for  $(a, \ldots, a)$ with $a$ repeated $\tau\geq 0$ times. We define $\quadratic(x, Q)=x^\top Q x$. Symbols $\one$ and $\zero$  denote the all-one and all-zero vectors/matrices respectively, with dimension implied by context.
The Euclidean ball with center $\zero$ and radius $R$ in $\mathbb{R}^n$ is denoted by $B_n(0, R)$.
We let $\norm{\cdot}$ denote the (induced) Euclidean norm for vectors (matrices).
The diameter of a set $\Theta$ is  $\diam(\Theta)\coloneqq \sup_{x,y \in \Theta} \|x-y\|$.
The projection onto the set $\Theta$ is  $\Pi_{\Theta}(x) = \argmin_{y \in \Theta} \norm{y - x}$.

\section{Preliminaries}\label{sec:Preliminaries}
\begin{figure}
    \begin{minipage}[c]{0.64\linewidth}
    \centering
    \begin{tikzpicture}[
	xscale=0.8,
	yscale=0.40,
	node/.style={draw},
	nodet/.style={diamond, inner sep=0.3mm, draw},
	edge1/.style={-{Straight Barb[length=1mm, width=1mm]}, draw=blue},      
	edge2/.style={-Implies, double, draw=red},
	edge3/.style={-{Stealth[length=1.5mm, width=1.5mm]}, dotted, thick, draw=black},
        edge4/.style={double, dashed, draw=red},
	baseline=(x0),
]
	\node[node] (x0) at (0, 0) {$x_0$};
	\node[node] (x1) at (2, 0) {$x_1$};
	\node[node] (x2) at (4, 0) {$x_2$};
	\node[node] (u0) at (1, -2) {$u_0$};
	\node[node] (u1) at (3, -2) {$u_1$};
	\node[nodet] (th0) at (1, -4) {$\theta_0$};
	\node[nodet] (th1) at (3, -4) {$\theta_1$};
    \node[node] (c1) at (1, 2) {$c_0$};
    \node[node] (c2) at (3, 2) {$c_1$};
	\draw[edge1] (th0) -- (u0);
	\draw[edge1] (th1) -- (u1);
	\draw[edge2] (u0) -- (x1);
	\draw[edge2] (u1) -- (x2);
	\draw[edge2] (x0) -- (x1);
	\draw[edge2] (x1) -- (x2);
	\draw[edge1] (x0) -- (u0);
	\draw[edge1] (x1) -- (u1);
    \draw[edge3] (u0) -- (c1);
	\draw[edge3] (x0) -- (c1);
	\draw[edge3] (u1) -- (c2);
	\draw[edge3] (x1) -- (c2);

    \node[node] (ct) at (7, 2) {$c_t$};
    \node[node] (xt) at (6, 0) {$x_t$};
    \node[node] (xtplus) at (8, 0) {$x_{t+1}$};
    \node[node] (ut) at (7, -2) {$u_t$};
    \node[nodet] (tht) at (7, -4) {$\theta_t$};
    \draw[edge1] (tht) -- (ut);
    \draw[edge1] (xt) -- (ut);
    \draw[edge2] (xt) -- (xtplus);
    \draw[edge2] (ut) -- (xtplus);
    \draw[edge3] (xt) -- (ct);
    \draw[edge3] (ut) -- (ct);
    
    \draw[edge4] (x2) -- (xt);

\begin{scope}[xshift=110pt]
 \matrix [draw, nodes={anchor=center, minimum height=1.5    em}, below right] at (5.3,1.8) {
  \draw[edge1] (-0.3,0) -- (0.3,0); & \node {$\pi_t$}; \\
  \draw[edge2] (-0.3,0) -- (0.3,0); & \node {$g_t$}; \\
  \draw[edge3] (-0.3,0) -- (0.3,0); & \node {$f_t$}; \\
};
\end{scope}

\end{tikzpicture}
    \end{minipage}
    \hfill
    \begin{minipage}[c]{0.3\linewidth}
    \caption{
        Diagram of the causal relationships between states, policy parameters, control inputs, and costs.
    }\label{fig:decision-making-diagram}
    \end{minipage}
\end{figure}

We consider online policy selection on a single trajectory. The setting is a discrete-time dynamical system with state $x_t \in \mathbb{R}^n$ for time index ${t \in \mathcal T \coloneqq[0:T-1]}$. At time step $t \in \mathcal T$, the policy picks a control action $u_t \in \mathbb{R}^m$, and the next state and the incurred cost are given by:
\[\text{Dynamics:}\ \ x_{t+1} = g_t(x_t, u_t),\ \ \ \ \ \ \ \ \ \text{Cost:}\ \ c_t \coloneqq f_t(x_t, u_t),\]
respectively, where $g_t(\cdot, \cdot)$ is a time-varying dynamics function and $f_t(\cdot, \cdot)$ is a time-varying stage cost. The goal is to minimize the total cost $\sum_{t=0}^{T-1} c_t$.

We consider parameterized time-varying policies of the form of $u_t = \ALG_t(x_t, \theta_t)$, where $x_t $ is the current state at time step~$t$ and $\theta_t \in \Theta$ is the current policy parameter. $\Theta$ is a closed convex subset of $\mathbb{R}^d$. We assume the dynamics, cost, and policy functions $\{g_t, f_t, \pi_t\}_{t \in \mathcal{T}}$ are oblivious, meaning they are fixed before the game begins. The online policy selection algorithm optimizes the total cost by selecting $\theta_t$ sequentially. We illustrate how the policy parameter sequence $\theta_{0:T-1}$ affects the trajectory $\{x_t, u_t\}_{t \in \mathcal{T}}$ and per-step costs $c_{0:T-1}$ in \Cref{fig:decision-making-diagram}. The online algorithm has access to the partial derivatives of the dynamics $f_t$ and cost $g_t$ \emph{along the visited trajectory}, but does not have oracle access to the $f_t, g_t$ for arbitrary states and actions.

We provide two motivating examples for our setting.
\Cref{appendix:examples} contains more details and a third example.
The first example is learning-augmented Model Predictive Control, a generalization of \cite{li2020online}.

\begin{example}[MPC with Confidence Coefficients]\label{example:MPC-confidence}
Consider a linear time-varying (LTV) system
$g_t(x_t, u_t) = A_t x_t + B_t u_t + w_t$,
with time-varying costs $f_t(x_t, u_t) = \quadratic(x_t, Q_t) + \quadratic(u_t, R_t)$.
At time~$t$, the policy observes $\{A_{t:t+k-1}, B_{t:t+k-1}, Q_{t:t+k-1}, R_{t:t+k-1}, w_{t:t+k-1\mid t}\}$, where $w_{\tau\mid t}$ is a (noisy) prediction of the future disturbance $w_\tau$. Then, $\pi_t(x_t, \theta_t)$ commits the first entry of
\begin{equation}
\label{eq:mpc-opt-confidence}
\begin{split}
     \argmin_{u_{t:t+k-1\mid t}}& \sum_{\tau = t}^{t + k-1} f_\tau(x_{\tau|t}, u_{\tau|t}) + \quadratic(x_{t+k\mid t}, \tilde{Q}) \\
    \operatorname{s.t.} \; &x_{t\mid t} = x_t,
    \quad x_{\tau+1\mid t} = A_\tau x_{\tau\mid t} + B_\tau u_{\tau\mid t} +\lambda_t^{[\tau - t]} w_{\tau\mid t} : \; t \leq \tau < t\!+\!k,
\end{split}
\end{equation}
where $\theta_t = \big(\lambda_t^{[0]}, \lambda_t^{[1]}, \ldots, \lambda_t^{[k-1]}\big), \Theta = [0, 1]^k$ and $\tilde{Q}$ is a fixed positive-definite matrix. Intuitively, $\lambda_t^{[i]}$ represents our level of confidence in the disturbance prediction $i$ steps into the future at time step~$t$, with entry $1$ being fully confident and $0$ being not confident at all.
\end{example}

The second example studies a nonlinear control model motivated by \cite{li2022certifying,qu2021exploiting}.

\begin{example}[Linear Feedback Control in Nonlinear Systems]\label{example:nonlinear-control}
Consider a time-varying nonlinear control problem with dynamics $g_t(x_t, u_t) = A x_t + B u_t + \delta_t(x_t, u_t)$ and costs $f_t(x_t, u_t) = q(x_t, Q) + q(u_t, R)$. Here, the nonlinear residual $\delta_t$ comes from linearization and is assumed to be sufficiently small and Lipschitz. Inspired by \cite{qu2021exploiting}, we construct an online policy based on the optimal controller $u_t = - \bar{K} x_t$ for the linear-quadratic regulator $\mathrm{LQR}(A, B, Q, R)$. Specifically, we let $\pi_t(x_t, \theta_t) = - K(\theta_t) x_t$ where $K$ is a mapping from $\Theta$ to $\mathbb{R}^{n\times m}$ such that $\norm{K(\theta_t) - \bar{K}}$ is uniformly bounded. 
\end{example}

\subsection{Policy Class and Performance Metrics}\label{sec:performance-metrics}
In our setting, the state $x_t$ at time $t$ is uniquely determined by the combination of 1) a state $x_\tau$ at a previous time $\tau <t$, and 2) the parameter sequence $\theta_{\tau:t-1}$.
Similarly, the cost at time $t$ is uniquely determined by $x_\tau$ and $\theta_{\tau:t}$.
Since we use these properties often, we introduce the following notation.

\begin{definition}[Multi-Step Dynamics and Cost]\label{def:multi-step-dynamics}
The multi-step dynamics $g_{t\mid \tau}$ between two time steps $\tau \leq t$ specifies the state $x_t$ as a function of the previous state $x_\tau$ and previous policy parameters $\theta_{\tau: t-1}$. It is defined recursively, with the base case $g_{\tau\mid \tau}(x_\tau) \coloneqq x_\tau$ and the recursive case 
\[
    g_{t+1\mid \tau}(x_\tau, \theta_{\tau:t}) = g_t\left(z_t, \ALG_t\left(z_t, \theta_{t}\right)\right),\ \forall\,t\geq \tau, 
\]
in which $z_t \coloneqq g_{t\mid \tau}(x_\tau, \theta_{\tau:t-1})$.\footnote{$z_t$ is an auxiliary variable to denote the state at $t$ under initial state $x_\tau$ and parameters $\theta_{\tau:t}$.} The multi-step cost $f_{t\mid \tau}$ specifies the cost $c_t$ as function of $x_\tau$ and $\theta_{\tau:t}$. It is defined as
$f_{t\mid \tau}(x_\tau, \theta_{\tau:t}) \coloneqq f_t\left(z_t, \ALG_t\left(z_t, \theta_{t}\right)\right).$
\end{definition}
In this paper, we frequently compare the trajectory of our algorithm against the trajectory achieved by applying a fixed parameter $\theta$ since time step $0$, which we denote as $\hat{x}_t(\theta) \coloneqq g_{t\mid 0}(x_0, \theta_{\times t})$ and $ \hat{u}_t(\theta)\coloneqq \pi_t(\hat{x}_t(\theta), \theta)$. A related concept that is heavily used is the \textit{surrogate cost}~$F_t$, which maps a single policy parameter to a real number.

\begin{definition}[Surrogate Cost]\label{def:surrogate-cost}
The surrogate cost function is defined as $F_t(\theta) \coloneqq f_t(\hat{x}_t(\theta), \hat{u}_t(\theta))$.
\end{definition}
\Cref{fig:decision-making-diagram} shows the overall causal structure, from which these concepts follow.

To measure the performance of an online algorithm, we adopt the objective of \textbf{\textit{adaptive policy regret}}, which has been used by \cite{hazan2007adaptive, gradu2020adaptive}. It is a stronger benchmark than the static policy regret \cite{agarwal2019online,chen2021black} and is more suited to time-varying environments. We use $\{x_t, u_t, \theta_t\}_{t \in \mathcal{T}}$ to denote the trajectory of the online algorithm throughout the paper. The adaptive policy regret $R^A(T)$ is defined as the maximum difference between the cost of the online policy and the cost of the optimal fixed-parameter policy over any sub-interval of the whole horizon $\mathcal{T}$, i.e.,
\begin{align}\label{equ:adaptive-policy-regret}
    \textstyle
    R^A(T) \coloneqq
    \max_{I = [t_1:t_2] \subseteq \mathcal{T}} \left(\sum_{t\in I} f_t(x_t, u_t) - \inf_{\theta \in \Theta} \sum_{t \in I} F_t(\theta)\right).
\end{align}
In contrast, the (static) policy regret defined in \cite{chen2021black,agarwal2019online} restricts the time interval $I$ to be the whole horizon $\mathcal{T}$.
Thus, a bound on adaptive regret is strictly stronger than the same bound on static regret. This metric is particularly useful in time-varying environments like Examples \ref{example:MPC-confidence} and \ref{example:nonlinear-control} because an online algorithm must adapt quickly to compete against a comparator policy parameter that can change indefinitely with every time interval \cite[Section 10.2]{hazan2016introduction}. 

In the general case when surrogate costs $F_{0:T-1}$ are nonconvex, it is difficult (if not impossible) for online algorithms to achieve meaningful guarantees on classic regret metrics like $R^A(T)$ or static policy regret because they lack oracle knowledge of the surrogate costs.
Therefore, we introduce the metric of \textbf{\textit{local regret}}, which bounds the sum of squared gradient norms over the whole horizon:
\begin{align}\label{equ:local-policy-regret}
    \textstyle
    R^L(T) \coloneqq \sum_{t=0}^{T-1} \norm{\nabla F_t(\theta_t)}^2.
\end{align}
Similar metrics have been adopted by previous works on online nonconvex optimization \cite{hazan2017efficient}. Intuitively, $R^L(T)$ measures how well the online agent chases the (changing) stationary point of the surrogate cost sequence $F_{0:T-1}$. Since the surrogate cost functions are changing over time, the bound on $R^L(T)$ will depend on
how much the system $\{g_t, f_t, \pi_t\}_{t \in \mathcal{T}}$ changes over the whole horizon $\mathcal{T}$. We defer the details to \Cref{sec:nonconvex-surrogate}.

\subsection{Contractive Perturbation and Stability}

In this section, we introduce two key properties needed for our sub-linear regret guarantees in adaptive online policy selection. We define both with respect to trajectories generated by ``slowly'' time-varying parameters, which are easier to analyze than arbitrary parameter sequences. 

\begin{definition}\label{def:constrained-policy-parameter-seq}
We denote the set of policy parameter sequences with $\varepsilon$-constrained step size by
\[S_\varepsilon(t_1:t_2) \coloneqq \{\theta_{t_1:t_2} \in \Theta^{t_2 - t_1 + 1}\mid \norm{\theta_{\tau+1} - \theta_{\tau}} \leq \varepsilon, \forall \tau \in [t_1 \mathbin{:} t_2 - 1]\}.\]
\end{definition}

The first property we require is an exponentially decaying, or ``contractive'', perturbation property of the closed-loop dynamics of the system with the policy class. We now formalize this property.

\begin{definition}[$\varepsilon$-Time-varying Contractive Perturbation]\label{def:epsilon-exp-decay-perturbation-property}
The $\varepsilon$-time-varying contractive perturbation property holds for $R_C > 0, C > 0$, $\decayfactor \in (0, 1)$, and $\varepsilon \geq 0$ if, for any $\theta_{\tau:t-1} \in S_{\varepsilon}(\tau:t-1)$,
\begin{align*}
    \norm{g_{t\mid \tau}(x_\tau, \theta_{\tau:t-1}) - g_{t\mid \tau}(x_\tau', \theta_{\tau:t-1})} \leq C \decayfactor^{t - \tau} \norm{x_\tau - x'_\tau}
\end{align*}
holds for arbitrary $x_\tau, x_\tau' \in B_n(0, R_C)$ and time steps $\tau \leq t$.
\end{definition}

Intuitively, $\varepsilon$-time-varying contractive perturbation requires two trajectories starting from different states (in a bounded ball) to converge towards each other if they adopt the same slowly time-varying policy parameter sequence.
We call the special case of $\varepsilon = 0$ \textit{time-invariant contractive perturbation}, meaning the policy parameter is fixed.
Although it may be difficult to verify the time-varying property directly since it allows the policy parameters to change, we show in \Cref{lemma:from-static-to-time-varying} that time-invariant contractive perturbation implies that the time-varying version also holds for some small $\varepsilon > 0$.

The time-invariant contractive perturbation property is closely related to discrete-time incremental stability \citep[e.g.][]{bayer2013discrete} and contraction theory \citep[e.g.][]{tsukamoto2021contraction}, which have been studied in control theory. While some specific policies including DAC and MPC satisfy $\varepsilon$-time-varying contractive perturbation globally in linear systems, in other case it is hard to verify. Our property is local and thus is easier to establish for broader applications in nonlinear systems (e.g., \Cref{example:nonlinear-control}).

Besides contractive perturbation, another important property we need is the stability of the policy class, which requires $\pi_{0:T-1}$ can stabilize the system starting from the zero state as long as the policy parameter varies slowly. This property is stated formally below:

\begin{definition}[$\varepsilon$-Time-varying Stability]\label{def:epsilon-time-varying-stability}
The $\varepsilon$-time-varying stability property holds for $R_S > 0$ and $\varepsilon \geq 0$ if, for any $\theta_{\tau:t-1} \in S_{\varepsilon}(\tau:t-1)$, $\norm{g_{t\mid \tau}(0, \theta_{\tau:t-1})} \leq R_S$ holds for any time steps $t\geq \tau$.
\end{definition}

Intuitively, $\varepsilon$-time-varying stability guarantees that the policy class $\pi_{0:T-1}$ can achieve stability if the policy parameters $\theta_{0:T-1}$ vary slowly.\footnote{This property is standard in online control and is satisfied by DAC \cite{agarwal2019online, hazan2020nonstochastic, chen2021black,simchowitz2020improper, gradu2020adaptive} as well as Examples \ref{example:MPC-confidence} \& \ref{example:nonlinear-control}.}
Similarly to contractive perturbation, one only needs to verify time-invariant stability (i.e., $\varepsilon = 0$ and the policy parameter is fixed) to claim time-varying stability holds for some strictly positive $\varepsilon$ (see \Cref{lemma:from-static-to-time-varying}). The reason we still use the time-varying contractive perturbation and stability in our assumptions is that they hold for $\varepsilon = +\infty$ in some cases, including DAC and MPC with confidence coefficients.
Applying \Cref{lemma:from-static-to-time-varying} for those systems will lead to a small, overly pessimistic $\varepsilon$.

\subsection{Key Assumptions}\label{subsec:assumptions}

We make two assumptions about the online policy selection problem to achieve regret guarantees. 

\begin{assumption}\label{assump:Lipschitz-and-smoothness}
The dynamics $g_{0:T-1}$, policies $\ALG_{0:T-1}$, and costs $f_{0:T-1}$ are differentiable at every time step and satisfy that, for any convex compact sets $\mathcal{X} \subseteq \mathbb{R}^n, \mathcal{U} \subseteq \mathcal{R}^m$, one can find Lipschitzness/smoothness constants (can depend on $\mathcal{X}$ and $\mathcal{U}$) such that:
\begin{enumerate}[wide, labelindent=0pt, nolistsep]
    \item The dynamics $g_t(x, u)$ is $(L_{g, x}, L_{g, u})$-Lipschitz and $(\ell_{g, x}, \ell_{g, u})$-smooth in $(x, u)$ on $\mathcal{X} \times \mathcal{U}$.
    \item The policy function $\ALG_t(x, \theta)$ is $(L_{\ALG, x}, L_{\ALG, \theta})$-Lipschitz and $(\ell_{\ALG, x}, \ell_{\ALG, \theta})$-smooth in $(x, \theta)$ on $\mathcal{X} \times \Theta$.
    \item The stage cost function $f_t(x, u)$ is $(L_f, L_f)$-Lipschitz and $(\ell_{f, x}, \ell_{f, u})$-smooth in $(x, u)$ on $\mathcal{X} \times \mathcal{U}$.
\end{enumerate}
\end{assumption}

\Cref{assump:Lipschitz-and-smoothness} is general because we only require the Lipschitzness/smoothness of $g_t$ and $f_t$ to hold for bounded states/actions within $\mathcal{X}$ and $\mathcal{U}$, where the coefficients may depend on $\mathcal{X}$ and $\mathcal{U}$.
Similar assumptions are common in the literature of online control/optimization \cite{lin2021perturbation,shi2020online,li2022certifying}.

Our second assumption is on the contractive perturbation and the stability of the closed-loop dynamics induced by a slowly time-varying policy parameter sequence.

\begin{assumption}\label{assump:contractive-and-stability}
Let $\mathcal{G}$ denote the set of all possible dynamics/policy sequences $\{g_t, \pi_t\}_{t\in\mathcal{T}}$ the environment/policy class may provide. For a fixed $\varepsilon \in \mathbb{R}_{\geq 0}$, the $\varepsilon$-time-varying contractive perturbation (\Cref{def:epsilon-exp-decay-perturbation-property}) holds with $(R_C, C, \rho)$ for any sequence in $\mathcal{G}$. The $\varepsilon$-time-varying stability (\Cref{def:epsilon-time-varying-stability}) holds with $R_S < R_C$ for any sequence in $\mathcal{G}$. We assume that the initial state satisfies $\norm{x_0} < (R_C - R_S)/C$. Further, we assume that if $\{g, \pi\}$ is the dynamics/policy at an intermediate time step of a sequence in $\mathcal{G}$, then the time-invariant sequence $\{g, \pi\}_{\times T}$ is also in $\mathcal{G}$.
\end{assumption}

Compared to other settings where contractive perturbation holds globally \cite{agarwal2019online,simchowitz2020improper, zhang2021regret}, our assumption brings a new challenge because we need to guarantee the starting state stays within $B(0, R_C)$ whenever we apply this property in the proof. Therefore, we assume $R_C > R_S + C\norm{x_0}$ in \Cref{assump:contractive-and-stability}. Similarly, to leverage the Lipschitzness/smoothness property, we require ${\mathcal{X} \supseteq B(0, R_x)}$ where $R_x \geq C(R_S + C\norm{x_0}) + R_S$ and $\mathcal{U} = \{\pi(x, \theta) \mid x \in \mathcal{X}, \theta \in \Theta, \pi \in \mathcal{G}\}$. 
Since the coefficients in \Cref{assump:Lipschitz-and-smoothness} depend on $\mathcal{X}$ and $\mathcal{U}$, we will set $\mathcal{X} = B(0, R_x)$ and $R_x = C(R_S + C\norm{x_0}) + R_S$ by default when presenting these constants.

For some systems, verifying \Cref{assump:contractive-and-stability} is straightforward (e.g., Example \ref{example:MPC-confidence}).  In other cases, we can rely on the following lemma, which can convert a time-invariant version of the property to general time-varying one. We defer its proof to \Cref{appendix:weak-contractive-to-strong-contractive}.

\begin{lemma}\label{lemma:from-static-to-time-varying}
Suppose \Cref{assump:contractive-and-stability} holds for $\varepsilon = 0$ and $(R_C, C, \rho, R_S)$, which satisfies $R_C > (C + 1)R_S$. Suppose \Cref{assump:Lipschitz-and-smoothness} also holds and let $\mathcal{X} \coloneqq B(0, R_x)$, where $R_x = (C + 1)^2 R_S$. Then, \Cref{assump:contractive-and-stability} also holds for $\hat{\varepsilon} > 0$, $(\hat{R}_C, \hat{C}, \hat{\rho}, \hat{R}_S)$, and $x_0$ that satisfies $(\hat{R}_C - \hat{R}_S)/C$. Here, $\hat{R}_S, \hat{R}_C, \hat{\rho}$ are arbitrary constants that satisfies $R_S < \hat{R}_S < \hat{R}_C < R_C/(C + 1)$ and $\rho < \hat{\rho} < 1$.
The positive constants $\hat{\varepsilon}$ and $\hat{C}$ are given detailed expressions in \Cref{appendix:weak-contractive-to-strong-contractive}.
\end{lemma}

\section{Method and Theoretical Results}\label{sec:continuous}
Our algorithm, Gradient-Based Adaptive Policy Selection (GAPS), is inspired by the classic online gradient descent (OGD) algorithm  \cite{hazan2016introduction, bansal2019potential}, with a novel approach for approximating the gradient of the surrogate stage cost $F_t$.
In the context of online optimization, OGD works as follows. At each time  $t$, the current stage cost describes how good the learner's current decision $\theta_t$ is.
The learner updates its decision by taking a gradient step with respect to this cost.
Mapping this intuition to online policy selection,
the \textit{ideal} OGD update rule would be the following.
\begin{definition}[Ideal OGD Update]\label{def:ideal-OGD}
At time step $t$, update $\theta_{t+1} = \prod_{\Theta}(\theta_t - \eta \nabla F_t(\theta_t))$.
\end{definition}
This is because the surrogate cost $F_t$ (\Cref{def:surrogate-cost}) characterizes how good $\theta_t$ is for time $t$ if we had applied $\theta_t$ from the start, i.e., without the impact of other historical policy parameters $\theta_{0:t-1}$.
However, since the complexity of computing $\nabla F_t$ exactly grows proportionally to $t$, the ideal OGD becomes intractable when the horizon $T$ is large.
As outlined in Algorithm \ref{alg:OCO-with-parameter-update}, GAPS
uses $G_t$  to approximate $\nabla F_t(\theta_t)$ efficiently.
To see this, we compare the decompositions, with key differences highlighted in colored text:
\begin{equation}
\label{equ:compare-G-t-and-ideal-gradient}
    \nabla F_t(\theta_t) =
    \sum_{b = 0}^{\textcolor{red}{t}}
        \left. \frac{\partial f_{t\mid 0}}{\partial \theta_{t - b}}
        \right|_{x_0, \textcolor{blue}{(\theta_t)_{\times (t+1)}}}
    \text{\quad and \quad }
    G_t =
    \sum_{b = 0}^{\textcolor{red}{\min\{B - 1, t\}}}
        \left. \frac{\partial f_{t\mid 0}}{\partial \theta_{t - b}}
        \right|_{x_0, \textcolor{blue}{\theta_{0:t}}}.
\end{equation}
GAPS uses two techniques to efficiently approximate $\nabla F_t(\theta_t)$.
First, we \textit{replace the ideal sequence $\textcolor{blue}{(\theta_t)_{\times (t+1)}}$ by the actual sequence $\textcolor{blue}{\theta_{0:t}}$.} This enables computing gradients along the actual trajectory experienced by the online policy without re-simulating the trajectory under $\theta_t$. Second, we \textit{truncate the whole historical dependence to \textcolor{red}{at most $B$} steps}. 
This bounds the memory used by GAPS.

\begin{algorithm}
\caption{Gradient-based Adaptive Policy Selection (GAPS)}\label{alg:OCO-with-parameter-update}
\begin{algorithmic}[1]
\REQUIRE Learning rate $\eta$, buffer length $B$, initial $\theta_0$.
\FOR{$t = 0, \ldots, T-1$}
    \STATE Observe the current state $x_t$.
    \STATE Pick the control action $u_t = \ALG_t(x_t, \theta_t)$.
    \STATE Incur the stage cost $c_t = f_t(x_t, u_t)$.
    \STATE Compute the approximated gradient:
    \(\displaystyle G_t = \sum_{b = 0}^{\min\{B - 1, t\}} \left. \frac{\partial f_{t\mid 0}}{\partial \theta_{t - b}}\right|_{x_0, \theta_{0:t}}.\)\label{alg:OCO-with-parameter-update:grad-approx}
    \STATE Perform the update $\theta_{t+1} = \prod_{\Theta}(\theta_t - \eta G_t)$.
\ENDFOR
\end{algorithmic}
\end{algorithm}

Algorithm \ref{alg:OCO-with-parameter-update} presents GAPS in its simplest form. 
Details of a time- and space-efficient implementation with $O(B)$ memory use are given in \Cref{alg:OCO-with-parameter-update-practical} of \Cref{appendix:practical-alg}.

Many previous works
\cite{agarwal2019online, hazan2020nonstochastic, chen2021black} take a different \textit{finite-memory reduction} approach toward reducing the online control problem to OCO with Memory \cite{anava2015online} by completely removing the dependence on policy parameters before time step $t - B$ for a fixed memory length $B$. 
In the finite-memory reduction, one must ``imaginarily'' reset the state at time $t - \tau$ to be $\zero$ and then use the $B$-step truncated multi-step cost function $f_{t\mid t - B}(\zero, \theta_{t-B:t})$ in the OGD with Memory algorithm \cite{agarwal2019online}. 
When applied to our setting, this is equivalent to replacing $G_t$ in line \ref{alg:OCO-with-parameter-update:grad-approx} of Algorithm \ref{alg:OCO-with-parameter-update} by $G_t' = \sum_{b = 0}^{B-1} \frac{\partial f_{t\mid t - B}}{\partial \theta_{t - b}} \mid_{0, (\theta_t)_{\times (B+1)}}$.
However, the estimator $G_t'$ has limitations compared with $G_t$ in GAPS.
First, computing $G_t'$ requires oracle access to the partial derivatives of the dynamics and cost functions for arbitrary state and actions.
Second, even if those are available, $G_t'$ is less computationally efficient than $G_t$ in GAPS, especially when the policy is expensive to execute.
Taking MPC (\Cref{example:MPC-confidence}) as an example, computing $G_t'$ at every time step requires solving $B$ MPC optimization problems when re-simulating the system, where $B = \Omega(\log T)$.
In contrast, computing $G_t$ in GAPS only requires solving one MPC optimization problem and $O(B)$ matrix multiplications to update the partial derivatives.

\subsection{Bounds on Truncation Error}\label{sec:main-assump-result}
We now present the first part of our main result, which states that the actual stage cost $f_t(x_t, u_t)$ incurred by GAPS is close to the ideal surrogate cost $F_t(\theta_t)$, and the approximated gradient $G_t$ is close to the ideal gradient $\nabla F_t(\theta_t)$.
In other words, GAPS mimics the ideal OGD update (\Cref{def:ideal-OGD}).

\begin{theorem}\label{thm:bridge-GAPS-and-OGD}
Suppose Assumptions \ref{assump:Lipschitz-and-smoothness} and \ref{assump:contractive-and-stability} hold. Let $\{(x_t, u_t, \theta_t)\}_{t \in \mathcal{T}}$ denote the trajectory of GAPS (Algorithm \ref{alg:OCO-with-parameter-update}) with buffer size $B$ and learning rate
$\eta \leq \Omega \left((1 - \rho)\varepsilon\right)$.
Then, we have
\begin{align*}\label{thm:main-regret-bound-convex:e01}
    \abs{f_t(x_t, u_t) - F_t(\theta_t)} = O\left({(1 - \decayfactor)^{-3}}{\eta}\right) \text{ and }
    \norm{G_t - \nabla F_t(\theta_t)} = O\left((1 - \decayfactor)^{-5}{\eta} + (1 - \decayfactor)^{-1}{\decayfactor^B}\right),
\end{align*}
where $\Omega(\cdot)$ and $O(\cdot)$ hide the dependence on the Lipschitz/smoothness constants defined in \Cref{assump:Lipschitz-and-smoothness} and $C$ in contractive perturbation
--- see details in \Cref{appendix:detailed-statements-proofs}.
\end{theorem}
We defer the proof of \Cref{thm:bridge-GAPS-and-OGD} to \Cref{appendix:detailed-statements-proofs}. Note that this result does not require any convexity assumptions on the surrogate cost $F_t$. 
\subsection{Regret Bounds for GAPS: Convex Surrogate Cost}\label{sec:convex-surrogate}
The second part of our main result studies the case when the surrogate cost $F_t$ is a convex function. This assumption is explicitly required or satisfied by the policy classes and dynamical systems in many prior works on online control and online policy selection \cite{agarwal2019online, hazan2020nonstochastic,zhang2021regret, chen2021black}.

The error bounds in \Cref{thm:bridge-GAPS-and-OGD} can reduce the problem of GAPS' regret bound in control to the problem of OGD's regret bound in online optimization, where the following result is well known: When the surrogate cost functions $F_t$ are convex, the ideal OGD update
(\Cref{def:ideal-OGD})
achieves the regret bound
when the step size $\eta$ is of the order $1/\sqrt{T}$ \cite{hazan2016introduction}.
By taking the biases on the stage costs and the gradients into consideration, we derive the adaptive regret bound in \Cref{thm:main-regret-bound-convex}.
Besides the adaptive regret, one can use a similar reduction approach to ``transfer'' other regret guarantees for OGD in online optimization to GAPS in control.
We include the derivation of a dynamic regret bound as an example in \Cref{appendix:gaps-dynamic-regret}.

\begin{theorem}\label{thm:main-regret-bound-convex}
Under the same assumptions as \Cref{thm:bridge-GAPS-and-OGD}, if we additionally assume $F_t$ is convex for every time $t$ and $\diam(\Theta)$ is bounded by a constant $D$, then GAPS achieves adaptive regret 
{
\begin{equation*}
    R^A(T) ={} O\big({\eta}^{-1} + (1 - \decayfactor)^{-5}{\eta T} + (1 - \decayfactor)^{-1}{\decayfactor^B} T
    + (1 - \decayfactor)^{-10}{\eta^3} T + (1 - \decayfactor)^{-2}{\decayfactor^{2B}\eta T}\big),\label{thm:main-regret-bound-convex:e02}
\end{equation*}
}
where $O(\cdot)$ hides the same constants as in \Cref{thm:bridge-GAPS-and-OGD} and $D$ --- see details in \Cref{appendix:detailed-statements-proofs}.
\end{theorem}
We discuss how to choose the learning rate and the regret it achieves in the following corollary.

\begin{corollary}\label{coro:main-regret-bounds-simplified}
Under the same assumptions as \Cref{thm:main-regret-bound-convex}, suppose the horizon length $T \gg \frac{1}{1 - \decayfactor}$ and the buffer length
$B \geq \frac{1}{2} \log(T) /\log(1/\decayfactor)$. If we set $\eta = (1 - \decayfactor)^{\frac{5}{2}} T^{-\frac{1}{2}}$, then GAPS achieves adaptive regret $R^A(T) = O((1-\decayfactor)^{-\frac{5}{2}}T^{\frac{1}{2}})$.
\end{corollary}
 
We defer the proof of \Cref{thm:main-regret-bound-convex} to \Cref{appendix:detailed-statements-proofs}. Compared to the (static) policy regret bounds of \citet{agarwal2019online, hazan2020nonstochastic}, our bound is tighter by a factor of $\log T$.
The key observation is that the impact of a past policy parameter $\theta_{t - b}$ on the current stage cost $c_t$ decays exponentially with respect to $b$ (see \Cref{appendix:gaps-outline} for details).
In comparison, the reduction-based approach first approximates $c_t$ with $\hat{c}_t$ that depends on $\theta_{t-B+1:t}$, and then applies general OCO with memory results on $\hat{c}_t$ \cite{agarwal2019online, hazan2020nonstochastic}. General OCO with memory cannot distinguish the different magnitudes of the contributions that $\theta_{t-B+1:t}$ make to $\hat{c}_t$, which leads to the regret gap of $B = O(\log T)$.

\subsection{Regret Bounds for GAPS: Nonconvex Surrogate Cost}\label{sec:nonconvex-surrogate}

The third part of our main result studies the case when the surrogate cost $F_t$ is nonconvex. Before presenting the result, we formally define the variation intensity that measures how much the system changes over the whole horizon.

\begin{definition}[Variation Intensity]\label{def:total-variation-policy-selection}
Let $\{g_t, \pi_t, f_t\}_{t\in \mathcal{T}}$ be a sequence of dynamics/policy/cost functions that the environment provides. The variation intensity $V$ of this sequence is defined as
\begin{small}
\[
    \sum_{t = 1}^{T-1}
    \sup_{x \in \mathcal{X}, u \in \mathcal{U}}
        \norm{g_t(x, u) - g_{t-1}(x, u)}
    + \sup_{x \in \mathcal{X}, \theta \in \Theta}
        \norm{\pi_t(x, \theta) - \pi_{t-1}(x, \theta)}
    + \sup_{x \in \mathcal{X}, u \in \mathcal{U}}
        \abs{f_t(x, u) - f_{t-1}(x, u)}.
\]
\end{small}
\end{definition}

Variation intensity is used as a measure of hardness for changing environments in the literature of online optimization that often appear in regret upper bounds (see \cite{mokhtari2016online} for an overview). \Cref{def:total-variation-policy-selection} generalizes one of the standard definitions
to online policy selection. Using this definition, we present our main result for GAPS applied to nonconvex surrogate costs using the metric of local regret \eqref{equ:local-policy-regret}.

\begin{theorem}\label{thm:GAPS-nonconvex-regret}
Under the same assumptions as \Cref{thm:bridge-GAPS-and-OGD}, if we additionally assume that $\Theta = \mathbb{R}^d$ for some integer $d$, then GAPS satisfies local regret
\[R^L(T) = O\left(\frac{1 + V}{(1 - \rho)^3 \eta} + \frac{\eta T}{(1 - \rho)^6} + \frac{\rho^B T}{(1 - \rho)^2} + \frac{\eta^3 T}{(1 - \rho)^{13}} + \frac{\rho^{2B} \eta T}{(1 - \rho)^5}\right),\]
where $O(\cdot)$ hides the same constants as in \Cref{thm:bridge-GAPS-and-OGD} --- see details in \Cref{appendix:gaps-nonconvex}.
\end{theorem}
We discuss how to choose the learning rate and the regret it achieves in the following corollary.

\begin{corollary}\label{coro:GAPS-nonconvex-regret-simplified}
Under the same assumptions as \Cref{thm:GAPS-nonconvex-regret}, suppose the horizon length $T \gg \frac{1}{1 - \decayfactor}$ and the buffer length
$B \geq \frac{1}{2} \log(T) /\log(1/\decayfactor)$. If we set $\eta = (1 - \decayfactor)^{\frac{3}{2}} (1 + V)^{\frac{1}{2}} T^{-\frac{1}{2}}$, GAPS achieves local regret $R^L(T) = O((1-\decayfactor)^{-\frac{9}{2}} (1 + V)^{\frac{1}{2}} T^{\frac{1}{2}})$.
\end{corollary}

We defer the proof of \Cref{thm:GAPS-nonconvex-regret} to \Cref{appendix:gaps-nonconvex}. Note that the local regret will be sublinear in $T$ if the variation intensity $V = o(T)$. To derive the local regret guarantee in \Cref{thm:GAPS-nonconvex-regret}, we address additional challenges compared to the convex case. First, we derive a local regret guarantee for OGD in online nonconvex optimization. We cannot directly apply results from the literature because they do not use ordinary OGD, and it is difficult to apply algorithms like Follow-the-Perturbed-Leader \citep[e.g.][]{suggala2020online} to online policy selection due to constraints on information and step size.
Then, to transfer the regret bound from online optimization to online policy selection,
we show how to convert the measure of variation defined on $F_{0:T-1}$ to our variation intensity $V$ defined on $\{g_t, \pi_t, f_t\}_{t\in \mathcal{T}}$.

A limitation of \Cref{thm:GAPS-nonconvex-regret} is that we need to assume $\Theta$ is a whole Euclidean space so that GAPS will not converge to a point at the boundary of $\Theta$ that is not a stationary point. \Cref{example:nonlinear-control} and \Cref{sec:nonlinear-control-details} show that one can re-parameterize the policy class to satisfy this assumption in some cases. Relaxing this assumption is our future work.

\section{Numerical Experiments}\label{sec:Numerics}
In this section we compare GAPS to strong baseline algorithms in settings based on \Cref{example:MPC-confidence,example:nonlinear-control}.
Details are deferred to \Cref{sec:numerical-appendix} due to space limitations.
\Cref{sec:numerical-appendix} also includes a third experiment comparing GAPS to a bandit-based algorithm for selecting the planning horizon in MPC. \footnote{The code for the experiments can be found here at \url{https://github.com/jpreiss/adaptive_policy_selection}.}

\begin{figure}[h]
    \centering
    \includegraphics[width=\textwidth]{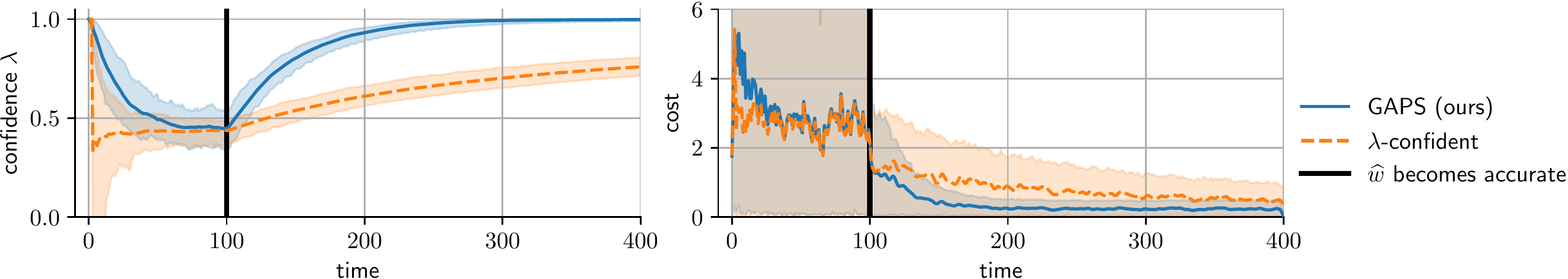}
    \caption{
        Comparing GAPS and baseline \citet{li2021robustness} for online adaptation of a confidence parameter for MPC with disturbance predictions.
        \emph{Left:} Confidence parameter. \emph{Right:} Per-step cost.
        Shaded bands show $10\%$-$90\%$ quantile range over randomized disturbance properties.
        See body for details.
    }
    \label{fig:ours-vs-lambda-confident}
\end{figure}

\paragraph{MPC confidence parameter.}
We compare GAPS to the follow-the-leader-type method of \citet{li2021robustness}
for tuning a scalar confidence parameter in model-predictive control
with noisy disturbance predictions.
The setting is close to \Cref{example:MPC-confidence}
but restricted to satisfy the conditions of the theoretical guarantees in \citet{li2021robustness}.
We consider the scalar system
\(
    x_{t+1} = 2 x_t + u_t + w_t
\)
under non-stochastic disturbances $w_t$
with the cost $f_t(x_t, u_t) = x_t^2 + u_t^2$.
For $t=0$ to $100$, the predictions of $w_t$ are corrupted by a large amount of noise.
After $t > 100$, the prediction noise is instantly reduced by a factor of $100$.
In this setup, an ideal algorithm should learn to decrease confidence level at first to account for the noise, but then increase to $\lambda \approx 1$ when the predictions become accurate.

\Cref{fig:ours-vs-lambda-confident}
shows
the values of the confidence parameter $\lambda$
and the per-timestep cost
generated by each algorithm.
Both methods are initialized to $\lambda = 1$. The method of \citet{li2021robustness} rapidly adjusts to an appropriate confidence level at first, while GAPS adjusts more slowly but eventually reaches the same value.
However, when the accuracy changes,
GAPS adapts more quickly. 
and obtains lower costs towards the end of the simulation.
In other words, we see that GAPS behaves essentially like an instance of OGD with constant step size, which is consistent with our theoretical results (\Cref{thm:bridge-GAPS-and-OGD}).

\begin{figure*}[h]
    \begin{subfigure}{0.63\textwidth}
        \includegraphics[width=\textwidth]{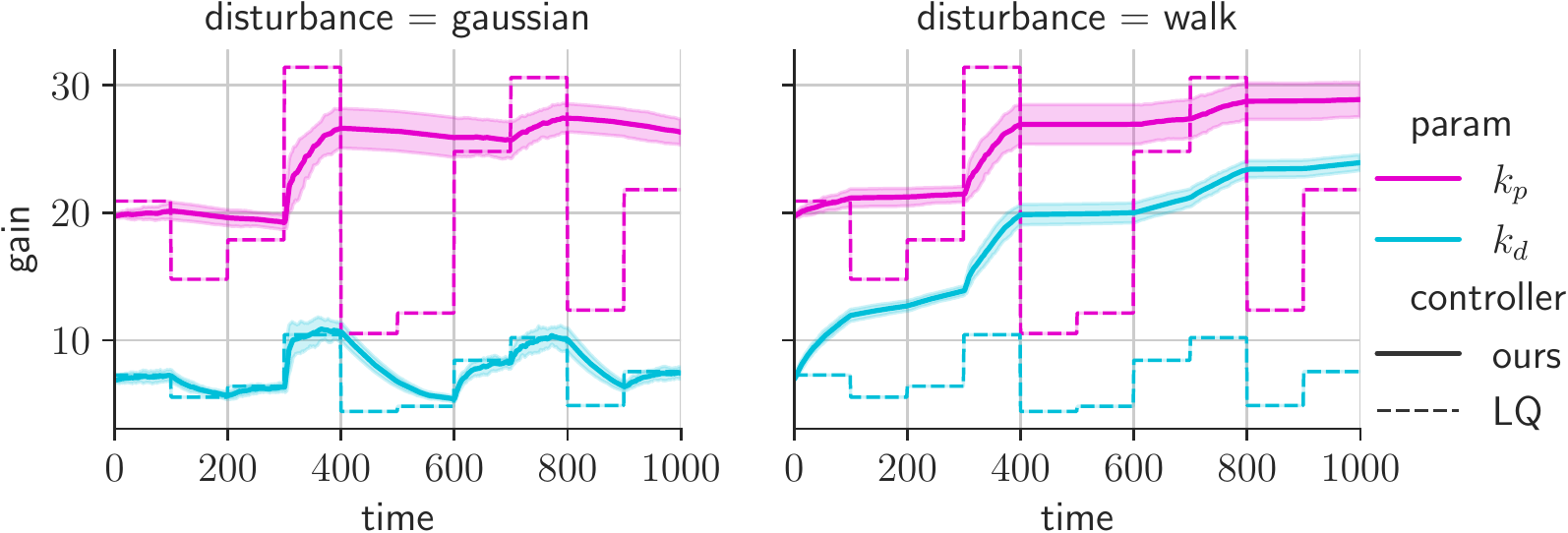}
        \caption{
            Linear gains $k_p,\ k_d$ tuned by GAPS compared to LQR optimal.
        }
        \label{fig:pendulum-params}
    \end{subfigure}
    \hfill
    \begin{subfigure}{0.335\textwidth}
        \hfill
        \includegraphics[width=\textwidth]{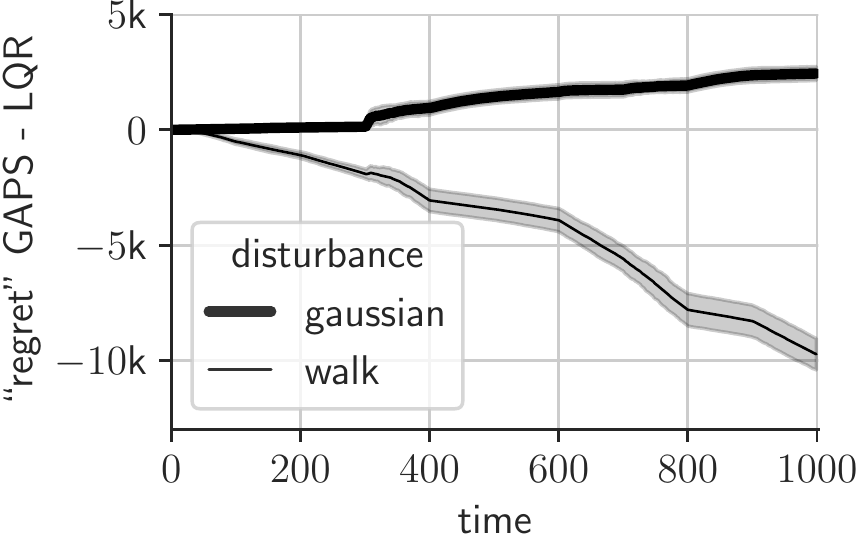}
        \caption{
            Cumulative cost difference.
        }
        \label{fig:pendulum-costs}
    \end{subfigure}
    \caption{
        Comparing GAPS and LQR baseline in nonlinear inverted pendulum system.
        Shaded bands show $\pm 1$ standard deviation over the randomness of the disturbances.
        See body for details.
    }
    \label{fig:pendulum}
\end{figure*}

\paragraph{Linear controller of nonlinear time-varying system.}
We apply GAPS to tune the gain parameters of a linear feedback controller in a nonlinear inverted pendulum system.
Every $100$ seconds, the pendulum mass changes.
The system reflects the smooth nonlinear dynamics and nonconvex surrogate costs in
\Cref{example:nonlinear-control}, although it differs in other details (see \Cref{sec:nonlinear-control-details,appendix:pendulum}).
We compare GAPS to a strong and dynamic baseline that deploys the infinite-horizon linear-quadratic regulator (LQR) optimal controller for the linearized dynamics at each mass.
We simulate two disturbances: 1)~i.i.d.\ Gaussian, and 2)~Ornstein-Uhlenbeck random walk.

\Cref{fig:pendulum-params} shows the controller parameters tuned by GAPS, along with the baseline LQR-optimal gains, for each disturbance type.
The derivative gain $k_d$ closely follows LQR for i.i.d.\ disturbances but diverges for random-walk disturbances, where LQR is no longer optimal.
This is reflected in the cumulative cost difference between GAPS and LQR,
shown in \Cref{fig:pendulum-costs}.
GAPS nearly matches LQR under i.i.d.\ disturbances,
but significantly outperforms it when the disturbance is a random walk.
The results show that GAPS can both
1) adapt to step changes in dynamics on a single trajectory almost as quickly as the comparator that benefits from knowledge of the near-optimal analytic solution,
and 2) outperform the comparator in more general settings where the analytic solution no longer applies.

\section{Conclusion and Future Directions}\label{sec:Conclusion}
In this paper, we study the problem of online adaptive policy selection under a general contractive perturbation property.
We propose GAPS, which can be implemented more efficiently and with less information than existing algorithms.
Under convexity assumptions, we show that GAPS achieves adaptive policy regret of $O(\sqrt{T})$, which closes the $\log T$ gap between online control and OCO left open by previous results.
When convexity does not hold, we show that GAPS achieves local regret of $O(\sqrt{(1 + V) T})$, where $V$ is the variation intensity of the time-varying system. 
This is the first local regret bound on online policy selection attained without any convexity assumptions on the surrogate cost functions.
Our numerical simulations demonstrate the effectiveness of GAPS, especially for fast adaptation in time-varying settings.

Our work motivates interesting future research directions. For example, a limitation is that GAPS assumes \textit{all} policy parameters can stabilize the system and satisfy contractive perturbation. A recent work on online policy selection relaxed this assumption by a bandit-based algorithm but requires $\Theta$ to be a finite set \citep{li2023online}.
An interesting future direction to study is what regret guarantees can be achieved when $\Theta$ is a continuous parameter set and not all of the candidate policies satisfy these assumptions.

\balance
\bibliographystyle{unsrt}
\bibliography{main}

\newpage
\appendix
\onecolumn
{\textbf{Outline of the appendices.}}
\vspace{-0.4cm}
\begin{itemize}[leftmargin=*]
    \item \Cref{appdx:math} 
    presents additional definitions and remarks that support the main body of the paper.
    \item \Cref{appendix:practical-alg} provides an efficient implementation of our algorithm GAPS for the continuous-parameter setting.
    \item
\Cref{appendix:weak-contractive-to-strong-contractive,%
appendix:gaps-outline,%
appendix:gaps-dynamic-regret,%
appendix:gaps-nonconvex}
contain the detailed proofs of our theoretical results.
    \item \Cref{appendix:finite} details a bandit-based algorithm for finite policy classes and its regret bound.
    \item \Cref{appendix:examples} provides additional details on our MPC and nonlinear control example problem settings, as well as a discussion of the disturbance-action controller popular in online control literature.
    \item \Cref{sec:numerical-appendix} contains 1) details on the numerical experiments in \Cref{sec:Numerics}, and
    2) a third numerical experiment whose results do not appear in \Cref{sec:Numerics}.
\end{itemize}

\section{Mathematical background and additional remarks}
\label{appdx:math}
\begin{definition}[Lipschitzness and Smoothness]\label{def:Lipschitz-and-smoothness}
 We say a function $h: \mathbb{R}^{n_1} \times \mathbb{R}^{n_2} \to \mathbb{R}^{n_3}$ is $(L_1, L_2)$-Lipschitz on a set $\mathcal{D} \in \mathbb{R}^{n_1} \times \mathbb{R}^{n_2}$ if for any $(y, z), (y', z') \in \mathcal{D}$, we have
\[
    \norm{h(y, z) - h(y', z')} \leq L_1 \norm{y - y'} + L_2 \norm{z - z'}.
\]
We say $h$ is $(\ell_1, \ell_2)$-smooth on $\mathcal{D}$ if $h$ is differentiable and for any $(y, z), (y', z') \in \mathcal{D}$, we have
\[
    \norm{\nabla h(y, z) - \nabla h(y', z')} \leq \ell_1 \norm{y - y'} + \ell_2 \norm{z - z'}.
\]
 \end{definition}

\begin{remark}[Relation to other notions of contraction/stability]
Contractive perturbation is closely related to incremental exponential stability in nonlinear control theory \cite{lohmiller1998contraction,angeli2002lyapunov,ruffer2013convergent}, which considers a time-varying dynamical system $x_{t+1}=h_t(x_t)$ and requires the following convergent property:
\[
\|x_{t}-x'_{t}\|\leq C\rho^{t-\tau}\|x_{\tau}-x'_{\tau}\|
\quad \text{for any} \quad
x_{\tau}, x_{\tau}'\in \mathbb{R}^n,\ x'_{t+1}=h_t(x'_t),\ t\geq \tau.
\]
Basically, the time-invariant contractive perturbation property requires that any closed-loop dynamics ${x_{t+1}=g_t(x_t,\ALG_t(x_t, \theta))}$ generated by policies with a fixed parameter $\theta \in \Theta$ guarantees incremental exponential stability.
A sufficient condition to verify the incremental exponential stability is the contraction analysis. See \cite{lohmiller1998contraction} for more details.
\end{remark}

\section{Practical Algorithm Implementation}\label{appendix:practical-alg}
In \Cref{alg:OCO-with-parameter-update-practical},
we present an equivalent version of the GAPS (\Cref{alg:OCO-with-parameter-update}) for practical implementation that specifies how to compute the partial derivative $\left.\frac{\partial f_{t\mid 0}}{\partial \theta_{t - b}}\right|_{x_0, \theta_{0:t}}$ efficiently using the buffered information.

\begin{algorithm}[htp]
\caption{Gradient-based Adaptive Policy Selection (For practical implementation)}\label{alg:OCO-with-parameter-update-practical}
\begin{algorithmic}[1]
\REQUIRE Learning rate $\eta$, buffer length $B$, initial parameter $\theta_0$.
\FOR{$t = 0, \ldots, T-1$}
    \STATE Observe the current state $x_t$.
    \STATE Pick control action $u_t = \ALG_t(x_t, \theta_t)$.
    \STATE Incur the stage cost $f_t(x_t, u_t)$.
    \STATE If $t \geq 1$, set $\frac{\partial x_t}{\partial \theta_{t-1}} \coloneqq \left.\frac{\partial g_{t-1}}{\partial u_{t-1}}\right|_{x_{t-1}, u_{t-1}} \cdot \left.\frac{\partial \ALG_{t-1}}{\partial \theta_{t-1}}\right|_{x_{t-1}, \theta_{t-1}}$.
    \STATE If $t \geq 1$, set $\frac{\partial x_t}{\partial x_{t-1}} \coloneqq \left.\frac{\partial g_{t-1}}{\partial x_{t-1}}\right|_{x_{t-1}, u_{t-1}} + \left.\frac{\partial g_{t-1}}{\partial u_{t-1}}\right|_{x_{t-1}, u_{t-1}} \cdot \left.\frac{\partial \ALG_{t-1}}{\partial x_{t-1}}\right|_{x_{t-1}, \theta_{t-1}}$.
    \FOR{$b = 2, \ldots, B-1$}
        \STATE If $t \geq b$, set $\frac{\partial x_t}{\partial \theta_{t-b}} \coloneqq \frac{\partial x_t}{\partial x_{t-1}} \cdot \frac{\partial x_{t-1}}{\partial \theta_{t-b}}$, where $\frac{\partial x_{t-1}}{\partial \theta_{t-b}}$ is in the buffer.
    \ENDFOR
    \STATE Compute the approximated gradient (ignore $\frac{\partial x_t}{\partial \theta_{t - b}}$ if $t < b$):
    \[G_t = \left(\left.\frac{\partial f_t}{\partial x_t}\right|_{x_t, u_t} + \left.\frac{\partial f_t}{\partial u_t}\right|_{x_t, u_t} \cdot \left.\frac{\partial \ALG_t}{\partial x_t}\right|_{x_t, \theta_t}\right) \cdot \sum_{b = 1}^{B - 1} \frac{\partial x_t}{\partial \theta_{t - b}} + \left.\frac{\partial f_t}{\partial u_t}\right|_{x_t, u_t} \cdot \left.\frac{\partial \ALG_t}{\partial \theta_t}\right|_{x_t, \theta_t}\]
    \STATE Perform the projected gradient update $\theta_{t+1} = \prod_{\Theta}(\theta_t - \eta G_t)$.
    \STATE Empty the buffer, and save $\left[\frac{\partial u_t}{\partial \theta_t}, \frac{\partial x_t}{\partial \theta_{t-1}}, \ldots, \frac{\partial x_t}{\partial \theta_{t-B+1}}\right]$ (if the term exists) into the buffer.
\ENDFOR
\end{algorithmic}
\end{algorithm}

\section{Proof of Lemma \ref{lemma:from-static-to-time-varying}}\label{appendix:weak-contractive-to-strong-contractive}
We first restate \Cref{lemma:from-static-to-time-varying} with detailed coefficients in \Cref{lemma:from-static-to-time-varying:detailed}:

\begin{lemma}\label{lemma:from-static-to-time-varying:detailed}
Suppose \Cref{assump:contractive-and-stability} holds for $\varepsilon = 0$ and $(R_C, C, \rho, R_S)$, which satisfies $R_C > (C + 1) R_S$. Suppose \Cref{assump:Lipschitz-and-smoothness} also holds and let $\mathcal{X} \coloneqq B_n(0, R_x)$, where $R_x = (C + 1)^2 R_S$. Then, \Cref{assump:contractive-and-stability} also holds for $\hat{\varepsilon} > 0$, $(\hat{R}_C, \hat{C}, \hat{\rho}, \hat{R}_S)$, and $x_0$ that satisfies the inequality $\norm{x_0} \leq (\hat{R}_C - \hat{R}_S)/C$. Here, $\hat{R}_S, \hat{R}_C, \hat{\rho}$ are arbitrary constants that satisfies $R_S < \hat{R}_S < \hat{R}_C < R_C/(C + 1)$ and $\rho < \hat{\rho} < 1$. Other coefficients are given by
\begin{align*}
    \hat{C} ={}& \left((1 + L_{\ALG, x}) \left(\ell_{g, x} + \ell_{g, u} L_{\ALG, x}\right) + L_{g, u}\ell_{\ALG, x}\right) \cdot \left(1 + L_{g, x} + L_{g, u} L_{\ALG, x}\right)^{2h} \hat{\rho}^{-h},\\
    \hat{\varepsilon} ={}& \min\left\{\frac{(\hat{\rho}^h - C \rho^h) (1 - \rho)^2}{C\cdot C' \rho \left(1 + L_{g, x} + L_{g, u} L_{\ALG, x}\right)^{2h}}, \frac{(1 - \rho)^2 (\hat{R}_S - R_S)}{C L_{g, u} L_{\pi, \theta}}\right\}, \text{ where}\\
    C' ={}& \bigg((1 + L_{\ALG, x}) \left(\ell_{g, x} + \ell_{g, u} \cdot (L_{\ALG, x} + L_{\ALG, \theta})\right) + L_{g, u} \cdot (\ell_{\ALG, x} + \ell_{\ALG, \theta})\bigg)(L_{g, u}L_{\ALG, \theta} + 1),
\end{align*}
where $h$ is a constant integer that satisfies $C \rho^h < \min\{\hat{\rho}^h, 1 - \hat{R}_S/\hat{R}_C\}$.
\end{lemma}

Before showing \Cref{lemma:from-static-to-time-varying:detailed}, we first show that the composition of Lipschitz and smooth functions is still Lipschitz and smooth in the following technical lemma:

\begin{lemma}\label{lemma:concatenation-of-smooth-function}
Suppose the sequence of functions $\iota_{1:t}$ satisfies that $\iota_i: D_i \to D_{i+1}$ is $L$-Lipschitz and $\ell$-smooth for all $i \in \{1, 2, \cdots, t\}$. Then, their composition $(\iota_t \circ \iota_{t-1} \circ \cdots \circ \iota_1)$ is $L^t$-Lipschitz and $\ell(1 + L)^{2t}$-smooth.
\end{lemma}
\begin{proof}[Proof of \Cref{lemma:concatenation-of-smooth-function}]
We show the conclusion by induction. For $t = 1$, $\iota_1$ is $L$-Lipschitz and $\ell(1 + L)$-smooth.

Suppose we have shown that $(\iota_t \circ \iota_{t-1} \circ \cdots \circ \iota_1)$ is $L^t$-Lipschitz and $\ell(1 + L)^{2t}$-smooth for any $t$ functions that satisfies the assumptions of \Cref{lemma:concatenation-of-smooth-function}. For $t+1$, we simplify the notation by defining $\hat{\iota} \coloneqq (\iota_{t+1} \circ \iota_{t} \circ \cdots \circ \iota_2)$. Our goal is to show $(\hat{\iota}\circ \iota_1)$ is $L^{t+1}$-Lipschitz and $\ell(1 + L)^{2(t+1)}$-smooth. Note that
\[\norm{(\hat{\iota}\circ \iota_1)(x) - (\hat{\iota}\circ \iota_1)(x')} \leq L^t \norm{\iota_1(x) - \iota_1(x')} \leq L^{t+1} \norm{x - x'},\]
where we use the induction assumption in the first inequality and the assumption of \Cref{lemma:concatenation-of-smooth-function} in the second inequality. We also see that
\begin{subequations}\label{lemma:concatenation-of-smooth-function:e1}
\begin{align}
    &\norm{\left.\frac{\partial (\hat{\iota} \circ \iota_1)}{\partial x}\right|_x - \left.\frac{\partial (\hat{\iota} \circ \iota_1)}{\partial x}\right|_{x'}}\nonumber\\
    ={}& \norm{\left.\frac{\partial \hat{\iota}}{\partial y}\right|_{\iota_1(x)} \cdot \left.\frac{\partial \iota_1}{\partial x}\right|_x - \left.\frac{\partial \hat{\iota}}{\partial y}\right|_{\iota_1(x')} \cdot \left.\frac{\partial \iota_1}{\partial x}\right|_{x'}}\label{lemma:concatenation-of-smooth-function:e1:s1}\\
    \leq{}&\norm{\left.\frac{\partial \hat{\iota}}{\partial y}\right|_{\iota_1(x)} - \left.\frac{\partial \hat{\iota}}{\partial y}\right|_{\iota_1(x')}} \cdot \norm{\left.\frac{\partial \iota_1}{\partial x}\right|_x} + \norm{\left.\frac{\partial \hat{\iota}}{\partial y}\right|_{\iota_1(x')}}\cdot \norm{\left.\frac{\partial \iota_1}{\partial x}\right|_x - \left.\frac{\partial \iota_1}{\partial x}\right|_{x'}}\nonumber\\
    \leq{}& \ell (1 + L)^{2t} \norm{\iota_1(x) - \iota_1(x')}\cdot L + L^t \cdot \ell\norm{x - x'}\label{lemma:concatenation-of-smooth-function:e1:s2}\\
    \leq{}& \ell \left(L^2 (1 + L)^{2t} + L^t\right)\norm{x - x'}\label{lemma:concatenation-of-smooth-function:e1:s3}\\
    \leq{}& \ell (1 + L)^{2(t+1)}\norm{x - x'},\nonumber
\end{align}
\end{subequations}
where we use the chain rule decomposition in \eqref{lemma:concatenation-of-smooth-function:e1:s1}; we use the induction assumption in \eqref{lemma:concatenation-of-smooth-function:e1:s2}; we use the assumption that $\iota_1$ is $L$-Lipschitz in \eqref{lemma:concatenation-of-smooth-function:e1:s3}.

Therefore, we have shown \Cref{lemma:concatenation-of-smooth-function} by induction.
\end{proof}

Now we are ready to show \Cref{lemma:from-static-to-time-varying:detailed}. 

We first discuss the intuition behind the proof. Since the multi-step dynamics is differentiable under \Cref{assump:Lipschitz-and-smoothness}, we only need to show an upper bound of $\norm{\left.\frac{\partial g_{t\mid \tau}}{\partial x_\tau}\right|_{x_\tau, \theta_{\tau:t-1}}}$ that is exponentially decaying. Intuitively, we use the chain rule to decompose the partial derivative $\left.\frac{\partial g_{t\mid \tau}}{\partial x_\tau}\right|_{x_\tau, \theta_{\tau:t-1}}$ as the product of multiple partial derivatives $\left.\frac{\partial g_{t_{i+1}\mid t_i}}{\partial x_{t_i}}\right|_{x_{t_i}, \theta_{t_i:t_{i+1}-1}}$, where each time interval $[t_i:t_{i+1}-1]$ has (approximately) length $h$. By the time-invariant contractive property, we know the norm of $\frac{\partial g_{t_{i+1}\mid t_i}}{\partial x_{t_i}}$ can be upper bounded by $C \rho^h$ once it is realized on the trajectory $\{x_{t_i}, (\theta_{t_i})_{\times h}\}$, where the policy parameter is repeating. By the smoothness guarantee derived in \Cref{lemma:concatenation-of-smooth-function}, we can show the difference between $\left.\frac{\partial g_{t_{i+1}\mid t_i}}{\partial x_{t_i}}\right|_{x_{t_i}, \theta_{t_i:t_{i+1}-1}}$ and $\left.\frac{\partial g_{t_{i+1}\mid t_i}}{\partial x_{t_i}}\right|_{x_{t_i}, (\theta_{t_i})_{\times h}}$ is in the order of $O(\varepsilon)$. This implies that the norm of $\left.\frac{\partial g_{t_{i+1}\mid t_i}}{\partial x_{t_i}}\right|_{x_{t_i}, \theta_{t_i:t_{i+1}-1}}$ can be bounded by $\hat{\rho}^h$ once $\varepsilon$ is sufficiently small with respect to the gap $\left(\hat{\rho}^h - C \rho^h\right)$.

We present the formal proof of \Cref{lemma:from-static-to-time-varying:detailed} below.

\begin{proof}[Proof of \Cref{lemma:from-static-to-time-varying:detailed}]
We first show that, under a time-invariant policy parameter $\theta$, starting from an arbitrary $x_\tau$ that satisfies $\norm{x_\tau} \leq R_C$, we have $\norm{x_t} \leq R_S + C \rho^{t-\tau} \norm{x_\tau}$ for all $t \geq \tau$.

To see this, note that by time-invariant contractive perturbation, we have
\begin{align*}
    \norm{g_{t\mid \tau}(x_\tau, \theta_{\times (t - \tau)})} \leq{}& \norm{g_{t\mid \tau}(x_\tau, \theta_{\times (t - \tau)}) - g_{t\mid \tau}(0, \theta_{\times (t - \tau)})} + \norm{g_{t\mid \tau}(0, \theta_{\times (t - \tau)})}\\
    \leq{}& C \rho^{t - \tau}\norm{x_\tau} + R_S.
\end{align*}

Now, we show that if starting from $x_\tau$ that satisfies $\norm{x_\tau} \leq \hat{R}_C$, the trajectory induced by an $\hat{\varepsilon}$-time-varying parameter sequence satisfies that $\norm{x_t} \leq \hat{R}_S + C \rho^{t - \tau} \norm{x_\tau}$ for all $t \geq \tau$. Since this ball is contained in $B(0, R_C)$, we know the time-invariant contractive perturbation always apply.

We show this statement by induction. Suppose the statement holds for all time steps $\tau, \ldots, t - 1$. We see that
\begin{align*}
    &\norm{g_{t\mid \tau}(x_\tau, \theta_{\tau:t-1}) - g_{t\mid \tau}(x_\tau, \theta_{\times (t - \tau)})}\\
    \leq{}& \sum_{j = \tau}^{t}\norm{g_{t\mid \tau}(x_\tau, \theta_{\tau:j}, (\theta_t)_{\times (t - j-1)}) - g_{t\mid \tau}(x_\tau, \theta_{\tau:j-1}, (\theta_t)_{\times (t - j)})}\\
    \leq{}& C L_{g, u} L_{\pi, \theta} \sum_{j = \tau}^t \rho^{t - j - 1} \norm{\theta_t - \theta_j} \leq \frac{C L_{g, u} L_{\pi, \theta} \hat{\varepsilon}}{(1 - \rho)^2}.
\end{align*}
Therefore, we see that
\begin{align*}
    \norm{g_{t\mid \tau}(x_\tau, \theta_{\tau:t-1})} \leq{}& \norm{g_{t\mid \tau}(x_\tau, \theta_{\times (t - \tau)})} + \frac{C L_{g, u} L_{\pi, \theta} \hat{\varepsilon}}{(1 - \rho)^2}\\
    \leq{}& C \rho^{t-\tau} \norm{x_\tau} + R_S + \frac{C L_{g, u} L_{\pi, \theta} \hat{\varepsilon}}{(1 - \rho)^2}\\
    ={}& C \rho^{t - \tau} \norm{x_\tau} + \hat{R}_S.
\end{align*}
This finishes the proof of $\varepsilon$-time-varying stability with $\hat{\varepsilon}$ and $\hat{R}_S$.

Note that we can decompose the partial derivative $\left.\frac{\partial g_{t\mid \tau}}{\partial x_\tau}\right|_{x_\tau, \theta_{\tau:t-1}}$ as
\begin{align}\label{lemma:from-static-to-time-varying:e1}
    &\left.\frac{\partial g_{t\mid \tau}}{\partial x_\tau}\right|_{x_\tau, \theta_{\tau:t-1}}\nonumber\\
    ={}& \left.\frac{\partial g_{t_p\mid t_{p-1}}}{\partial x_{t_{p-1}}}\right|_{x_{t_{p-1}}, \theta_{t_{p-1}:t_p-1}}\cdot \left.\frac{\partial g_{t_{p-1}\mid t_{p-2}}}{\partial x_{t_{p-2}}}\right|_{x_{t_{p-2}}, \theta_{t_{p-2}:t_{p-1}-1}} \cdots \left.\frac{\partial g_{t_1\mid t_0}}{\partial x_{t_0}}\right|_{x_{t_0}, \theta_{t_0:t_1-1}},
\end{align}
where $t_0 = \tau, t_p = t$; $t_i = t_{i-1} + h$ holds for $i = 1, \ldots, p-1$, and $t_{p-1} < t_p \leq t_{p-1} + h$.

For $i \in [0, p-2]$, we see that
\begin{subequations}\label{lemma:from-static-to-time-varying:e2}
\begin{align}
    &\norm{\left.\frac{\partial g_{t_{i+1}\mid t_{i}}}{\partial x_{t_i}}\right|_{x_{t_i}, \theta_{t_i:t_{i+1}-1}} - \left.\frac{\partial g_{t_{i+1}\mid t_i}}{\partial x_{t_i}}\right|_{x_{t_i}, (\theta_{t_i})\times h}}\nonumber\\
    \leq{}& \sum_{j = 0}^{h-1} \norm{\left.\frac{\partial g_{t_{i+1}\mid t_i}}{\partial x_{t_i}}\right|_{x_{t_i}, (\theta_{t_i})_{\times j}, \theta_{t_i+j; t_{i+1}-1}} - \left.\frac{\partial g_{t_{i+1}\mid t_i}}{\partial x_{t_i}}\right|_{x_{t_i}, (\theta_{t_i})_{\times (j+1)}, \theta_{t_i+j+1; t_{i+1}-1}}}\label{lemma:from-static-to-time-varying:e2:s1}\\
    \leq{}& \sum_{j = 0}^{h-1} \norm{\left.\frac{\partial g_{t_i+j}}{\partial x_{t_i}}\right|_{x_{t_i}, (\theta_{t_i})_{\times j}}}\cdot\nonumber\\
    &\norm{\left.\frac{\partial g_{t_{i+1}\mid (t_i + j)}}{\partial x_{t_i + j}}\right|_{\bar{x}_{t_i + j}, \theta_{t_i+j; t_{i+1}-1}} - \left.\frac{\partial g_{t_{i+1}\mid (t_i + j)}}{\partial x_{t_i + j}}\right|_{\bar{x}_{t_i + j}, \theta_{t_i}, \theta_{t_i+j+1; t_{i+1}-1}}}\label{lemma:from-static-to-time-varying:e2:s2}\\
    \leq{}& C \sum_{j = 0}^{h-1} \rho^j \norm{\left.\frac{\partial g_{t_{i+1}\mid (t_i + j)}}{\partial x_{t_i + j}}\right|_{\bar{x}_{t_i + j}, \theta_{t_i+j; t_{i+1}-1}} - \left.\frac{\partial g_{t_{i+1}\mid (t_i + j)}}{\partial x_{t_i + j}}\right|_{\bar{x}_{t_i + j}, \theta_{t_i}, \theta_{t_i+j+1; t_{i+1}-1}}},\label{lemma:from-static-to-time-varying:e2:s3}
\end{align}
\end{subequations}
where we use the shorthand notation $\bar{x}_{t_i + j} = g_{t_i + j\mid t_i}(x_{t_i}, (\theta_{t_i})_{\times j})$. We use the triangle inequality in \eqref{lemma:from-static-to-time-varying:e2:s1}, the chain rule decomposition in \eqref{lemma:from-static-to-time-varying:e2:s2}. In \eqref{lemma:from-static-to-time-varying:e2:s3}, we can apply the time-invariant contractive perturbation property because $\norm{x_{\tau}} \leq \hat{R}_C$, which implies that $\norm{x_{t_i}} \leq R_C$. Note that
\begin{align*}
    &\norm{\left.\frac{\partial g_{t_{i+1}\mid (t_i + j)}}{\partial x_{t_i + j}}\right|_{\bar{x}_{t_i + j}, \theta_{t_i+j; t_{i+1}-1}} - \left.\frac{\partial g_{t_{i+1}\mid (t_i + j)}}{\partial x_{t_i + j}}\right|_{\bar{x}_{t_i + j}, \theta_{t_i}, \theta_{t_i+j+1; t_{i+1}-1}}}\\
    \leq{}& \norm{\left.\frac{\partial g_{t_{i+1}\mid (t_i + j + 1)}}{\partial x_{t_i + j + 1}}\right|_{\bar{x}_{t_i + j + 1}, \theta_{t_i+j+1; t_{i+1}-1}}} \cdot\\
    &\norm{\left.\frac{\partial g_{(t_i + j + 1)\mid (t_i + j)}}{\partial x_{t_i + j}}\right|_{\bar{x}_{t_i + j}, \theta_{t_i+j}} - \left.\frac{\partial g_{(t_i + j + 1)\mid (t_i + j)}}{\partial x_{t_i + j}}\right|_{\bar{x}_{t_i + j}, \theta_{t_i}}}\\
    &+ \norm{\left.\frac{\partial g_{t_{i+1}\mid (t_i + j + 1)}}{\partial x_{t_i + j + 1}}\right|_{\bar{x}_{t_i + j + 1}, \theta_{t_i+j+1; t_{i+1}-1}} - \left.\frac{\partial g_{t_{i+1}\mid (t_i + j + 1)}}{\partial x_{t_i + j + 1}}\right|_{\bar{x}_{t_i + j + 1}', \theta_{t_i+j+1; t_{i+1}-1}}} \cdot\\
    &\norm{\left.\frac{\partial g_{(t_i + j + 1)\mid (t_i + j)}}{\partial x_{t_i + j}}\right|_{\bar{x}_{t_i + j}, \theta_{t_i + j}}}\\
    \leq{}& \left(L_{g, x} + L_{g, u} L_{\ALG, x}\right)^{h - j - 1} \cdot \left((1 + L_{\ALG, x}) \ell_{g, u} L_{\ALG, \theta} + L_{g, u} \ell_{\ALG, \theta}\right) \norm{\theta_{t_i + j} - \theta_{t_i}}\\
    &+ \left((1 + L_{\ALG, x}) \left(\ell_{g, x} + \ell_{g, u} \cdot L_{\ALG, x}\right) + L_{g, u} \cdot \ell_{\ALG, x}\right)\cdot \left(1 + L_{g, x} + L_{g, u} L_{\ALG, x}\right)^{2(h - j)} \cdot\\
    &L_{g, u}L_{\ALG, \theta} \norm{\theta_{t_i + j} - \theta_{t_i}}\\
    \leq{}& C' \left(1 + L_{g, x} + L_{g, u} L_{\ALG, x}\right)^{2h} \cdot j \varepsilon,
\end{align*}
where we adopt the shorthand $\bar{x}_{t_i + j + 1}' \coloneqq g_{t_i+j+1\mid t_i + j}(\bar{x}_{t_i + j}, \theta_{t_i + j})$, and define
\begin{align*}
    C' = \left((1 + L_{\ALG, x}) \left(\ell_{g, x} + \ell_{g, u} \cdot (L_{\ALG, x} + L_{\ALG, \theta})\right) + L_{g, u} \cdot (\ell_{\ALG, x} + \ell_{\ALG, \theta})\right)(L_{g, u}L_{\ALG, \theta} + 1).
\end{align*}
Here, we use the chain rule and the triangle inequality in the first inequality. We can apply \Cref{lemma:concatenation-of-smooth-function} and \[
    \norm{\bar{x}_{t_i + j + 1}' - \bar{x}_{t_i + j + 1}} \leq L_{g, u}L_{\ALG, \theta}\norm{\theta_{t_i+j} - \theta_{t_i}}
\] in the second inequality because the trajectory induced by $(\theta_{t_i})_{\times j}, \theta_{t_i+j:t_{i+1}-1}$ always stay within the ball $B(0, R_x)$ where the Lipschitzness/smoothness of dynamics/policies hold.
Substituting this into \eqref{lemma:from-static-to-time-varying:e2} gives
\begin{align}\label{lemma:from-static-to-time-varying:e3}
    \norm{\left.\frac{\partial g_{t_{i+1}\mid t_{i}}}{\partial x_{t_i}}\right|_{x_{t_i}, \theta_{t_i:t_{i+1}-1}} - \left.\frac{\partial g_{t_{i+1}\mid t_i}}{\partial x_{t_i}}\right|_{x_{t_i}, (\theta_{t_i})\times h}} \leq{}& C C' \left(1 + L_{g, x} + L_{g, u} L_{\ALG, x}\right)^{2h} \varepsilon \sum_{j = 0}^{h-1} \rho^j \cdot j\nonumber\\
    \leq{}& \frac{C C' \rho \left(1 + L_{g, x} + L_{g, u} L_{\ALG, x}\right)^{2h}}{(1 - \rho)^2}\cdot \varepsilon\nonumber\\
    \leq{}& \hat{\rho}^h - C \rho^h.
\end{align}
Therefore, by \eqref{lemma:from-static-to-time-varying:e1}, we see that
\begin{subequations}\label{lemma:from-static-to-time-varying:e4}
\begin{align}
    &\norm{\left.\frac{\partial g_{t\mid \tau}}{\partial x_\tau}\right|_{x_\tau, \theta_{\tau:t-1}}}\nonumber\\
    \leq{}& \norm{\left.\frac{\partial g_{t_p\mid t_{p-1}}}{\partial x_{t_{p-1}}}\right|_{x_{t_{p-1}}, \theta_{t_{p-1}:t_p-1}}}\cdot \norm{\left.\frac{\partial g_{t_{p-1}\mid t_{p-2}}}{\partial x_{t_{p-2}}}\right|_{x_{t_{p-2}}, \theta_{t_{p-2}:t_{p-1}-1}}}\cdots \norm{\left.\frac{\partial g_{t_1\mid t_0}}{\partial x_{t_0}}\right|_{x_{t_0}, \theta_{t_0:t_1-1}}}\nonumber\\
    \leq{}& \left((1 + L_{\ALG, x}) \left(\ell_{g, x} + \ell_{g, u} \cdot L_{\ALG, x}\right) + L_{g, u} \cdot \ell_{\ALG, x}\right) \cdot \left(1 + L_{g, x} + L_{g, u} L_{\ALG, x}\right)^{2h} \cdot (\hat{\rho}^h)^{p-1}\label{lemma:from-static-to-time-varying:e4:s1}\\
    \leq{}& \hat{C} (\hat{\rho})^{t - \tau}, \label{lemma:from-static-to-time-varying:e4:s2}
\end{align}
\end{subequations}
where we use \eqref{lemma:from-static-to-time-varying:e3} in \eqref{lemma:from-static-to-time-varying:e4:s1}; we use the definition of $\hat{C}$ in \eqref{lemma:from-static-to-time-varying:e4:s2}.
This finishes the proof of $\varepsilon$-time-varying contractive perturbation with $\hat{\epsilon}, \hat{R}_C, \hat{C}, \hat{\rho}$.
\end{proof}


\section{Regret of GAPS under convex surrogate costs}\label{appendix:gaps-outline}

In this section, we provide a proof outline of the adaptive regret bound for GAPS. As we discussed in \Cref{sec:continuous}, the intuition behind GAPS is to mimic the ideal OGD update $\theta_{t+1} = \prod_{\Theta}(\theta_t - \eta_t \nabla F_t(\theta_t))$ with limited memory size and computational complexity. 
While the existing literature of OCO guarantees that the ideal OGD update with constant step size $\eta$ of the order $1/\sqrt{T}$ achieves a policy regret of $O\left(\sqrt{T}\right)$, GAPS incurs an approximation error at every time step since it uses $G_t$ (Algorithm \ref{alg:OCO-with-parameter-update}, Line \ref{alg:OCO-with-parameter-update:grad-approx}) instead of $\nabla F_t(\theta_t)$ to implement gradient descent. We characterize how a per-step bias in the gradient estimation may affect the regret guarantee of the OGD in \Cref{thm:OGD-with-biased-gradient-regret}. The proof is deferred to \Cref{appendix:OGD-with-bias}.

\begin{theorem}\label{thm:OGD-with-biased-gradient-regret}
Consider the update rule $\theta_{t+1} = \prod_{\Theta}(\theta_t - \eta G_t)$. Suppose $\Theta$ is a convex compact set with diameter~$D$. If $F_t$ is convex and $\norm{\nabla F_t(\theta)} \leq W$ for all $\theta \in \Theta$, and $\norm{\nabla F_t(\theta_t) - G_t} \leq \alpha$ holds for all time steps $t$, then, for arbitrary $I = [r:s] \subseteq \mathcal{T}$,
{
\begin{align*}
\sum_{t=r}^s F_t(\theta_t) - \min_{\theta_I \in \Theta} \sum_{t=r}^s F_t(\theta_I) \leq \alpha D T + (W^2 + \alpha^2)\eta T + \frac{D^2}{2 \eta}.
\end{align*}
}

\end{theorem}

With \Cref{thm:OGD-with-biased-gradient-regret}, obtaining the policy regret bounds for GAPS reduces to showing both $\abs{f_t(x_t, u_t) - F_t(\theta_t)}$ and $\norm{\nabla F_t(\theta_t) - G_t}$ are in the order of $O(1/\sqrt{T})$. Here, we only consider the order of magnitude with respect to the horizon $T$ for clarity.
As we will show in \Cref{thm:OCO-with-parameter-update-trajectory-distance} and \Cref{thm:OCO-with-parameter-update-gradient-bias}, both of these quantities are in the order of $O(\eta)$ when GAPS adopts the learning rate $\eta$.

To obtain these results, we first show a lemma about the stability of the trajectory achieved by an $\varepsilon$-time-varying policy parameter sequence.

\begin{lemma}\label{lemma:epsilon-time-varying-stability}
Suppose Assumptions \ref{assump:Lipschitz-and-smoothness} and \ref{assump:contractive-and-stability} hold. For any starting state $x_\tau \in B_n(0, R_S + C\norm{x_0})$ and $\theta_{\tau:t-1} \in S_\varepsilon(\tau:t-1)$, the final state $x_t \coloneqq g_{t\mid \tau}(x_\tau, \theta_{\tau:t-1})$ satisfies $\norm{x_t} \leq C \rho^{t - \tau} \norm{x_\tau} + R_S$.
\end{lemma}
\begin{proof}[Proof of \Cref{lemma:epsilon-time-varying-stability}]
By $\varepsilon$-time-varying contractive perturbation, we see that
\begin{align*}
    \norm{x_t - g_{t\mid \tau}(0, \theta_{\tau:t-1})} \leq C \rho^{t-\tau}\norm{x_\tau}.
\end{align*}
Thus, by the triangle inequality, we see that
\[\norm{x_t} \leq \norm{x_t - g_{t\mid \tau}(0, \theta_{\tau:t-1})} + \norm{g_{t\mid \tau}(0, \theta_{\tau:t-1})} \leq C \rho^{t-\tau}\norm{x_\tau} + R_S,\]
where we use $\varepsilon$-time-varying stability in the last inequality.
\end{proof}

Next, we show a lemma about the contractive property of the partial derivatives of the multi-step dynamics.

\begin{lemma}[Lipschitzness/Smoothness of the Multi-Step Dynamics]\label{lemma:smooth-multi-step-dynamics}
Suppose Assumptions \ref{assump:Lipschitz-and-smoothness} and \ref{assump:contractive-and-stability} hold. Given two time steps $t > \tau$, for any $x_\tau, x_\tau' \in B_n(0, R_S + C \norm{x_0})$ and $\theta_\tau, \theta_\tau' \in \Theta,$ $\theta_{\tau+1:t-1} \in S_{\varepsilon}(\tau+1:t-1)$, if $x_{\tau+1}' \coloneqq g_{\tau + 1\mid \tau}(x_\tau', \theta_\tau')$ is also in $B_n(0, R_S + C \norm{x_0})$, the multi-step dynamical function $g_{t\mid \tau}$ satisfies that
\begin{align*}
    &\norm{\left.\frac{\partial g_{t\mid \tau}}{\partial x_\tau} \right|_{x_\tau, \theta_{\tau:t-1}}} \leq C_{L, g, x} \decayfactor^{t - \tau}, \quad \norm{\left.\frac{\partial g_{t\mid \tau}}{\partial \theta_\tau} \right|_{x_\tau, \theta_{\tau:t-1}}} \leq C_{L, g, \theta} \decayfactor^{t-\tau}, \forall \theta_{\tau:t-1} \in S_{\varepsilon}(\tau:t-1),\\
    &\norm{\left.\frac{\partial g_{t\mid \tau}}{\partial x_\tau} \right|_{x_\tau, \theta_{\tau:t-1}} - \left.\frac{\partial g_{t\mid \tau}}{\partial x_\tau} \right|_{x_\tau', \theta_\tau', \theta_{\tau+1:t-1}}} \leq C_{\ell, g, (x, x)} \decayfactor^{t - \tau} \norm{x_\tau - x_\tau'} + C_{\ell, g, (x, \theta)} \decayfactor^{t - \tau} \norm{\theta_\tau - \theta_\tau'},\\
    &\norm{\left.\frac{\partial g_{t\mid \tau}}{\partial \theta_\tau} \right|_{x_\tau, \theta_{\tau:t-1}} - \left.\frac{\partial g_{t\mid \tau}}{\partial \theta_\tau} \right|_{x_\tau', \theta_\tau', \theta_{\tau+1:t-1}}} \leq C_{\ell, g, (\theta, x)} \decayfactor^{t - \tau} \norm{x_\tau - x_\tau'} + C_{\ell, g, (\theta, \theta)} \decayfactor^{t - \tau} \norm{\theta_\tau - \theta_\tau'},
\end{align*}
where $C_{L, g, x} = C, C_{L, g, \theta} = \frac{C L_{g, u} L_{\ALG, \theta}}{\decayfactor}$, and
\begin{align*}
    C_{\ell, g, (x, x)} ={}& {\left((1 + L_{\ALG, x}) \left(\ell_{g, x} + \ell_{g, u} L_{\ALG, x}\right) + L_{g, x} \ell_{\ALG, x}\right) C^3}{\decayfactor^{-1}(1 - \decayfactor)^{-1}},\\
    C_{\ell, g, (x, \theta)} ={}& {\left((1 + L_{\ALG, x}) \left(\ell_{g, x} + \ell_{g, u} L_{\ALG, x}\right) + L_{g, x} \ell_{\ALG, x}\right) C^3 L_{g, u} L_{\ALG, \theta}}{\decayfactor^{-1}(1 - \decayfactor)^{-1}}\\
    &+ {\left((1 + L_{\ALG, x}) \ell_{g, u} L_{\ALG, \theta} + L_{g, u} \ell_{\ALG, \theta}\right) C}{\decayfactor^{-1}(1 - \decayfactor)^{-1}},\\
    C_{\ell, g, (\theta, x)} ={}& {\left((1 + L_{\ALG, x}) \left(\ell_{g, x} + \ell_{g, u} L_{\ALG, x}\right) + L_{g, x} \ell_{\ALG, x}\right) (L_{g, x} + L_{g, u} L_{\ALG, x}) C^3 L_{g, u} L_{\ALG, \theta}}{\decayfactor^{-2}(1 - \decayfactor)^{-1}}\\
    &+ {C \left(L_{\ALG, \theta}(\ell_{g, x} + \ell_{g, u} L_{\ALG, x}) + L_{g, u} \ell_{\ALG, x}\right)}{\decayfactor^{-1}},\\
    C_{\ell, g, (\theta, \theta)} ={}& {\left((1 + L_{\ALG, x}) \left(\ell_{g, x} + \ell_{g, u} \cdot L_{\ALG, x}\right) + L_{g, x} \cdot \ell_{\ALG, x}\right) L_{g, u}^2 L_{\ALG, \theta}^2 C^3}{\decayfactor^{-2}(1 - \decayfactor)^{-1}}\\
    &+ {\left(L_{g, u} \ell_{\ALG, \theta} + \ell_{g, u} L_{\ALG, \theta}^2\right)C}{\decayfactor^{-1}}.
\end{align*}
\end{lemma}

Intuitively, \Cref{lemma:smooth-multi-step-dynamics} shows that the dependence of the state $x_t$ on the previous state $x_\tau$ and $\theta_\tau$ decays exponentially with respect to their time distance $t - \tau$. Specifically, recall that the multi-step dynamics $g_{t\mid \tau}$ writes $x_t$ as a function of $x_\tau$ and $\theta_{\tau:t-1}$. When other variables are fixed, the Lipschitzness and smoothness constants with respect to $x_\tau$ and $\theta_\tau$ are both $O(\decayfactor^{t - \tau})$. While the contractive Lipschitzness on $x_\tau$ is automatically guaranteed by $\varepsilon$-time-varying contractive perturbation (\Cref{def:epsilon-exp-decay-perturbation-property}), we use this property and the chain rule decomposition to show the Lipschitzness on $\theta_\tau$ and the smoothness.

The first inequality in \Cref{lemma:smooth-multi-step-dynamics} directly follows from $\varepsilon$-time-varying contractive perturbation. To reflect the main technical difficulty, we show the third inequality here with the assumption that the first two inequalities hold. The proof of other inequalities are deferred to \Cref{appendix:smooth-multi-step-dynamics}.

\begin{proof}[Proof of the 3rd inequality in \Cref{lemma:smooth-multi-step-dynamics}]
Note that we have the chain rule decomposition
\begin{align}\label{thm:smooth-dynamics-main:e0}
    \left.\frac{\partial g_{t\mid \tau}}{\partial x_\tau} \right|_{x_\tau, \theta_\tau, \theta_{\tau+1:t-1}} &= \left.\frac{\partial g_{t\mid t-1}}{\partial x_{t-1}}\right|_{x_{t-1}, \theta_{t-1}} \cdot \frac{\partial g_{t-1\mid t-2}}{\partial x_{t-2}}\mid_{x_{t-2}, \theta_{t-2}} \cdots \left.\frac{\partial g_{\tau+1\mid \tau}}{\partial x_{\tau}}\right|_{x_{\tau}, \theta_\tau},\nonumber\\
    \left.\frac{\partial g_{t\mid \tau}}{\partial x_\tau} \right|_{x_\tau', \theta_\tau', \theta_{\tau+1:t-1}} &= \left.\frac{\partial g_{t\mid t-1}}{\partial x_{t-1}}\right|_{x_{t-1}', \theta_{t-1}} \cdot \left.\frac{\partial g_{t-1\mid t-2}}{\partial x_{t-2}}\right|_{x_{t-2}', \theta_{t-2}} \cdots \left.\frac{\partial g_{\tau+1\mid \tau}}{\partial x_{\tau}}\right|_{x_{\tau}', \theta_\tau'},
\end{align}
where we use the notation $x_{\tau'} = g_{\tau'\mid \tau}(x_\tau, \theta_{\tau:\tau'-1})$ and $x_{\tau'}' = g_{\tau'\mid \tau}(x_\tau', \theta_\tau', \theta_{\tau+1:\tau'-1})$ for $\tau' \in [\tau + 1: t-1]$.

Note that for any $i \in [1: t - \tau]$ and any $\theta_{t-i}' \in \Theta$, we have the decomposition
\begin{align*}
    &\left.\frac{\partial g_{t-i+1\mid t-i}}{\partial x_{t-i}}\right|_{x_{t-i}, \theta_{t-i}} - \left.\frac{\partial g_{t-i+1\mid t-i}}{\partial x_{t-i}}\right|_{x_{t-i}', \theta_{t-i}'}\\
    ={}& \left.\frac{\partial g_{t-i}}{\partial x_{t-i}}\right|_{x_{t-i}, u_{t-i}} - \left.\frac{\partial g_{t-i}}{\partial x_{t-i}}\right|_{x_{t-i}', u_{t-i}'} + \left.\frac{\partial g_{t-i}}{\partial u_{t-i}}\right|_{x_{t-i}, u_{t-i}}\left.\frac{\partial \ALG_{t-i}}{\partial x_{t-i}}\right|_{x_{t-i}, \theta_{t-i}}\\
    &- \left.\frac{\partial g_{t-i}}{\partial x_{t-i}}\right|_{x_{t-i}', u_{t-i}'}\left.\frac{\partial \ALG_{t-i}}{\partial x_{t-i}}\right|_{x_{t-i}', \theta_{t-i}'}\\
    ={}& \left.\frac{\partial g_{t-i}}{\partial x_{t-i}}\right|_{x_{t-i}, u_{t-i}} - \left.\frac{\partial g_{t-i}}{\partial x_{t-i}}\right|_{x_{t-i}', u_{t-i}'} + \left(\left.\frac{\partial g_{t-i}}{\partial u_{t-i}}\right|_{x_{t-i}, u_{t-i}} - \left.\frac{\partial g_{t-i}}{\partial u_{t-i}}\right|_{x_{t-i}', u_{t-i}'}\right)\left.\frac{\partial \ALG_{t-i}}{\partial x_{t-i}}\right|_{x_{t-i}, \theta_{t-i}}\\
    &+ \left.\frac{\partial g_{t-i}}{\partial u_{t-i}}\right|_{x_{t-i}', u_{t-i}'} \left(\left.\frac{\partial \ALG_{t-i}}{\partial x_{t-i}}\right|_{x_{t-i}, \theta_{t-i}} - \left.\frac{\partial \ALG_{t-i}}{\partial x_{t-i}}\right|_{x_{t-i}', \theta_{t-i}'}\right),
\end{align*}
where we use the notation $u_{t-i} = \ALG_{t-i}(x_{t-i}, \theta_{t-i}), u_{t-i}' = \ALG_{t-i}(x_{t-i}', \theta_{t-i}')$. Taking norms on both sides of the equation and applying the triangle inequality gives
\begin{subequations}\label{thm:smooth-dynamics-main:e1}
\begin{align}
    &\norm{\left.\frac{\partial g_{t-i+1\mid t-i}}{\partial x_{t-i}}\right|_{x_{t-i}, \theta_{t-i}} - \left.\frac{\partial g_{t-i+1\mid t-i}}{\partial x_{t-i}}\right|_{x_{t-i}', \theta_{t-i}'}}\nonumber\\
    \leq{}& \ell_{g, x} \norm{x_{t-i} - x_{t-i}'} + \ell_{g, u} \norm{\ALG_{t-i}(x_{t-i}, \theta_{t-i}) - \ALG_{t-i}(x_{t-i}', \theta_{t-i}')}\nonumber\\
    &+ L_{\ALG, x} \left(\ell_{g, x} \norm{x_{t-i} - x_{t-i}'} + \ell_{g, u} \norm{\ALG_{t-i}(x_{t-i}, \theta_{t-i}) - \ALG_{t-i}(x_{t-i}', \theta_{t-i}')}\right)\nonumber\\ 
    &+ L_{g, u} \cdot \left(\ell_{\ALG, x} \norm{x_{t-i} - x_{t-i}'} + \ell_{\ALG, \theta} \norm{\theta_{t-i} - \theta_{t-i}'}\right)\label{thm:smooth-dynamics-main:e1:s1}\\
    \leq{}& \left((1 + L_{\ALG, x}) \left(\ell_{g, x} + \ell_{g, u} \cdot L_{\ALG, x}\right) + L_{g, u} \cdot \ell_{\ALG, x}\right) \norm{x_{t-i} - x_{t-i}'}\nonumber\\
    &+ \left((1 + L_{\ALG, x}) \ell_{g, u} L_{\ALG, \theta} + L_{g, u} \ell_{\ALG, \theta}\right)\norm{\theta_{t-i} - \theta_{t-i}'},\label{thm:smooth-dynamics-main:e1:s2}
\end{align}
\end{subequations}
where we use \Cref{assump:Lipschitz-and-smoothness} and the definition of $u_{t-i}, u_{t-i}'$ in \eqref{thm:smooth-dynamics-main:e1:s1}; and \Cref{assump:Lipschitz-and-smoothness} in \eqref{thm:smooth-dynamics-main:e1:s2}. Therefore, by \eqref{thm:smooth-dynamics-main:e0} and \eqref{thm:smooth-dynamics-main:e1}, we see that
\begin{subequations}\label{thm:smooth-dynamics-main:e2}
\begin{align}
    &\norm{\left.\frac{\partial g_{t\mid \tau}}{\partial x_\tau} \right|_{x_\tau, \theta_{\tau:t-1}} - \left.\frac{\partial g_{t\mid \tau}}{\partial x_\tau} \right|_{x_\tau', \theta_{\tau:t-1}}}\nonumber\\
    \leq{}& \sum_{i = 1}^{t - \tau - 1} \Bigg(\norm{\prod_{\tau' = 1}^{i - 1} \left.\frac{\partial g_{t-\tau'+1\mid t-\tau'}}{\partial x_{t-\tau'}}\right|_{x_{t-\tau'}', \theta_{t - \tau'}}}\cdot \norm{\left.\frac{\partial g_{t-i+1\mid t-i}}{\partial x_{t-i}}\right|_{x_{t-i}, \theta_{t-i}} - \left.\frac{\partial g_{t-i+1\mid t-i}}{\partial x_{t-i}}\right|_{x_{t-i}', \theta_{t-i}}}\cdot \nonumber\\
    &\norm{\prod_{\tau' = i+1}^{t - \tau} \left.\frac{\partial g_{t-\tau'+1\mid t-\tau'}}{\partial x_{t-\tau'}}\right|_{x_{t-\tau'}, \theta_{t - \tau'}}} \Bigg)\nonumber\\
    &+ \norm{\prod_{\tau' = 1}^{t - \tau - 1} \left.\frac{\partial g_{t-\tau'+1\mid t-\tau'}}{\partial x_{t-\tau'}}\right|_{x_{t-\tau'}', \theta_{t - \tau'}}}\cdot \norm{\left.\frac{\partial g_{\tau+1\mid \tau}}{\partial x_{\tau}}\right|_{x_{\tau}, \theta_{\tau}} - \left.\frac{\partial g_{\tau+1\mid \tau}}{\partial x_{\tau}}\right|_{x_{\tau}', \theta_{\tau}'}}\nonumber\\
    \leq{}& \sum_{i = 1}^{t - \tau} (C \decayfactor^{i-1}) \cdot \left((1 + L_{\ALG, x}) \left(\ell_{g, x} + \ell_{g, u} \cdot L_{\ALG, x}\right) + L_{g, x} \cdot \ell_{\ALG, x}\right) \norm{x_{t-i} - x_{t-i}'} \cdot (C \decayfactor^{t-\tau-i})\nonumber\\
    &+ C \decayfactor^{t - \tau - 1} \cdot \left((1 + L_{\ALG, x}) \ell_{g, u} L_{\ALG, \theta} + L_{g, u} \ell_{\ALG, \theta}\right)\norm{\theta_{\tau} - \theta_{\tau}'}\label{thm:smooth-dynamics-main:e2:s1}\\
    ={}& \left((1 + L_{\ALG, x}) \left(\ell_{g, x} + \ell_{g, u} \cdot L_{\ALG, x}\right) + L_{g, x} \cdot \ell_{\ALG, x}\right) C^2 \cdot \decayfactor^{t - \tau - 1} \sum_{i = 1}^{t - \tau} \norm{x_{t-i} - x_{t-i}'}\nonumber\\
    &+ C \decayfactor^{t - \tau - 1} \cdot \left((1 + L_{\ALG, x}) \ell_{g, u} L_{\ALG, \theta} + L_{g, u} \ell_{\ALG, \theta}\right)\norm{\theta_{\tau} - \theta_{\tau}'}\nonumber\\
    \leq{}& C_{\ell, g, (x, x)} \cdot \decayfactor^{t - \tau} \norm{x_\tau - x_\tau'} + C_{\ell, g, (x, \theta)} \cdot \decayfactor^{t - \tau} \norm{\theta_\tau - \theta_\tau'}, \label{thm:smooth-dynamics-main:e2:s2}
\end{align}
\end{subequations}
where we use the $\varepsilon$-time-varying contractive perturbation property and \eqref{thm:smooth-dynamics-main:e1} in \eqref{thm:smooth-dynamics-main:e2:s1}; we use the first two inequalities to bound $\norm{x_{t-i} - x_{t-i}'} \leq C \decayfactor^{t-i-\tau}\norm{x_\tau - x_\tau'} + \frac{C L_{g, u} L_{\ALG, \theta}}{\decayfactor} \cdot \decayfactor^{t-i-\tau}\norm{\theta_\tau - \theta_\tau'}$ in \eqref{thm:smooth-dynamics-main:e2:s2}.
\end{proof}

Since we will need more general forms of policy sequence later to bound $\norm{G_t - \nabla F_t}$ than the sequence with small step sizes, we state the contractive Lipschitzness/smoothness of the multi-step cost function $f_{t\mid \tau}$. This is an implication of \Cref{lemma:smooth-multi-step-dynamics} because for any $\tau < t$, the previous state $x_\tau$ and previous policy parameter $\theta_\tau$ can only affect the current stage cost $c_t$ by affecting the current state $x_t$. We formalize this result in \Cref{coro:smooth-multi-step-costs} and defer the detailed proof to \Cref{appendix:smooth-multi-step-costs}.

\begin{corollary}[Lipschitzness/Smoothness of the Multi-Step Costs]\label{coro:smooth-multi-step-costs}
Under the same assumptions as \Cref{lemma:smooth-multi-step-dynamics}, let
\[x_t \coloneqq g_{t\mid \tau}(x_\tau, \theta_{\tau:t-1}), u_t \coloneqq \ALG_t(x_t, \theta_t); \text{ and }x_t' \coloneqq g_{t\mid \tau}(x_\tau', \theta_\tau', \theta_{\tau+1:t-1}), u_t' \coloneqq \ALG_t(x_t', \theta_t).\]
Then, the multi-step cost function $f_{t\mid \tau}$ satisfies that
\begin{align*}
    &\norm{\left.\frac{\partial f_{t\mid \tau}}{\partial x_\tau} \right|_{x_\tau, \theta_{\tau:t}}} \leq C_{L, f, x} \decayfactor^{t - \tau}, \norm{\left.\frac{\partial f_{t\mid \tau}}{\partial \theta_\tau} \right|_{x_\tau, \theta_{\tau:t}}} \leq C_{L, f, \theta} \decayfactor^{t-\tau},\\
    &\norm{\left.\frac{\partial f_{t\mid \tau}}{\partial x_\tau} \right|_{x_\tau, \theta_\tau, \theta_{\tau+1:t}} - \left.\frac{\partial f_{t\mid \tau}}{\partial x_\tau} \right|_{x_\tau', \theta_\tau', \theta_{\tau+1:t}}} \leq C_{\ell, f, (x, x)} \decayfactor^{t - \tau} \norm{x_\tau - x_\tau'} + C_{\ell, f, (x, \theta)} \decayfactor^{t - \tau} \norm{\theta_\tau - \theta_\tau'},\\
    &\norm{\left.\frac{\partial f_{t\mid \tau}}{\partial \theta_\tau} \right|_{x_\tau, \theta_\tau, \theta_{\tau+1:t}} - \left.\frac{\partial f_{t\mid \tau}}{\partial \theta_\tau} \right|_{x_\tau', \theta_\tau', \theta_{\tau+1:t}}} \leq C_{\ell, f, (\theta, x)} \decayfactor^{t - \tau} \norm{x_\tau - x_\tau'} + C_{\ell, f, (\theta, \theta)} \decayfactor^{t - \tau} \norm{\theta_\tau - \theta_\tau'},
\end{align*}
where $C_{L, f, x} = L_f C (1 + L_{\ALG, x}), C_{L, f, \theta} = L_f \max\{C_{L, g, \theta} (1 + L_{\ALG, x}), L_{\ALG, \theta}\}$, and
\begin{align*}
    C_{\ell, f, (x, x)} &= L_f (1 + L_{\ALG, x}) C_{\ell, g, (x, x)} + ((\ell_{f, x} + \ell_{f, u}L_{\ALG, x})(1 + L_{\ALG, x}) + L_f \ell_{\ALG, x}) C_{L, g, x}^2,\\
    C_{\ell, f, (x, \theta)} &= L_f (1 + L_{\ALG, x}) C_{\ell, g, (x, \theta)} + ((\ell_{f, x} + \ell_{f, u}L_{\ALG, x})(1 + L_{\ALG, x}) + L_f \ell_{\ALG, x}) C_{L, g, x} C_{L, g, \theta},\\
    C_{\ell, f, (\theta, x)} &= L_f (1 + L_{\ALG, x}) C_{\ell, g, (\theta, x)} + ((\ell_{f, x} + \ell_{f, u}L_{\ALG, x})(1 + L_{\ALG, x}) + L_f \ell_{\ALG, x}) C_{L, g, x} C_{L, g, \theta},\\
    C_{\ell, f, (\theta, \theta)} &= L_f (1 + L_{\ALG, x}) C_{\ell, g, (\theta, \theta)} + ((\ell_{f, x} + \ell_{f, u}L_{\ALG, x})(1 + L_{\ALG, x}) + L_f \ell_{\ALG, x}) C_{L, g, \theta}^2.
\end{align*}
\end{corollary}

With the help of \Cref{lemma:smooth-multi-step-dynamics} and \Cref{coro:smooth-multi-step-costs}, we first bound the cost difference $\abs{f_t(x_t, u_t) - F_t(\theta_t)}$ in \Cref{thm:OCO-with-parameter-update-trajectory-distance}. This inequality bounds the difference between the actual stage cost $f_t(x_t, u_t)$ incurred by GAPS and the ideal cost $F_t(\theta_t)$. Besides this inequality, in \Cref{thm:OCO-with-parameter-update-trajectory-distance}, we also bound the distance between GAPS' trajectory and the imaginary trajectory if the same policy parameter $\theta_t$ had been used from time $0$ to time $t$, which will be useful for showing \Cref{thm:OCO-with-parameter-update-gradient-bias} later in this section.

\begin{theorem}\label{thm:OCO-with-parameter-update-trajectory-distance}
Suppose Assumptions \ref{assump:Lipschitz-and-smoothness} and \ref{assump:contractive-and-stability} hold. Let $\{x_t, u_t, \theta_t\}_{t \in \mathcal{T}}$ denote the trajectory of GAPS (Algorithm \ref{alg:OCO-with-parameter-update}) with buffer size $B$ and a constant learning rate $\eta \leq \frac{(1 - \rho)\varepsilon}{C_{L, f, \theta}}$.
Then, both $\norm{G_t}$ and $\norm{\nabla F_t(\theta_t)}$ are upper bounded by $\frac{C_{L, f, \theta}}{1 - \decayfactor}$, and the following inequalities holds for any two time steps $\tau, t$ ($\tau \leq t$):
\begin{align*}
    &\norm{\theta_t - \theta_{\tau}} \leq \frac{C_{L, f, \theta}}{1 - \decayfactor}\cdot (t - \tau)\eta, \text{ and } \norm{x_\tau - \hat{x}_\tau(\theta_t)} \leq \frac{C_{L, f, \theta} C_{L, g, \theta} \decayfactor}{(1 - \decayfactor)^2} \left((t - \tau) + \frac{1}{1 - \decayfactor}\right) \cdot \eta,
\end{align*}
where we use the notation $\hat{x}_\tau(\theta) \coloneqq g_{0,\tau}(x_0, \theta_{\times (\tau + 1)}), \forall \theta \in \Theta$. Further, we have that
\[\abs{f_t(x_t, u_t) - F_t(\theta_t)} \leq \frac{C_{L, f, \theta} C_{L, g, \theta}L_f (1 + L_{\ALG, x}) \decayfactor}{(1 - \decayfactor)^3} \cdot \eta.\]
In addition, for any parameter sequence $\tilde{\theta}_{0:t} \in \Theta^{t+1}$, let $\tilde{x}_t$ and $\tilde{u}_t$ be the state/control action achieved by this sequence $\tilde{x}_t \coloneqq g_{t\mid 0}(x_0, \tilde{\theta}_{0:t-1})$ and $\tilde{u}_t \coloneqq \ALG_t(\tilde{x}_t, \tilde{\theta}_t)$. If $\norm{\tilde{x}_t} \leq \min\{R_C, R_x\}$ holds for all $t$, then the following inequality holds for all time $t$:
\[\abs{f_t(\tilde{x}_t, \tilde{u}_t) - F_t(\tilde{\theta}_t)} \leq \frac{C L_{\ALG, \theta} L_{g, x} L_f (1 + L_{\ALG, x})}{1 - \rho}\sum_{\tau = 0}^{t-1} \rho^{t - \tau - 1} \norm{\tilde{\theta}_{\tau + 1} - \tilde{\theta}_\tau}.\]
\end{theorem}

To show \Cref{thm:OCO-with-parameter-update-trajectory-distance}, we first derive a uniform upper bound on the norm the estimated gradient $G_t$, which implies that the policy parameter sequence does not vary too quickly, i.e. it is in the same order as the constant learning rate $\eta$. We then leverage  strong contractive perturbation  to bound $\norm{x_\tau - \hat{x}_\tau(\theta_t)}$ and use it to bound $\abs{f_t(x_t, u_t) - F_t(\theta_t)}$ by the Lipschitzness of $f_t$. We defer the detailed proof to \Cref{appendix:OCO-with-parameter-update-trajectory-distance}.

In \Cref{thm:OCO-with-parameter-update-gradient-bias} below, we bound the difference between the estimated gradient $G_t$ used by GAPS and the ideal gradient $\nabla F_t(\theta_t)$ used by the ideal OGD.

\begin{theorem}[Gradient Bias]\label{thm:OCO-with-parameter-update-gradient-bias}
Suppose Assumptions \ref{assump:Lipschitz-and-smoothness} and \ref{assump:contractive-and-stability} hold. Let $\{x_t, u_t, \theta_t\}_{t \in \mathcal{T}}$ denote the trajectory of GAPS (Algorithm \ref{alg:OCO-with-parameter-update}) with buffer size $B$ and learning rate $\eta \leq \frac{(1 - \rho)\varepsilon}{C_{L, f, \theta}}$. Then, the following holds for all $\tau \leq t$:
\begin{align*}
    \norm{\left.\frac{\partial f_{t\mid 0}}{\partial \theta_\tau} \right|_{x_0, \theta_{0:t}} - \left.\frac{\partial f_{t\mid 0}}{\partial \theta_\tau} \right|_{x_0, (\theta_t)_{\times (t+1)}}}
    ={}& O\left(\left(\frac{1}{(1-\rho)^4} + \frac{t - \tau}{(1 - \rho)^3} + \frac{(t - \tau)^2}{(1 - \rho)^2}\right) \decayfactor^{t - \tau} \cdot \eta\right),
\end{align*}
Further, we see that
\begin{align*}
    \norm{G_t - \nabla F_t(\theta_t)} \leq O\left(\frac{\eta}{(1 - \decayfactor)^5} + \frac{ \decayfactor^B}{1 - \decayfactor}\right).
\end{align*}
(See \Cref{appendix:OCO-with-parameter-update-gradient-bias} for the detailed expressions.)
\end{theorem}

The key technique we used to show \Cref{thm:OCO-with-parameter-update-gradient-bias} is a sequential decomposition of the error based on the triangle inequality. Specifically, note that \Cref{coro:smooth-multi-step-costs} only allow us to compare the partial derivatives when $\theta_{\tau+1: t}$ are fixed and the only perturbations are on $x_\tau$ and $\theta_\tau$. To compare the partial derivatives realized on two trajectory instances $(\hat{x}_\tau(\theta_t), (\theta_t)_{\times (t - \tau + 1)})$ and $(x_\tau, \theta_{\tau:t})$, we change the parameters sequentially one by one, following the path
\begin{align*}
&(\hat{x}_\tau(\theta_t), (\theta_t)_{\times (t - \tau + 1)}) \to (x_\tau, \theta_\tau, (\theta_t)_{\times (t - \tau)}) \to (x_\tau, \theta_{\tau:\tau+1}, (\theta_t)_{\times (t - \tau-1)}) \to \cdots \to (x_\tau, \theta_{\tau:t}).
\end{align*}
A complete proof of \Cref{thm:OCO-with-parameter-update-gradient-bias} can be found in \Cref{appendix:OCO-with-parameter-update-gradient-bias}.

The bounds in \Cref{thm:OCO-with-parameter-update-trajectory-distance,thm:OCO-with-parameter-update-gradient-bias} show that we can achieve our desired bounds $\abs{f_t(x_t, u_t) - F_t(\theta_t)} = O(1/\sqrt{T})$ and $\norm{\nabla F_t(\theta_t) - G_t} = O(1/\sqrt{T})$ if we set the learning rate and the buffer length to be  $O(1/\sqrt{T})$ and $O(\log T)$ respectively. Substituting these bounds into \Cref{thm:OGD-with-biased-gradient-regret} with a more careful analysis on the order of the factor $1/(1 - \decayfactor)$ will finish the proof of \Cref{thm:main-regret-bound-convex}. The detailed proof can be found in \Cref{appendix:OCO-with-parameter-update-gradient-bias}.

\subsection{Detailed Statement and Proofs of Theorems \ref{thm:bridge-GAPS-and-OGD} and \ref{thm:main-regret-bound-convex}}\label{appendix:detailed-statements-proofs}

We restate \Cref{thm:bridge-GAPS-and-OGD} with detailed expressions in \Cref{thm:bridge-GAPS-and-OGD:detailed}.

\begin{theorem}\label{thm:bridge-GAPS-and-OGD:detailed}
Suppose Assumptions \ref{assump:Lipschitz-and-smoothness} and \ref{assump:contractive-and-stability} hold. Let $\{x_t, u_t, \theta_t\}_{t \in \mathcal{T}}$ denote the trajectory of GAPS (Algorithm \ref{alg:OCO-with-parameter-update}) with buffer size $B$ and learning rate $\eta_t = \eta \leq \frac{(1 - \rho)\varepsilon}{C_{L, f, \theta}}$, where $C_{L, f, \theta}$ is defined in \Cref{coro:smooth-multi-step-costs}. Then, we have
\begin{align}\label{thm:main-regret-bound-convex:e01:detailed}
    \abs{f_t(x_t, u_t) - F_t(\theta_t)} \leq{}& \frac{C_{L, f, \theta} C_{L, g, \theta}L_f (1 + L_{\ALG, x}) \decayfactor}{(1 - \decayfactor)^3} \cdot \eta, \text{ and}\nonumber\\
    \norm{G_t - \nabla F_t(\theta_t)} \leq{}& \left({\hat{C}_0}{(1 - \decayfactor)^{-1}} + {(\hat{C}_1 + \hat{C}_2)}{(1 - \decayfactor)^{-2}} + {\hat{C}_2}{(1 - \decayfactor)^{-3}}\right) \eta \nonumber\\
    &+ {C_{L, f, \theta}}{(1 - \decayfactor)^{-1}} \cdot \decayfactor^B,
\end{align}
where $\hat{C}_0, \hat{C}_1, \hat{C}_2$ are defined in \Cref{thm:OCO-with-parameter-update-gradient-bias:full}.
\end{theorem}
\begin{proof}[Proof of \Cref{thm:bridge-GAPS-and-OGD:detailed}]
    \Cref{thm:bridge-GAPS-and-OGD:detailed} directly followed from \Cref{thm:OCO-with-parameter-update-trajectory-distance,thm:OCO-with-parameter-update-gradient-bias}.
\end{proof}

We restate \Cref{thm:main-regret-bound-convex} with detailed expressions in \Cref{thm:main-regret-bound-convex:detailed}.

\begin{theorem}\label{thm:main-regret-bound-convex:detailed}
Under the same assumptions as \Cref{thm:bridge-GAPS-and-OGD}, if we additionally assume the surrogate stage cost $F_t$ is convex for every time step $t$, then GAPS achieves the adaptive regret bound
\begin{align*}
    R^A(T)\leq{}& \left(\frac{C_{L, f, \theta}^2}{(1 - \decayfactor)^3} + \left(\frac{\hat{C}_0}{1 - \rho} + \frac{\hat{C}_1 + \hat{C}_2}{(1 - \rho)^2} + \frac{\hat{C}_2}{(1 - \rho)^3}\right)D\right)\eta T + \frac{D^2}{2\eta} + \frac{D C_{L, f,\theta}}{1 - \rho} \cdot \rho^B T\\
    &+2\left(\frac{\hat{C}_0}{1 - \rho} + \frac{\hat{C}_1 + \hat{C}_2}{(1 - \rho)^2} + \frac{\hat{C}_2}{(1 - \rho)^3}\right)^2 D^2 \eta^3 T + \frac{2 C_{L, f, \theta}^2}{(1 - \rho)^2}\rho^{2B}\eta T.
\end{align*}
\end{theorem}
\begin{proof}[Proof of \Cref{thm:main-regret-bound-convex:detailed}]
The first two inequalities are shown in \Cref{thm:OCO-with-parameter-update-trajectory-distance} and \Cref{thm:OCO-with-parameter-update-gradient-bias}. Thus, we focus on the adaptive regret part here in the proof.

Fix a time interval $I = [r: s] \subseteq \mathcal{T}$ and let $\theta_I$ be an arbitrary policy parameter in $\Theta$. By \Cref{thm:OGD-with-biased-gradient-regret} and \Cref{thm:OCO-with-parameter-update-gradient-bias}, we see that the sequence of policy parameters of the online policy satisfies that
\begin{align}\label{thm:main-regret-bound-convex:e1}
    \sum_{t = r}^{s}F_t(\theta_t) - \sum_{t=r}^{s}F_t(\theta_I) \leq (W^2 + \alpha^2)\eta T + \frac{D^2}{2\eta} + \alpha D T,
\end{align}
where $W = \frac{C_{L, f,\theta}}{1 - \rho}$ by \Cref{thm:OCO-with-parameter-update-trajectory-distance} and $\alpha = \left(\frac{\hat{C}_0}{1 - \rho} + \frac{\hat{C}_1 + \hat{C}_2}{(1 - \rho)^2} + \frac{\hat{C}_2}{(1 - \rho)^3}\right) \eta + \frac{C_{L, f, \theta}}{1 - \rho} \cdot \rho^B$.

Note that by definition, we have $F_t(\theta_I) = f_t(\hat{x}_t(\theta_I), \hat{u}_t(\theta_I))$ and by \Cref{thm:OCO-with-parameter-update-trajectory-distance}, we have
\begin{align}\label{thm:main-regret-bound-convex:e2}
    \abs{f_t(x_t, u_t) - F_t(\theta_t)} \leq \frac{C_{L, f, \theta} C_{L, g, \theta}L_f (1 + L_{\ALG, x}) \decayfactor}{(1 - \decayfactor)^3} \cdot \eta.
\end{align}
Substituting these into \eqref{thm:main-regret-bound-convex:e1} gives that
\begin{align*}
    &\sum_{t = r}^{s}f_t(x_t, u_t) - \sum_{t = r}^{s}f_t(\hat{x}_t(\theta_I), \hat{u}_t(\theta_I))\\
    \leq{}& \left(\sum_{t = r}^{s} F_t(\theta_t) - \sum_{t=r}^{s} F_t(\theta_I)\right) + \sum_{t = r}^{s}\abs{f_t(x_t, u_t) - F_t(\theta_t)}\\
    \leq{}& \frac{D^2}{2\eta} + \alpha D T + \left(W^2 + \alpha^2 + \frac{C_{L, f, \theta} C_{L, g, \theta}L_f (1 + L_{\ALG, x}) \decayfactor}{(1 - \decayfactor)^3}\right)\cdot \eta T\\
    \leq{}& \left(\frac{C_{L, f, \theta}^2}{(1 - \decayfactor)^3} + \left(\frac{\hat{C}_0}{1 - \rho} + \frac{\hat{C}_1 + \hat{C}_2}{(1 - \rho)^2} + \frac{\hat{C}_2}{(1 - \rho)^3}\right)D\right)\eta T + \frac{D^2}{2\eta} + \frac{D C_{L, f,\theta}}{1 - \rho} \cdot \rho^B T\\
    &+2\left(\frac{\hat{C}_0}{1 - \rho} + \frac{\hat{C}_1 + \hat{C}_2}{(1 - \rho)^2} + \frac{\hat{C}_2}{(1 - \rho)^3}\right)^2 D^2 \eta^3 T + \frac{2 C_{L, f, \theta}^2}{(1 - \rho)^2}\rho^{2B}\eta T,
\end{align*}
where we used \eqref{thm:main-regret-bound-convex:e1} and \eqref{thm:main-regret-bound-convex:e2} in the second inequality.
\end{proof}

\subsection{Proof of Theorem \ref{thm:OGD-with-biased-gradient-regret}}\label{appendix:OGD-with-bias}
Our proof is inspired by the proof of Theorem 2.1 in \cite{bansal2019potential}. For a fixed time interval $I = [r: s] \subseteq \mathcal{T}$ and $\theta_I \in \Theta$, we consider the potential function $\Phi_t = \frac{1}{2\eta} \norm{\theta_t - \theta_I}^2$. Note that $\theta_I$ satisfies $\norm{\theta_r-\theta_I}\leq D$ because we assume $\mathrm{diam}(\Theta) \leq D$. To simplify the notation, we define $\theta_{t+1}'=\theta_t-\eta G_t$.

By proposition 2.2 in \cite{bansal2019potential}, we see that
\begin{align*}
    \frac{1}{2} (\norm{\theta_{t+1}-\theta_I}^2-\norm{\theta_t-\theta_I}^2) &\leq \frac{1}{2} (\norm{\theta'_{t+1}-\theta_I}^2-\norm{\theta_t-\theta_I}^2)\\
    &= \langle \theta'_{t+1}-\theta_t, \theta_t-\theta_I\rangle + \frac{1}{2}\norm{\theta'_{t+1}-\theta_t}^2 \\
    &= \eta \langle G_t, \theta_I-\theta_t\rangle + \frac{\eta^2}{2} \norm{G_t}^2.
\end{align*}
Using this inequality, we see that
\begin{subequations}\label{thm:OGD-with-biased-gradient:e1}
\begin{align}
    &F_t(\theta_t)-F_t(\theta_I)+\Phi_{t+1}-\Phi_t \nonumber\\
    ={}& F_t(\theta_t)-F_t(\theta_I) + \langle G_t, \theta_I-\theta_t\rangle + \frac{\eta}{2} \norm{G_t}^2\nonumber\\
    ={}& F_t(\theta_t)-F_t(\theta_I) + \langle \nabla F_t(\theta_t) + (G_t - \nabla F_t(\theta_t)) , \theta_I-\theta_t\rangle + \frac{\eta}{2} \norm{\nabla F_t(\theta_t)+ (G_t - F_t(\theta_t))}^2 \nonumber\\
    \leq{}& F_t(\theta_t)-F_t(\theta_I) + \langle \nabla F_t(\theta_t), \theta_I-\theta_t\rangle + \langle G_t - \nabla F_t(\theta_t), \theta_I-\theta_t\rangle 
    + \eta \norm{\nabla F_t(\theta_t)}^2\nonumber\\
    &+\eta \norm{G_t - \nabla F_t(\theta_t)}^2 \label{thm:OGD-with-biased-gradient:e1:s1}\\
    \leq{}& 0 + \norm{G_t - \nabla F_t(\theta_t)} \cdot \norm{\theta_I-\theta_t} + \eta \norm{\nabla F_t(\theta_t)}^2+\eta \alpha^2 \label{thm:OGD-with-biased-gradient:e1:s2}\\
    \leq{}& \alpha D + W^2\eta + \eta\alpha^2, \label{thm:OGD-with-biased-gradient:e1:s3}
\end{align}
\end{subequations}
where we used the triangle inequality and the AM-GM inequality in \eqref{thm:OGD-with-biased-gradient:e1:s1};
we used the assumption that $F_t$ is convex, $\norm{G_t - \nabla F_t(\theta_t)} \leq \alpha$, and the Cauchy-Schwarz inequality in \eqref{thm:OGD-with-biased-gradient:e1:s2};
and we used the assumptions $\norm{G_t - \nabla F_t(\theta_t)} \leq \alpha$, $\mathrm{diam}(\Theta) \leq D$, $\norm{\nabla F_t(\theta_t)} \leq W$ in \eqref{thm:OGD-with-biased-gradient:e1:s3}.

Summing \eqref{thm:OGD-with-biased-gradient:e1} over the time interval $[r: s]$ gives that
\begin{align*}
    \sum_{t = r}^s \left(F_t(\theta_t) - F_t(\theta_I)\right) \leq{}& ({s - r}) \cdot \left(\alpha D + W^2\eta + \eta\alpha^2\right) + ({\Phi_r - \Phi_{s+1}})\\
    \leq{}& \left(\alpha D + W^2\eta + \eta\alpha^2\right) T + \frac{D^2}{2 \eta},
\end{align*}
where we used $\mathrm{diam}(\Theta) \leq D$ and $\Phi_{s + 1} \geq 0$ in the last inequality. Since this inequality holds for any time interval $I = [r: s]$ and $\theta_I \in \Theta$, this finishes the proof of the first part of \Cref{thm:OGD-with-biased-gradient-regret}.

\subsection{Proof of Inequalities 1,2, and 4 in Lemma \ref{lemma:smooth-multi-step-dynamics}}\label{appendix:smooth-multi-step-dynamics}
The first inequality directly follows from $\varepsilon$-time-varying contractive perturbation (\Cref{def:epsilon-exp-decay-perturbation-property}). For the second inequality, when $t = \tau + 1$, note that
$\left.\frac{\partial g_{\tau + 1\mid \tau}}{\partial \theta_\tau}\right|_{x_\tau, \theta_\tau} = \left.\frac{\partial g_\tau}{\partial u_\tau}\right|_{x_\tau, u_\tau} \cdot \left.\frac{\partial \ALG_\tau}{\partial \theta_\tau}\right|_{x_\tau, \theta_\tau},$
where $u_\tau = \ALG_\tau(x_\tau, \theta_\tau)$. Taking norms of both sides of the equation gives
\begin{align*}
    \norm{\left.\frac{\partial g_{\tau + 1\mid \tau}}{\partial \theta_\tau}\right|_{x_\tau, \theta_\tau}} = \norm{\left.\frac{\partial g_\tau}{\partial u_\tau}\right|_{x_\tau, u_\tau} \cdot \left.\frac{\partial \ALG_\tau}{\partial \theta_\tau}\right|_{x_\tau, \theta_\tau}} \leq \norm{\left.\frac{\partial g_\tau}{\partial u_\tau}\right|_{x_\tau, u_\tau}} \cdot \norm{\left.\frac{\partial \ALG_\tau}{\partial \theta_\tau}\right|_{x_\tau, \theta_\tau}} \leq L_{g, u} L_{\ALG, \theta}.
\end{align*}
When $t > \tau + 1$, we see that
\begin{align*}
    \norm{\left.\frac{\partial g_{t\mid \tau}}{\partial \theta_\tau} \right|_{x_\tau, \theta_{\tau:t-1}}} &= \norm{\left.\frac{\partial g_{t\mid \tau+1}}{\partial x_{\tau+1}} \right|_{x_{\tau+1}, \theta_{\tau+1:t-1}} \cdot \left.\frac{\partial g_{\tau+1\mid \tau}}{\partial \theta_\tau}\right|_{x_\tau, \theta_\tau}}\\
    &\leq \norm{\left.\frac{\partial g_{t\mid \tau+1}}{\partial x_{\tau+1}} \right|_{x_{\tau+1}, \theta_{\tau+1:t-1}}}\cdot \norm{\left.\frac{\partial g_{\tau+1\mid \tau}}{\partial \theta_\tau}\right|_{x_\tau, \theta_\tau}} \leq \frac{C_0 L_{g, u} L_{\ALG, \theta}}{\decayfactor} \cdot \decayfactor^{t - \tau},
\end{align*}
where $x_{\tau + 1} = g_{\tau+1\mid \tau}(x_\tau, \theta_\tau)$.

For the last inequality of \Cref{lemma:smooth-multi-step-dynamics}, when $t = \tau + 1$, we see that
\begin{subequations}\label{thm:smooth-multi-step-dynamics:e1}
\begin{align}
    &\norm{\left.\frac{\partial g_{\tau+1\mid \tau}}{\partial \theta_\tau}\right|_{x_\tau, \theta_\tau} - \left.\frac{\partial g_{\tau+1\mid \tau}}{\partial \theta_\tau}\right|_{x_\tau', \theta_\tau'}}\nonumber\\
    ={}& \norm{\left.\frac{\partial g_\tau}{\partial u_\tau}\right|_{x_\tau, u_\tau} \cdot \left.\frac{\partial \ALG_\tau}{\partial \theta_\tau}\right|_{x_\tau, \theta_\tau} - \left.\frac{\partial g_\tau}{\partial u_\tau}\right|_{x_\tau', u_\tau'} \cdot \left.\frac{\partial \ALG_\tau}{\partial \theta_\tau}\right|_{x_\tau', \theta_\tau'}}\nonumber\\
    \leq{}& \norm{\left(\left.\frac{\partial g_\tau}{\partial u_\tau}\right|_{x_\tau, u_\tau} - \left.\frac{\partial g_\tau}{\partial u_\tau}\right|_{x_\tau', u_\tau'}\right)\cdot \left.\frac{\partial \ALG_\tau}{\partial \theta_\tau}\right|_{x_\tau, \theta_\tau}} + \norm{\left.\frac{\partial g_\tau}{\partial u_\tau}\right|_{x_\tau', u_\tau'} \left(\left.\frac{\partial \ALG_\tau}{\partial \theta_\tau}\right|_{x_\tau, \theta_\tau} - \left.\frac{\partial \ALG_\tau}{\partial \theta_\tau}\right|_{x_\tau', \theta_\tau'}\right)}\label{thm:smooth-multi-step-dynamics:e1:s1}\\
    \leq{}& L_{\ALG, \theta} \left(\ell_{g, x}\norm{x_\tau - x_\tau'} + \ell_{g, u} \norm{u_\tau - u_\tau'}\right) + L_{g, u} \left(\ell_{\ALG, x} \norm{x_\tau - x_\tau'} + \ell_{\ALG, \theta} \norm{\theta_\tau - \theta_\tau'}\right)\label{thm:smooth-multi-step-dynamics:e1:s2}\\
    \leq{}& \left(L_{\ALG, \theta}(\ell_{g, x} + \ell_{g, u} L_{\ALG, x}) + L_{g, u} \ell_{\ALG, x}\right)\norm{x_\tau - x_\tau'} + (L_{\ALG, \theta}^2 \ell_{g, u} + L_{g, u} \ell_{\ALG, \theta})\norm{\theta_\tau - \theta_\tau'}, \label{thm:smooth-multi-step-dynamics:e1:s3}
\end{align}
\end{subequations}
where we use the notations $u_\tau = \ALG_\tau(x_\tau, \theta_\tau), u_\tau' = \ALG_\tau(x_\tau, \theta_\tau')$. We use the triangle inequality in \eqref{thm:smooth-multi-step-dynamics:e1:s1}; we use \Cref{assump:Lipschitz-and-smoothness} in both \eqref{thm:smooth-multi-step-dynamics:e1:s2} and \eqref{thm:smooth-multi-step-dynamics:e1:s3}.

When $t > \tau + 1$, we see that
\begin{subequations}\label{thm:smooth-multi-step-dynamics:e2}
\begin{align}
    &\norm{\left.\frac{\partial g_{t\mid \tau}}{\partial \theta_\tau} \right|_{x_\tau, \theta_\tau, \theta_{\tau+1:t-1}} - \left.\frac{\partial g_{t\mid \tau}}{\partial \theta_\tau} \right|_{x_\tau', \theta_\tau', \theta_{\tau+1:t-1}}}\nonumber\\
    \leq{}& \norm{\left.\frac{\partial g_{t\mid \tau+1}}{\partial x_{\tau+1}} \right|_{x_{\tau+1}, \theta_{\tau+1:t-1}}\cdot \left.\frac{\partial g_{\tau+1\mid \tau}}{\partial \theta_\tau}\right|_{x_\tau, \theta_\tau} - \left.\frac{\partial g_{t\mid \tau+1}}{\partial x_{\tau+1}} \right|_{x_{\tau+1}', \theta_{\tau+1:t-1}}\cdot \left.\frac{\partial g_{\tau+1\mid \tau}}{\partial \theta_\tau}\right|_{x_\tau', \theta_\tau'}}\label{thm:smooth-multi-step-dynamics:e2:s1}\\
    \leq{}& \norm{\left(\left.\frac{\partial g_{t\mid \tau+1}}{\partial x_{\tau+1}} \right|_{x_{\tau+1}, \theta_{\tau+1:t-1}} - \left.\frac{\partial g_{t\mid \tau+1}}{\partial x_{\tau+1}} \right|_{x_{\tau+1}', \theta_{\tau+1:t-1}}\right)\left.\frac{\partial g_{\tau+1\mid \tau}}{\partial \theta_\tau}\right|_{x_\tau, \theta_\tau}}\nonumber\\
    &+ \norm{\left.\frac{\partial g_{t\mid \tau+1}}{\partial x_{\tau+1}} \right|_{x_{\tau+1}', \theta_{\tau+1:t-1}}\cdot \left(\left.\frac{\partial g_{\tau+1\mid \tau}}{\partial \theta_\tau}\right|_{x_\tau, \theta_\tau} - \left.\frac{\partial g_{\tau+1\mid \tau}}{\partial \theta_\tau}\right|_{x_\tau', \theta_\tau'}\right)}\label{thm:smooth-multi-step-dynamics:e2:s2}\\
    \leq{}& \norm{\left.\frac{\partial g_{t\mid \tau+1}}{\partial x_{\tau+1}} \right|_{x_{\tau+1}, \theta_{\tau+1:t-1}} - \left.\frac{\partial g_{t\mid \tau+1}}{\partial x_{\tau+1}} \right|_{x_{\tau+1}', \theta_{\tau+1:t-1}}} \cdot \norm{\left.\frac{\partial g_{\tau+1\mid \tau}}{\partial \theta_\tau}\right|_{x_\tau, \theta_\tau}}\nonumber\\
    &+ \norm{\left.\frac{\partial g_{t\mid \tau+1}}{\partial x_{\tau+1}} \right|_{x_{\tau+1}', \theta_{\tau+1:t-1}}} \cdot \norm{\left.\frac{\partial g_{\tau+1\mid \tau}}{\partial \theta_\tau}\right|_{x_\tau, \theta_\tau} - \left.\frac{\partial g_{\tau+1\mid \tau}}{\partial \theta_\tau}\right|_{x_\tau', \theta_\tau'}}\nonumber\\
    \leq{}& \frac{\left((1 + L_{\ALG, x}) \left(\ell_{g, x} + \ell_{g, u} \cdot L_{\ALG, x}\right) + L_{g, x} \cdot \ell_{\ALG, x}\right) C_0^3}{\decayfactor(1 - \decayfactor)} \cdot \decayfactor^{t - \tau - 1} \cdot \norm{x_{\tau+1} - x_{\tau+1}'} \cdot L_{g, u} L_{\ALG, \theta}\nonumber\\
    &+ C_0 \cdot \decayfactor^{t-\tau-1} \cdot \left(L_{\ALG, \theta}(\ell_{g, x} + \ell_{g, u} L_{\ALG, x}) + L_{g, u} \ell_{\ALG, x}\right)\norm{x_\tau - x_\tau'}\nonumber\\
    &+ C_0 \cdot \decayfactor^{t - \tau - 1} \cdot (L_{\ALG, \theta}^2 \ell_{g, u} + L_{g, u} \ell_{\ALG, \theta})\norm{\theta_\tau - \theta_\tau'}\label{thm:smooth-multi-step-dynamics:e2:s3}\\
    \leq{}& C_{\ell, g, (\theta, x)} \decayfactor^{t - \tau} \norm{x_\tau - x_\tau'} + C_{\ell, g, (\theta, \theta)} \decayfactor^{t-\tau} \norm{\theta_\tau - \theta_\tau'},\label{thm:smooth-multi-step-dynamics:e2:s4}
\end{align}
\end{subequations}
where we use the notations $x_{\tau + 1} = g_{\tau+1\mid \tau}(x_\tau, \theta_\tau), x_{\tau + 1}' = g_{\tau+1\mid \tau}(x_\tau, \theta_\tau')$. We use the chain rule decomposition in \eqref{thm:smooth-multi-step-dynamics:e2:s1}; we use the triangle inequality in \eqref{thm:smooth-multi-step-dynamics:e2:s2}; we use the first and the third inequality of \Cref{lemma:smooth-multi-step-dynamics} as well as \eqref{thm:smooth-multi-step-dynamics:e1} in \eqref{thm:smooth-multi-step-dynamics:e2:s3}; we use the first two inequalities of \Cref{lemma:smooth-multi-step-dynamics} in \eqref{thm:smooth-multi-step-dynamics:e2:s4}.

\subsection{Proof of Corollary \ref{coro:smooth-multi-step-costs}}\label{appendix:smooth-multi-step-costs}
To show the first inequality, note that
\begin{align}\label{coro:smooth-multi-step-costs:e1}
    \left.\frac{\partial f_{t\mid \tau}}{\partial x_\tau} \right|_{x_\tau, \theta_{\tau:t}} = \left(\left.\frac{\partial f_t}{\partial x_t}\right|_{x_t, u_t} + \left.\frac{\partial f_t}{\partial u_t}\right|_{x_t, u_t}\cdot \left.\frac{\partial \ALG_t}{\partial x_t}\right|_{x_t, \theta_t} \right)\cdot \left.\frac{\partial g_{t\mid \tau}}{\partial x_\tau}\right|_{x_\tau, \theta_{\tau:t-1}},
\end{align}
where $x_t = g_{t\mid \tau}(x_\tau, \theta_{\tau:t-1}), u_t = \ALG_t(x_t, \theta_t)$. Thus, by $\varepsilon$-time-varying contractive perturbation, we see that
\begin{align*}
    \norm{\left.\frac{\partial f_{t\mid \tau}}{\partial x_\tau} \right|_{x_\tau, \theta_{\tau:t}}} \leq{}& \left(\norm{\left.\frac{\partial f_t}{\partial x_t}\right|_{x_t, u_t}} + \norm{\left.\frac{\partial f_t}{\partial u_t}\right|_{x_t, u_t}}\cdot \norm{\left.\frac{\partial \ALG_t}{\partial x_t}\right|_{x_t, \theta_t}} \right)\cdot \norm{\left.\frac{\partial g_{t\mid \tau}}{\partial x_\tau}\right|_{x_\tau, \theta_{\tau:t-1}}}\\*
    \leq{}& L_f (1 + L_{\ALG, x}) \cdot C \decayfactor^{t - \tau}.
\end{align*}

For the second inequality, when $\tau = t$, since $x_t \in B_n(0, R_x')$ and $u_t \in B_m(0, R_u')$, we see that
\begin{align*}
    \norm{\left.\frac{\partial f_{t\mid t}}{\partial \theta_t} \right|_{x_t, \theta_t}} = \norm{\left.\frac{\partial f_t}{\partial u_t}\right|_{x_t, u_t} \cdot \left.\frac{\partial \ALG_t}{\partial \theta_t}\right|_{x_t, \theta_t}} \leq \norm{\left.\frac{\partial f_t}{\partial u_t}\right|_{x_t, u_t}} \cdot \norm{\left.\frac{\partial \ALG_t}{\partial \theta_t}\right|_{x_t, \theta_t}} \leq L_f L_{\ALG, \theta}.
\end{align*}
When $\tau < t$, the second inequality can be shown similarly with the first inequality in \Cref{coro:smooth-multi-step-costs} because we have the chain-rule decomposition
\begin{align}\label{coro:smooth-multi-step-costs:e2}
    \left.\frac{\partial f_{t\mid \tau}}{\partial \theta_\tau} \right|_{x_\tau, \theta_{\tau:t}} = \left(\left.\frac{\partial f_t}{\partial x_t}\right|_{x_t, u_t} + \left.\frac{\partial f_t}{\partial u_t}\right|_{x_t, u_t}\cdot \left.\frac{\partial \ALG_t}{\partial x_t}\right|_{x_t, \theta_t} \right)\cdot \left.\frac{\partial g_{t\mid \tau}}{\partial \theta_\tau}\right|_{x_\tau, \theta_{\tau:t-1}}.
\end{align}
Applying \Cref{lemma:smooth-multi-step-dynamics} gives that $C_{L, f, \theta} = L_f C_{L, g, \theta} (1 + L_{\ALG, x})$.

For the third inequality, using \eqref{coro:smooth-multi-step-costs:e1}, we see that
\begin{subequations}\label{coro:smooth-multi-step-costs:e3}
\begin{align}
    &\norm{\left.\frac{\partial f_{t\mid \tau}}{\partial x_\tau} \right|_{x_\tau, \theta_\tau, \theta_{\tau+1:t}} - \left.\frac{\partial f_{t\mid \tau}}{\partial x_\tau} \right|_{x_\tau', \theta_\tau', \theta_{\tau+1:t}}}\nonumber\\
    \leq{}& \norm{\left.\frac{\partial f_t}{\partial x_t}\right|_{x_t, u_t} \cdot \left(\left.\frac{\partial g_{t\mid \tau}}{\partial x_\tau}\right|_{x_\tau, \theta_\tau, \theta_{\tau+1:t-1}} - \left.\frac{\partial g_{t\mid \tau}}{\partial x_\tau}\right|_{x_\tau', \theta_\tau', \theta_{\tau+1:t-1}}\right)}\nonumber\\
    &+ \norm{\left(\left.\frac{\partial f_t}{\partial x_t}\right|_{x_t, u_t} - \left.\frac{\partial f_t}{\partial x_t}\right|_{x_t', u_t'}\right)\cdot \left.\frac{\partial g_{t\mid \tau}}{\partial x_\tau}\right|_{x_\tau', \theta_\tau', \theta_{\tau+1:t-1}}}\nonumber\\
    &+ \norm{\left.\frac{\partial f_t}{\partial u_t}\right|_{x_t, u_t}\cdot \left.\frac{\partial \ALG_t}{\partial x_t}\right|_{x_t, \theta_t} \cdot \left(\left.\frac{\partial g_{t\mid \tau}}{\partial x_\tau}\right|_{x_\tau, \theta_\tau, \theta_{\tau+1:t-1}} - \left.\frac{\partial g_{t\mid \tau}}{\partial x_\tau}\right|_{x_\tau', \theta_\tau', \theta_{\tau+1:t-1}}\right)}\nonumber\\
    &+ \norm{\left.\frac{\partial f_t}{\partial u_t}\right|_{x_t, u_t}\cdot \left(\left.\frac{\partial \ALG_t}{\partial x_t}\right|_{x_t, \theta_t} - \left.\frac{\partial \ALG_t}{\partial x_t}\right|_{x_t', \theta_t}\right)\cdot \left.\frac{\partial g_{t\mid \tau}}{\partial x_\tau}\right|_{x_\tau', \theta_\tau', \theta_{\tau+1:t-1}}}\nonumber\\
    &+ \norm{\left(\left.\frac{\partial f_t}{\partial u_t}\right|_{x_t, u_t} - \left.\frac{\partial f_t}{\partial u_t}\right|_{x_t', u_t'}\right)\cdot \left.\frac{\partial \ALG_t}{\partial x_t}\right|_{x_t', \theta_t} \cdot \left.\frac{\partial g_{t\mid \tau}}{\partial x_\tau}\right|_{x_\tau', \theta_\tau', \theta_{\tau+1:t-1}}} \label{coro:smooth-multi-step-costs:e3:s1}\\
    \leq{}& L_f \decayfactor^{t - \tau} \left(C_{\ell, g, (x, x)} \norm{x_\tau - x_\tau'} + C_{\ell, g, (x, \theta)} \norm{\theta_\tau - \theta_\tau'}\right)\nonumber\\
    &+ (\ell_{f, x}\norm{x_t - x_t'} + \ell_{f, u}\norm{u_t - u_t'})\cdot C_{L, g, x} \decayfactor^{t - \tau}\nonumber\\
    &+ L_f L_{\ALG, x} \decayfactor^{t - \tau} \left(C_{\ell, g, (x, x)} \norm{x_\tau - x_\tau'} + C_{\ell, g, (x, \theta)} \norm{\theta_\tau - \theta_\tau'}\right)\nonumber\\
    &+ L_f \ell_{\ALG, x} \norm{x_t - x_t'}\cdot C_{L, g, x} \decayfactor^{t - \tau} + (\ell_{f, x}\norm{x_t - x_t'} + \ell_{f, u}\norm{u_t - u_t'})\cdot L_{\ALG, x} C_{L, g, x} \decayfactor^{t - \tau}, \label{coro:smooth-multi-step-costs:e3:s2}
\end{align}
\end{subequations}
where we use the notations $x_t = g_{t\mid \tau}(x_\tau, \theta_{\tau:t-1}), x_t' = g_{t\mid \tau}(x_\tau', \theta_{\tau:t-1}), u_t = \ALG_t(x_t, \theta_t), u_t' = \ALG_t(x_t', \theta_t)$. We use \eqref{coro:smooth-multi-step-costs:e1} and the triangle inequality in \eqref{coro:smooth-multi-step-costs:e3:s1}; we use \Cref{lemma:smooth-multi-step-dynamics} in \eqref{coro:smooth-multi-step-costs:e3:s2}. Note that by the first two inequalities in \Cref{lemma:smooth-multi-step-dynamics}, we have
\begin{align*}
    \norm{x_t - x_t'} &\leq \decayfactor^{t-\tau} \left(C_{L, g, x} \norm{x_\tau - x_\tau'} + C_{L, g, \theta} \norm{\theta_\tau - \theta_\tau'}\right),\\
    \norm{u_t - u_t'} &\leq L_{\ALG, x} \decayfactor^{t-\tau} \left(C_{L, g, x} \norm{x_\tau - x_\tau'} + C_{L, g, \theta} \norm{\theta_\tau - \theta_\tau'}\right).
\end{align*}
Substituting these two inequalities into \eqref{coro:smooth-multi-step-costs:e3} finishes the proof of the third inequality.

For the last inequality, when $\tau = t$, we have that
\begin{subequations}\label{coro:smooth-multi-step-costs:e4}
\begin{align}
    &\norm{\left.\frac{\partial f_{t, t}}{\partial \theta_t} \right|_{x_t, \theta_{t}} - \left.\frac{\partial f_{t, t}}{\partial \theta_t} \right|_{x_t', \theta_{t}'}}\nonumber\\
    ={}& \norm{\left.\frac{\partial f_t}{\partial u_t}\right|_{x_t, u_t} \cdot \left.\frac{\partial \ALG_t}{\partial \theta_t}\right|_{x_t, \theta_t} - \left.\frac{\partial f_t}{\partial u_t}\right|_{x_t', u_t'} \cdot \left.\frac{\partial \ALG_t}{\partial \theta_t}\right|_{x_t', \theta_t'}}\label{coro:smooth-multi-step-costs:e4:s1}\\
    \leq{}& \norm{\left.\frac{\partial f_t}{\partial u_t}\right|_{x_t, u_t} \cdot \left(\left.\frac{\partial \ALG_t}{\partial \theta_t}\right|_{x_t, \theta_t} - \left.\frac{\partial \ALG_t}{\partial \theta_t}\right|_{x_t', \theta_t'}\right)} + \norm{\left(\left.\frac{\partial f_t}{\partial u_t}\right|_{x_t, u_t} - \left.\frac{\partial f_t}{\partial u_t}\right|_{x_t', u_t'}\right)\cdot \left.\frac{\partial \ALG_t}{\partial \theta_t}\right|_{x_t', \theta_t'}}\label{coro:smooth-multi-step-costs:e4:s2}\\
    \leq{}& L_f \left(\ell_{\ALG, \theta} \norm{\theta_t - \theta_t'} + \ell_{\ALG, x} \norm{x_t - x_t'}\right) + \left(\ell_{f, x}\norm{x_t - x_t'} + \ell_{f, u}\norm{u_t - u_t'}\right) \cdot L_{\ALG, \theta}\label{coro:smooth-multi-step-costs:e4:s3}\\
    \leq{}& \left(L_f \ell_{\ALG, \theta} + (\ell_{f, x} + \ell_{f, u} L_{\ALG, x})L_{\ALG, \theta}\right)\norm{x_t - x_t'} + \left(L_f \ell_{\ALG, \theta} + \ell_{f, u} L_{\ALG, \theta}^2\right) \norm{\theta_t - \theta_t'},\label{coro:smooth-multi-step-costs:e4:s4}
\end{align}
\end{subequations}
where we use the chain rule decomposition in \eqref{coro:smooth-multi-step-costs:e4:s1}; we use the triangle inequality in \eqref{coro:smooth-multi-step-costs:e4:s2}; we use \Cref{assump:Lipschitz-and-smoothness} in both \eqref{coro:smooth-multi-step-costs:e4:s3} and \eqref{coro:smooth-multi-step-costs:e4:s4}.

When $\tau < t$, by \eqref{coro:smooth-multi-step-costs:e2}, we have that
\begin{subequations}\label{coro:smooth-multi-step-costs:e5}
\begin{align}
    &\norm{\left.\frac{\partial f_{t\mid \tau}}{\partial \theta_\tau} \right|_{x_\tau, \theta_\tau, \theta_{\tau+1:t}} - \left.\frac{\partial f_{t\mid \tau}}{\partial \theta_\tau} \right|_{x_\tau', \theta_\tau', \theta_{\tau+1:t}}}\nonumber\\
    \leq{}& \norm{\left.\frac{\partial f_t}{\partial x_t}\right|_{x_t, u_t} \cdot \left(\left.\frac{\partial g_{t\mid \tau}}{\partial \theta_\tau}\right|_{x_\tau, \theta_\tau, \theta_{\tau+1:t-1}} - \left.\frac{\partial g_{t\mid \tau}}{\partial \theta_\tau}\right|_{x_\tau', \theta_\tau', \theta_{\tau+1:t-1}}\right)}\nonumber\\
    &+ \norm{\left(\left.\frac{\partial f_t}{\partial x_t}\right|_{x_t, u_t} - \left.\frac{\partial f_t}{\partial x_t}\right|_{x_t', u_t'}\right)\cdot \left.\frac{\partial g_{t\mid \tau}}{\partial \theta_\tau}\right|_{x_\tau', \theta_\tau', \theta_{\tau+1:t-1}}}\nonumber\\
    &+ \norm{\left.\frac{\partial f_t}{\partial u_t}\right|_{x_t, u_t}\cdot \left.\frac{\partial \ALG_t}{\partial x_t}\right|_{x_t, \theta_t} \cdot \left(\left.\frac{\partial g_{t\mid \tau}}{\partial \theta_\tau}\right|_{x_\tau, \theta_\tau, \theta_{\tau+1:t-1}} - \left.\frac{\partial g_{t\mid \tau}}{\partial \theta_\tau}\right|_{x_\tau', \theta_\tau', \theta_{\tau+1:t-1}}\right)}\nonumber\\
    &+ \norm{\left.\frac{\partial f_t}{\partial u_t}\right|_{x_t, u_t}\cdot \left(\left.\frac{\partial \ALG_t}{\partial x_t}\right|_{x_t, \theta_t} - \left.\frac{\partial \ALG_t}{\partial x_t}\right|_{x_t', \theta_t}\right)\cdot \left.\frac{\partial g_{t\mid \tau}}{\partial \theta_\tau}\right|_{x_\tau', \theta_\tau', \theta_{\tau+1:t-1}}}\nonumber\\
    &+ \norm{\left(\left.\frac{\partial f_t}{\partial u_t}\right|_{x_t, u_t} - \left.\frac{\partial f_t}{\partial u_t}\right|_{x_t', u_t'}\right)\cdot \left.\frac{\partial \ALG_t}{\partial x_t}\right|_{x_t', \theta_t} \cdot \left.\frac{\partial g_{t\mid \tau}}{\partial \theta_\tau}\right|_{x_\tau', \theta_\tau', \theta_{\tau+1:t-1}}}\label{coro:smooth-multi-step-costs:e5:s1}\\
    \leq{}& L_f \decayfactor^{t - \tau} \left(C_{\ell, g, (\theta, x)}\norm{x_\tau - x_\tau'} + C_{\ell, g, (\theta, \theta)}\norm{\theta_\tau - \theta_\tau'}\right)\nonumber\\ 
    &+ (\ell_{f, x}\norm{x_t - x_t'} + \ell_{f, u}\norm{u_t - u_t'})\cdot C_{L, g, \theta} \decayfactor^{t - \tau}\nonumber\\
    &+ L_f L_{\ALG, x} \decayfactor^{t - \tau} \left(C_{\ell, g, (\theta, x)}\norm{x_\tau - x_\tau'} + C_{\ell, g, (\theta, \theta)}\norm{\theta_\tau - \theta_\tau'}\right)\nonumber\\
    &+ L_f \ell_{\ALG, x} \norm{x_t - x_t'}\cdot C_{L, g, \theta} \decayfactor^{t - \tau} + (\ell_{f, x}\norm{x_t - x_t'} + \ell_{f, u}\norm{u_t - u_t'})\cdot L_{\ALG, x} C_{L, g, \theta} \decayfactor^{t - \tau}, \label{coro:smooth-multi-step-costs:e5:s2}
\end{align}
\end{subequations}
where we use \eqref{coro:smooth-multi-step-costs:e2} and the triangle inequality in \eqref{coro:smooth-multi-step-costs:e5:s1}; we use \Cref{assump:Lipschitz-and-smoothness} in \eqref{coro:smooth-multi-step-costs:e5:s2}. Note that by the first two inequalities in \Cref{lemma:smooth-multi-step-dynamics}, we have
\begin{align*}
    \norm{x_t - x_t'} &\leq \decayfactor^{t-\tau} \left(C_{L, g, x} \norm{x_\tau - x_\tau'} + C_{L, g, \theta} \norm{\theta_\tau - \theta_\tau'}\right),\\
    \norm{u_t - u_t'} &\leq L_{\ALG, x} \decayfactor^{t-\tau} \left(C_{L, g, x} \norm{x_\tau - x_\tau'} + C_{L, g, \theta} \norm{\theta_\tau - \theta_\tau'}\right).
\end{align*}
Substituting these into \eqref{coro:smooth-multi-step-costs:e5} finishes the proof of the fourth inequality.

\subsection{Proof of Theorem \ref{thm:OCO-with-parameter-update-trajectory-distance}}\label{appendix:OCO-with-parameter-update-trajectory-distance}
To simplify the notation, we use the shorthand $\hat{x}_\tau(\theta) \coloneqq g_{\tau\mid 0}(x_0, \theta_{\times \tau})$ and $\hat{u}_\tau(\theta) = \ALG_\tau(\hat{x}_\tau(\theta), \theta)$ for any time $\tau$ and policy parameter $\theta$.

We first derive an upper bound of $G_t$ in order to bound the difference between $\theta_t$ and $\theta_{t+1}$. Recall that
\begin{align}\label{thm:OCO-with-parameter-update-trajectory-distance:e0-0}
    G_t \coloneqq \sum_{\tau = 0}^{\min\{t, B-1\}} \left.\frac{\partial f_{t\mid 0}}{\partial \theta_{t - \tau}}\right|_{x_{0}, \theta_{0:t}} = \sum_{\tau = 0}^{\min\{t, B-1\}} \left.\frac{\partial f_{t\mid t-\tau}}{\partial \theta_{t - \tau}}\right|_{x_{t - \tau}, \theta_{t-\tau:t}}.
\end{align}
Now we use induction to show that for all time step $t \in \mathcal{T}$,
\begin{equation}\label{thm:OCO-with-parameter-update-trajectory-distance:e0}
    \norm{G_t} \leq \frac{C_{L, f, \theta}}{1 - \decayfactor}, x_t \in B_n(0, R_S + C \norm{x_0}), u_t \in \mathcal{U}, \text{ and }\norm{\theta_{t+1} - \theta_t} \leq \varepsilon.
\end{equation}
Note that $\norm{G_0} \leq C_{L, f, \theta} \leq \frac{C_{L, f, \theta}}{1 - \decayfactor}$ by \Cref{coro:smooth-multi-step-costs}. We also have $x_0 \in B_n(0, R_S + C\norm{x_0})$ and $u_0 \in \mathcal{U}$.

Suppose $\norm{G_{t-1}} \leq \frac{C_{L, f, \theta}}{1 - \decayfactor}$ for some $t \geq 1$. Then, since $\eta \leq \frac{(1 - \rho)\varepsilon}{C_{L, f, \theta}}$ and the projection onto $\Theta$ is a contraction (see Theorem 1.2.1 in \cite{schneider2014convex}), we see that
\[\norm{\theta_t - \theta_{t-1}} \leq \norm{\eta G_{t-1}} \leq \varepsilon.\]
Suppose $\norm{\theta_\tau - \theta_{\tau-1}} \leq \varepsilon$ holds for all $\tau \leq t$, i.e., $\theta_{0:t} \in S_{\varepsilon}(0:t)$. By \Cref{lemma:epsilon-time-varying-stability}, we see that
\[x_t \in B_n(0, R_S + C \norm{x_0}), \text{ and } u_t \in \mathcal{U}.\]
Taking norm on both sides of \eqref{thm:OCO-with-parameter-update-trajectory-distance:e0-0}, we see that
\begin{subequations}\label{thm:OCO-with-parameter-update-trajectory-distance:e1}
\begin{align}
    \norm{G_t} ={}& \norm{\sum_{\tau = 0}^{\min\{t, B-1\}} \left.\frac{\partial f_{t\mid t-\tau}}{\partial \theta_{t - \tau}}\right|_{x_{t - \tau}, \theta_{t-\tau:t}}}\nonumber\\
    \leq{}& \sum_{\tau = 0}^{\min\{t, B-1\}} \norm{\left.\frac{\partial f_{t\mid t-\tau}}{\partial \theta_{t - \tau}}\right|_{x_{t - \tau}, \theta_{t-\tau:t}}} \label{thm:OCO-with-parameter-update-trajectory-distance:e1:s1}\\
    \leq{}& \sum_{\tau = 0}^{\min\{t, B-1\}} C_{L, f, \theta} \decayfactor^\tau \label{thm:OCO-with-parameter-update-trajectory-distance:e1:s2}\\
    \leq{}& \frac{C_{L, f, \theta}}{1 - \decayfactor},\nonumber
\end{align}
\end{subequations}
where we use the triangle inequality in \eqref{thm:OCO-with-parameter-update-trajectory-distance:e1:s1} and \Cref{coro:smooth-multi-step-costs} in \eqref{thm:OCO-with-parameter-update-trajectory-distance:e1:s2}. Note that we can apply \Cref{coro:smooth-multi-step-costs} because $x_t \in B_n(0, R_S + C \norm{x_0})$. Therefore, we have shown \eqref{thm:OCO-with-parameter-update-trajectory-distance:e0} by induction. One can use the same technique as \eqref{thm:OCO-with-parameter-update-trajectory-distance:e1} to show $\norm{\nabla F_t(\theta_t)} \leq \frac{C_{L, f, \theta}}{1 - \decayfactor}$.

Since the projection onto the set $\Theta$ is a contraction, we obtain that for any $t > \tau$,
\begin{align}\label{thm:OCO-with-parameter-update-trajectory-distance:e2}
    \norm{\theta_t - \theta_\tau} \leq \frac{C_{L, f, \theta} \eta (t - \tau)}{1 - \decayfactor}.
\end{align}
Now we bound the distance between $x_\tau$ and $\hat{x}_\tau(\theta_t)$ for $\tau \leq t$. We see that
\begin{subequations}\label{thm:OCO-with-parameter-update-trajectory-distance:e3}
\begin{align}
    \norm{x_\tau - \hat{x}_\tau(\theta_t)} ={}& \norm{g_{\tau\mid 0}(x_0, \theta_{0:\tau-1}) - g_{\tau\mid 0}(x_0, (\theta_t)_{\times \tau})}\nonumber\\
    \leq{}& \sum_{\tau' = 0}^{\tau-1} \norm{g_{\tau\mid 0}(x_0, \theta_{0:\tau'}, (\theta_t)_{\times (\tau - \tau' -1)}) - g_{\tau\mid 0}(x_0, \theta_{0:\tau'-1}, (\theta_t)_{\times (\tau - \tau')})} \label{thm:OCO-with-parameter-update-trajectory-distance:e3:s1}\\
    \leq{}& \sum_{\tau' = 0}^{\tau-1}\norm{g_{\tau\mid \tau'}(x_{\tau'}, \theta_{\tau'}, (\theta_t)_{\times (\tau - \tau' -1)}) - g_{\tau\mid \tau'}(x_{\tau'}, (\theta_t)_{\times (\tau - \tau')})} \label{thm:OCO-with-parameter-update-trajectory-distance:e3:s2}\\
    \leq{}& \sum_{\tau' = 0}^{\tau-1} C_{L, g, \theta} \decayfactor^{\tau - \tau'} \norm{\theta_t - \theta_{\tau'}} \label{thm:OCO-with-parameter-update-trajectory-distance:e3:s3}\\
    \leq{}& \frac{C_{L, f, \theta} C_{L, g, \theta} \eta}{1 - \decayfactor} \sum_{\tau' = 0}^{\tau-1} (t - \tau') \decayfactor^{\tau - \tau'} \label{thm:OCO-with-parameter-update-trajectory-distance:e3:s4}\\
    \leq{}& \frac{C_{L, f, \theta} C_{L, g, \theta} \decayfactor}{(1 - \decayfactor)^2} \left((t - \tau) + \frac{1}{1 - \decayfactor}\right) \cdot \eta, \nonumber
\end{align}
\end{subequations}
where we use the triangle inequality in \eqref{thm:OCO-with-parameter-update-trajectory-distance:e3:s1};
we use the definition of multi-step dynamics in \eqref{thm:OCO-with-parameter-update-trajectory-distance:e3:s2};
we use \Cref{lemma:smooth-multi-step-dynamics} in \eqref{thm:OCO-with-parameter-update-trajectory-distance:e3:s3};
we use \eqref{thm:OCO-with-parameter-update-trajectory-distance:e2} in \eqref{thm:OCO-with-parameter-update-trajectory-distance:e3:s4}.

Similarly, since \eqref{thm:OCO-with-parameter-update-trajectory-distance:e0-0} guarantees that $x_t \in B_n(0, R_S + C \norm{x_0})$ and we also see that $\hat{x}_t(\theta_t) \in B_n(0, R_S + C \norm{x_0})$, we obtain that
\begin{subequations}\label{thm:OCO-with-parameter-update-trajectory-distance:e4}
\begin{align}
    \abs{f_t(x_t, u_t) - F_t(\theta_t)} ={}& \abs{f_t(x_t, u_t) - f_t(\hat{x}_t(\theta_t), \hat{u}_t(\theta_t))}\nonumber\\
    \leq{}& L_f\left(\norm{x_t - \hat{x}_t(\theta_t)} + \norm{u_t - \hat{u}_t(\theta_t)}\right) \label{thm:OCO-with-parameter-update-trajectory-distance:e4:s1}\\
    ={}& L_f\left(\norm{x_t - \hat{x}_t(\theta_t)} + \norm{\ALG_t(x_t, \theta_t) - \ALG_t(\hat{x}_t(\theta_t), \theta_t)}\right) \nonumber\\
    \leq{}& L_f (1 + L_{\ALG, x}) \norm{x_t - \hat{x}_t(\theta_t)} \label{thm:OCO-with-parameter-update-trajectory-distance:e4:s3}\\
    \leq{}& \frac{C_{L, f, \theta} C_{L, g, \theta}L_f (1 + L_{\ALG, x}) \decayfactor}{(1 - \decayfactor)^3} \cdot \eta, \label{thm:OCO-with-parameter-update-trajectory-distance:e4:s4}
\end{align}
\end{subequations}
where we use \Cref{assump:Lipschitz-and-smoothness} in \eqref{thm:OCO-with-parameter-update-trajectory-distance:e4:s1} and \eqref{thm:OCO-with-parameter-update-trajectory-distance:e4:s3};
we use \eqref{thm:OCO-with-parameter-update-trajectory-distance:e3} in \eqref{thm:OCO-with-parameter-update-trajectory-distance:e4:s4}.

To show the last inequality in \Cref{thm:OCO-with-parameter-update-trajectory-distance}, note that we have
\begin{subequations}\label{thm:OCO-with-parameter-update-trajectory-distance:e5-0}
\begin{align}
    \norm{\tilde{x}_t - \hat{x}_t(\tilde{\theta}_t)} &\leq \sum_{\tau = 0}^{t-1} \norm{g_{t\mid 0}\left(x_0, \tilde{\theta}_{0:\tau-1}, (\tilde{\theta}_t)_{\times (t-\tau)}\right) - g_{t\mid 0}\left(x_0, \tilde{\theta}_{0:\tau}, (\tilde{\theta}_t)_{\times (t-\tau-1)}\right)}\label{thm:OCO-with-parameter-update-trajectory-distance:e5-0:s0-0}\\
    &\leq C L_{\ALG, \theta} L_{g, x} \sum_{\tau = 0}^{t-1}\rho^{t - \tau - 1}\norm{\tilde{\theta}_t - \tilde{\theta}_{\tau}}\label{thm:OCO-with-parameter-update-trajectory-distance:e5-0:s0-1}\\
    &\leq C L_{\ALG, \theta} L_{g, x} \sum_{\tau = 0}^{t-1}\rho^{t - \tau - 1} \sum_{\tau' = \tau}^{t-1} \norm{\tilde{\theta}_{\tau'+1} - \tilde{\theta}_{\tau'}}\label{thm:OCO-with-parameter-update-trajectory-distance:e5-0:s1}\\
    &\leq \frac{C L_{\ALG, \theta} L_{g, x}}{1 - \rho}\sum_{\tau = 0}^{t-1} \rho^{t - \tau - 1} \norm{\tilde{\theta}_{\tau + 1} - \tilde{\theta}_\tau},\label{thm:OCO-with-parameter-update-trajectory-distance:e5-0:s2}
\end{align}
\end{subequations}
where we use the triangle inequality in \eqref{thm:OCO-with-parameter-update-trajectory-distance:e5-0:s0-0} and \eqref{thm:OCO-with-parameter-update-trajectory-distance:e5-0:s1}; we use the assumption that $\norm{\tilde{x_t}} \leq \min\{R_C, R_x\}$ and the time-invariant contractive perturbation property in \eqref{thm:OCO-with-parameter-update-trajectory-distance:e5-0:s0-1}; we rearrange the terms and use $\sum_{\tau=0}^\infty \rho^\tau \leq \frac{1}{1 - \rho}$ in \eqref{thm:OCO-with-parameter-update-trajectory-distance:e5-0:s2}.

Therefore, since $\tilde{x}_t, \hat{x}_t(\tilde{\theta}_t) \in \mathcal{X}$ and $\tilde{u}_t, \hat{u}_t(\tilde{\theta}_t) \in \mathcal{U}$, we see that
\begin{subequations}\label{thm:OCO-with-parameter-update-trajectory-distance:e5}
\begin{align}
    \abs{f_t(\tilde{x}_t, \tilde{u}_t) - F_t(\tilde{\theta}_t)} ={}& \abs{f_t(\tilde{x}_t, \tilde{u}_t) - f_t(\hat{x}_t(\tilde{\theta}_t), \hat{u}_t(\tilde{\theta}_t))}\nonumber\\
    \leq{}& L_f\left(\norm{\tilde{x}_t - \hat{x}_t(\tilde{\theta}_t)} + \norm{\tilde{u}_t - \hat{u}_t(\tilde{\theta}_t)}\right) \label{thm:OCO-with-parameter-update-trajectory-distance:e5:s1}\\
    ={}& L_f\left(\norm{\tilde{x}_t - \hat{x}_t(\tilde{\theta}_t)} + \norm{\ALG_t(\tilde{x}_t, \tilde{\theta}_t) - \ALG_t(\hat{x}_t(\tilde{\theta}_t), \tilde{\theta}_t)}\right) \nonumber\\
    \leq{}& L_f (1 + L_{\ALG, x}) \norm{\tilde{x}_t - \hat{x}_t(\tilde{\theta}_t)} \label{thm:OCO-with-parameter-update-trajectory-distance:e5:s3}\\
    \leq{}& \frac{C L_{\ALG, \theta} L_{g, x} L_f (1 + L_{\ALG, x})}{1 - \rho}\sum_{\tau = 0}^{t-1} \rho^{t - \tau - 1} \norm{\tilde{\theta}_{\tau + 1} - \tilde{\theta}_\tau}, \label{thm:OCO-with-parameter-update-trajectory-distance:e5:s4}
\end{align}
\end{subequations}
where we use \Cref{assump:Lipschitz-and-smoothness} in \eqref{thm:OCO-with-parameter-update-trajectory-distance:e5:s1} and \eqref{thm:OCO-with-parameter-update-trajectory-distance:e5:s3}; we use \eqref{thm:OCO-with-parameter-update-trajectory-distance:e5-0} in \eqref{thm:OCO-with-parameter-update-trajectory-distance:e5:s4}.

\subsection{Proof of Theorem \ref{thm:OCO-with-parameter-update-gradient-bias}}\label{appendix:OCO-with-parameter-update-gradient-bias}
\begin{theorem}[Gradient Bias]\label{thm:OCO-with-parameter-update-gradient-bias:full}
Suppose Assumptions \ref{assump:Lipschitz-and-smoothness} and \ref{assump:contractive-and-stability} hold. Let $\{x_t, u_t, \theta_t\}_{t \in \mathcal{T}}$ denote the trajectory of GAPS (Algorithm \ref{alg:OCO-with-parameter-update}) with buffer size $B$ and learning rate $\eta \leq \frac{(1 - \rho)\varepsilon}{C_{L, f, \theta}}$. Then, the following holds for all $\tau \leq t$:
\begin{align*}
    \norm{\left.\frac{\partial f_{t\mid 0}}{\partial \theta_\tau} \right|_{x_0, \theta_{0:t}} - \left.\frac{\partial f_{t\mid 0}}{\partial \theta_\tau} \right|_{x_0, (\theta_t)_{\times (t+1)}}} \leq \left(\hat{C}_0 + \hat{C}_1 (t - \tau) + \hat{C}_2 (t - \tau)^2\right) \decayfactor^{t - \tau} \cdot \eta,
\end{align*}
for
\begin{align*}
    &\hat{C}_0 = \frac{\decayfactor C_{L, f, \theta} C_{L, g, \theta} C_{\ell, f, (\theta, x)}}{(1 - \decayfactor)^3}, \hat{C}_1 = \frac{(1 - \decayfactor)C_{L, f, \theta} C_{\ell, f, (\theta, x)} + \decayfactor C_{L, f, \theta} C_{L, g, \theta} C_{\ell, f, (\theta, \theta)}}{(1 - \decayfactor)^2},\\
    &\hat{C}_2 = \frac{C_{\ell, f, (x, \theta) C_{L, g, \theta} C_{L, f, \theta}}}{1 - \decayfactor}.
\end{align*}
Next,
\[\norm{G_t - \nabla F_t(\theta_t)} \leq \left(\frac{\hat{C}_0}{1 - \decayfactor} + \frac{\hat{C}_1 + \hat{C}_2}{(1 - \decayfactor)^{2}} + \frac{\hat{C}_2}{(1 - \decayfactor)^{3}}\right) \eta + \frac{C_{L, f, \theta}}{1 - \decayfactor} \cdot \decayfactor^B.\]
\end{theorem}
\begin{proof}[Proof of \Cref{thm:OCO-with-parameter-update-gradient-bias:full}]
To simplify the notation, we adopt the shorthand notations $\hat{x}_\tau(\theta) \coloneqq g_{\tau\mid 0}(x_0, \theta_{\times \tau})$ and $\hat{u}_\tau(\theta) \coloneqq \ALG_\tau(\hat{x}_\tau(\theta), \theta)$ throughout the proof.

As we discussed below \Cref{thm:OCO-with-parameter-update-gradient-bias} in the proof outline, we use the triangle inequality to do the decomposition
\begin{align}\label{thm:OCO-with-parameter-update-gradient-bias:e1}
    &\norm{\left.\frac{\partial f_{t\mid 0}}{\partial \theta_\tau} \right|_{x_0, \theta_{0:t}} - \left.\frac{\partial f_{t\mid 0}}{\partial \theta_\tau} \right|_{x_0, (\theta_t)_{\times (t + 1)}}}\nonumber\\
    ={}&\norm{\left.\frac{\partial f_{t\mid \tau}}{\partial \theta_\tau} \right|_{x_\tau, \theta_{\tau:t}} - \left.\frac{\partial f_{t\mid \tau}}{\partial \theta_\tau} \right|_{\hat{x}_\tau(\theta_t), (\theta_t)_{\times (t - \tau + 1)}}} \nonumber\\
    \leq{}& \norm{\left.\frac{\partial f_{t\mid \tau}}{\partial \theta_\tau} \right|_{x_\tau, \theta_\tau, (\theta_t)_{\times (t - \tau)}} - \left.\frac{\partial f_{t\mid \tau}}{\partial \theta_\tau} \right|_{\hat{x}_\tau(\theta_t), (\theta_t)_{\times (t - \tau + 1)}}}\nonumber\\
    &+ \sum_{\tau' = \tau + 1}^{t-1} \norm{\left.\frac{\partial f_{t\mid \tau}}{\partial \theta_\tau} \right|_{x_\tau, \theta_{\tau:\tau'}, (\theta_t)_{\times (t - \tau')}} - \left.\frac{\partial f_{t\mid \tau}}{\partial \theta_\tau} \right|_{x_\tau, \theta_{\tau:\tau'-1}, (\theta_t)_{\times (t - \tau' + 1)}}}.
\end{align}

Note that we can apply \Cref{coro:smooth-multi-step-costs} to bound each term in \eqref{thm:OCO-with-parameter-update-gradient-bias:e1}. For the first term in \eqref{thm:OCO-with-parameter-update-gradient-bias:e1}, since $x_\tau, \hat{x}_\tau(\theta_t), x_{\tau+1} \in B_n(0, R_S + C\norm{x_0})$, we see that
\begin{subequations}\label{thm:OCO-with-parameter-update-gradient-bias:e2}
\begin{align}
    &\norm{\left.\frac{\partial f_{t\mid \tau}}{\partial \theta_\tau} \right|_{x_\tau, \theta_\tau, (\theta_t)_{\times (t - \tau)}} - \left.\frac{\partial f_{t\mid \tau}}{\partial \theta_\tau} \right|_{\hat{x}_\tau(\theta_t), (\theta_t)_{\times (t - \tau + 1)}}}\nonumber\\
    \leq{}& \decayfactor^{t - \tau} \left(C_{\ell, f, (\theta, x)} \norm{x_\tau - \hat{x}_\tau(\theta_t)} + C_{\ell, f, (\theta, \theta)} \norm{\theta_t - \theta_\tau}\right) \label{thm:OCO-with-parameter-update-gradient-bias:e2:s1}\\
    \leq{}& \frac{(1 - \decayfactor)C_{L, f, \theta} C_{\ell, f, (\theta, x)} + \decayfactor C_{L, f, \theta} C_{L, g, \theta} C_{\ell, f, (\theta, \theta)}}{(1 - \decayfactor)^2} \cdot (t - \tau) \decayfactor^{t - \tau} \cdot \eta\nonumber\\
    &+ \frac{\decayfactor C_{L, f, \theta} C_{L, g, \theta} C_{\ell, f, (\theta, x)}}{(1 - \decayfactor)^3}\cdot \decayfactor^{t - \tau} \cdot \eta, \label{thm:OCO-with-parameter-update-gradient-bias:e2:s2}
\end{align}
\end{subequations}
where we use \Cref{coro:smooth-multi-step-costs} in \eqref{thm:OCO-with-parameter-update-gradient-bias:e2:s1} and \Cref{thm:OCO-with-parameter-update-trajectory-distance} in \eqref{thm:OCO-with-parameter-update-gradient-bias:e2:s2}.

For any $\tau' \in [\tau+1: t-1]$, since $x_{\tau'}, x_{\tau'+1} \in B_n(0, R_S + C\norm{x_0})$, we see that
\begin{subequations}\label{thm:OCO-with-parameter-update-gradient-bias:e3}
\begin{align}
    &\norm{\left.\frac{\partial f_{t\mid \tau}}{\partial \theta_\tau} \right|_{x_\tau, \theta_{\tau:\tau'}, (\theta_t)_{\times (t - \tau')}} - \left.\frac{\partial f_{t\mid \tau}}{\partial \theta_\tau} \right|_{x_\tau, \theta_{\tau:\tau'-1}, (\theta_t)_{\times (t - \tau' + 1)}}}\nonumber\\
    ={}& \norm{\left(\left.\frac{\partial f_{t\mid \tau'}}{\partial x_{\tau'}} \right|_{x_{\tau'}, \theta_{\tau'}, (\theta_t)_{\times (t - \tau')}} - \left.\frac{\partial f_{t\mid \tau'}}{\partial x_{\tau'}} \right|_{x_{\tau'}, (\theta_t)_{\times (t - \tau' + 1)}}\right)\left.\frac{\partial g_{\tau'\mid \tau}}{\partial \theta_\tau}\right|_{x_\tau, \theta_{\tau:\tau'-1}}}\nonumber\\
    \leq{}& \norm{\left.\frac{\partial f_{t\mid \tau'}}{\partial x_{\tau'}} \right|_{x_{\tau'}, \theta_{\tau'}, (\theta_t)_{\times (t - \tau')}} - \left.\frac{\partial f_{t\mid \tau'}}{\partial x_{\tau'}} \right|_{x_{\tau'}, (\theta_t)_{\times (t - \tau' + 1)}}} \cdot \norm{\left.\frac{\partial g_{\tau'\mid \tau}}{\partial \theta_\tau}\right|_{x_\tau, \theta_{\tau:\tau'-1}}}\nonumber\\
    \leq{}& C_{\ell, f, (x, \theta)} \decayfactor^{t - \tau'} \norm{\theta_t - \theta_{\tau'}} \cdot C_{L, g, \theta} \decayfactor^{\tau' - \tau} \label{thm:OCO-with-parameter-update-gradient-bias:e3:s1}\\
    \leq{}& \frac{C_{\ell, f, (x, \theta) C_{L, g, \theta} C_{L, f, \theta}}}{1 - \decayfactor} \cdot (t - \tau)\decayfactor^{t - \tau} \cdot \eta, \label{thm:OCO-with-parameter-update-gradient-bias:e3:s2}
\end{align}
\end{subequations}
where we use \Cref{lemma:smooth-multi-step-dynamics} and \Cref{coro:smooth-multi-step-costs} in \eqref{thm:OCO-with-parameter-update-gradient-bias:e3:s1}; we use \Cref{thm:OCO-with-parameter-update-trajectory-distance} in \eqref{thm:OCO-with-parameter-update-gradient-bias:e3:s2}. Substituting \eqref{thm:OCO-with-parameter-update-gradient-bias:e2} and \eqref{thm:OCO-with-parameter-update-gradient-bias:e3} into \eqref{thm:OCO-with-parameter-update-gradient-bias:e1} finishes the proof of the first inequality.

For the second inequality, recall that $G_t$ and $\nabla \ell_t(\theta_t)$ are given by
\[G_t \coloneqq \sum_{\tau = 0}^{\min\{t, B-1\}} \left.\frac{\partial f_{t\mid 0}}{\partial \theta_{t - \tau}}\right|_{x_0, \theta_{0:t}}, \nabla \ell_t(\theta_t) = \sum_{\tau=0}^t \left.\frac{\partial f_{t\mid 0}}{\partial \theta_{t - \tau}}\right|_{x_0, (\theta_t)_{\times (t+1)}}.\]
Therefore, we see that
\begin{subequations}\label{thm:OCO-with-parameter-update-gradient-bias:e4}
\begin{align}
    \norm{G_t - \nabla F_t(\theta_t)}
    ={}& \norm{\sum_{\tau = 0}^{\min\{t, B-1\}} \left.\frac{\partial f_{t\mid 0}}{\partial \theta_{t - \tau}}\right|_{x_0, \theta_{0:t}} - \sum_{\tau=0}^t \left.\frac{\partial f_{t\mid 0}}{\partial \theta_{t - \tau}}\right|_{x_0, (\theta_t)_{\times (t+1)}}}\nonumber\\
    \leq{}& \sum_{\tau = 0}^{\min\{t, B - 1\}} \norm{\left.\frac{\partial f_{t\mid 0}}{\partial \theta_{t - \tau}}\right|_{x_0, \theta_{0:t}} - \left.\frac{\partial f_{t\mid 0}}{\partial \theta_{t - \tau}}\right|_{x_0, (\theta_t)_{\times (t+1)}}}\nonumber\\
    &+ \mathbf{1}(t \geq B) \sum_{\tau = B}^t \norm{\left.\frac{\partial f_{t\mid t - \tau}}{\partial \theta_{t - \tau}}\right|_{\hat{x}_{t-\tau}(\theta_t), (\theta_t)_{\times (\tau + 1)}}}\label{thm:OCO-with-parameter-update-gradient-bias:e4:s1}\\
    \leq{}& \sum_{\tau = 0}^{B - 1} \left(\hat{C}_0 + \hat{C}_1 \tau + \hat{C}_2 \tau^2\right) \decayfactor^\tau \eta + \sum_{\tau = B}^t C_{L, f, \theta} \decayfactor^\tau\label{thm:OCO-with-parameter-update-gradient-bias:e4:s2}\\
    \leq{}& \left(\frac{\hat{C}_0}{1 - \decayfactor} + \frac{\hat{C}_1 + \hat{C}_2}{(1 - \decayfactor)^2} + \frac{\hat{C}_2}{(1 - \decayfactor)^3}\right) \eta + \frac{C_{L, f, \theta}}{1 - \decayfactor} \cdot \decayfactor^B, \nonumber
\end{align}
\end{subequations}
where we use the triangle inequality in \eqref{thm:OCO-with-parameter-update-gradient-bias:e4:s1}; we use the first inequality in \Cref{thm:OCO-with-parameter-update-gradient-bias} that we have shown and \Cref{coro:smooth-multi-step-costs} in \eqref{thm:OCO-with-parameter-update-gradient-bias:e4:s2}.
\end{proof}

\section{Dynamic regret of GAPS}\label{appendix:gaps-dynamic-regret}
When the surrogate cost $F_t$ is convex, dynamic policy regret is another benchmark designed for time-varying environments \cite{hazan2016introduction, zhou2023efficient}. To define the dynamic regret, we first need to define the optimal total cost under the parameter path length constraint (denoted by $\hat{J}_{\mathcal{T}}(\Upsilon)$):
\begin{align*}
\min_{\tilde{\theta}_{0:T-1} \in \Theta^T} J_{\mathcal{T}}(\tilde{\theta}_{0:T-1})
\text{ s.t. } \sum_{\tau = 1}^{T-1}\norm{\tilde{\theta}_\tau - \tilde{\theta}_{\tau-1}} \leq \Upsilon-1,\text{ and }\norm{\tilde{x}_t} \leq \min\{R_C, R_x\}, \forall t \in \mathcal{T}, 
\end{align*}
where $\tilde{x}_t \coloneqq g_{t\mid 0}(x_0, \tilde{\theta}_{0:t-1})$ and $\Upsilon-1 \geq 0$ is the upper bound on the total path length of the parameter sequence. The dynamic regret is defined as
\begin{align}\label{equ:dynamic-policy-regret}
    R_{\Upsilon}^D(T) \coloneqq J_{\mathcal{T}}(\theta_{0:T-1}) - \hat{J}_{\mathcal{T}}(\Upsilon).
\end{align}
Compared with the (static) policy regret $R^S(T)$ and the adaptive policy regret $R^A(T)$, the dynamic policy regret $R^D(T)$ is more flexible by allowing the offline benchmark to change policy parameters over time.

To bound the dynamic regret of GAPS, we follow a similar outline as the proof of the adaptive regret (\Cref{appendix:gaps-outline}): We first show a dynamic regret of the Ideal OGD algorithm that is robust to gradient biases.

\begin{theorem}\label{thm:OGD-with-biased-gradient-dynamic-regret}
Under the same assumptions as \Cref{thm:OGD-with-biased-gradient-regret}, for any sequence $\tilde{\theta}_{0:T-1}\in \Theta^T$, we have
{
\begin{align*}
\sum_{t=0}^{T-1}F_t(\theta_t) - \sum_{t=0}^{T-1}F_t(\tilde{\theta}_t) \leq \frac{3 D^2}{\eta} + \frac{2D}{\eta}\sum_{t=1}^{T-1} \norm{\tilde{\theta}_t - \tilde{\theta}_{t-1}} + \eta (W + \alpha)^2 T + 2 \alpha D T.
\end{align*}
}
\end{theorem}

The proof of \Cref{thm:OGD-with-biased-gradient-dynamic-regret} is deferred to \Cref{appendix:OGD-dynamic-regret}.

Recall that we have estabilished the similarity between GAPS and Ideal OGD in \Cref{thm:bridge-GAPS-and-OGD}. Using these bounds on costs/gradients differences, we can ``transfer'' the dynamic regret guarantee on OGD (for online optimization) to GAPS (for online policy selection). We formally state this result \Cref{thm:main-regret-bound-convex:dynamic} and defer the proof to \Cref{appendix:GAPS-dynamic-regret}.

\begin{theorem}\label{thm:main-regret-bound-convex:dynamic}
Under the same assumptions as \Cref{thm:main-regret-bound-convex}, GAPS achieves the dynamic regret bound
{
\begin{align}
    R_{\Upsilon}^D(T) ={}& O\bigg((\eta T)^{-1} \Upsilon + (1 - \decayfactor)^{-5}{\eta T} + (1 - \decayfactor)^{-1}{\decayfactor^B T} + (1 - \decayfactor)^{-10}{\eta^3 T} + (1 - \decayfactor)^{-2}{\decayfactor^{2B}\eta}\bigg). \label{thm:main-regret-bound-convex:e03}
\end{align}
}
(See \Cref{appendix:GAPS-dynamic-regret} for the detailed expressions.)
\end{theorem}

By choosing the appropriate learning rate $\eta$, one can achieve sublinear dynamic regret when the parameter path length constraint $\Upsilon$ of the benchmark is $o(T)$. We discuss how to set the learning rate and the dynamic regret bound it achieves in \Cref{coro:main-regret-bounds-simplified:dynamic}

\begin{corollary}\label{coro:main-regret-bounds-simplified:dynamic}
Under the same assumptions as \Cref{thm:main-regret-bound-convex}, suppose $T \gg \frac{1}{1 - \decayfactor}$ and
$B \geq \frac{1}{2} \log(T) /\log(1/\decayfactor)$: If we set $\eta = (1 - \rho)^{\frac{5}{2}} (\Upsilon/T)^{\frac{1}{2}}$, we get the dynamic regret of $R_{\Upsilon}^D(T) = O\left((1-\rho)^{-\frac{5}{2}}(\Upsilon T)^{\frac{1}{2}}\right)$.
\end{corollary}

It is worth noticing that the approach to derive the dynamic regret for GAPS can generalize to other performance metrics: If we can derive some performance bounds for OGD in online optimization, we can use \Cref{thm:bridge-GAPS-and-OGD} to show a corresponding guarantee for GAPS in online policy selection. An interesting future direction is using this approach to derive results for GAPS under nonconvex surrogate costs $F_t$.

\subsection{Proof of Theorem \ref{thm:OGD-with-biased-gradient-dynamic-regret}}\label{appendix:OGD-dynamic-regret}
This proof is inspired by the proof of Theorem 10.1 in \cite{hazan2016introduction}. 

Without loss of generality, we can assume $\theta_0 = 0$, because otherwise we change the variable by subtracting all other parameters by $\theta_0$. We see that
\begin{align*}
    \norm{\theta_{t+1} - \tilde{\theta}_t}^2 \leq{}& \norm{\theta_{t+1}' - \tilde{\theta}_t}^2\\
    ={}& \norm{\theta_t - \eta G_t - \tilde{\theta}_t}^2\\
    ={}& \norm{\theta_t - \tilde{\theta}_t}^2 + \eta^2 \norm{G_t}^2 - 2\eta \langle G_t, \theta_t - \tilde{\theta}_t\rangle.
\end{align*}
Thus, we see that
\begin{align*}
    &2\langle \nabla F_t(\theta_t), \theta_t - \tilde{\theta}_t\rangle\\
    \leq{}& 2\langle \nabla F_t(\theta_t) - G_t, \theta_t - \tilde{\theta}_t\rangle + \frac{1}{\eta}\left(\norm{\theta_t - \tilde{\theta}_t}^2 - \norm{\theta_{t+1} - \tilde{\theta}_t}^2\right) + \eta (W + \alpha)^2\\
    \leq{}& \frac{1}{\eta}\left(\norm{\theta_t - \tilde{\theta}_t}^2 - \norm{\theta_{t+1} - \tilde{\theta}_t}^2\right) + \eta (W + \alpha)^2 + 2 \alpha D.
\end{align*}
Therefore, we obtain that
\begin{align*}
    &2\left(\sum_{t=0}^{T-1} F_t(\theta_t) - F_t(\tilde{\theta}_t)\right)\\
    \leq{}& 2 \sum_{t=0}^{T-1} \langle \nabla F_t(\theta_t), \theta_t - \tilde{\theta}_t\rangle\\
    \leq{}& \frac{1}{\eta} \sum_{t=0}^{T-1} \left(\norm{\theta_t - \tilde{\theta}_t}^2 - \norm{\theta_{t+1} - \tilde{\theta}_t}^2\right) + \eta (W + \alpha)^2 T + 2 \alpha D T\\
    \leq{}& \frac{1}{\eta} \sum_{t=0}^{T-1} \left(\norm{\theta_t}^2 - \norm{\theta_{t+1}}^2 + 2 \langle \tilde{\theta}_t, \theta_{t+1} - \theta_t\rangle\right) + \eta (W + \alpha)^2 T + 2 \alpha D T\\
    \leq{}& \frac{1}{\eta}\left(D^2 + 2\sum_{t=1}^{T-1} \langle \theta_t, \tilde{\theta}_{t-1} - \tilde{\theta}_t\rangle + 2 D^2\right) + \eta (W + \alpha)^2 T + 2 \alpha D T\\
    \leq{}& \frac{3 D^2}{\eta} + \frac{2D}{\eta}\sum_{t=1}^{T-1} \norm{\tilde{\theta}_t - \tilde{\theta}_{t-1}} + \eta (W + \alpha)^2 T + 2 \alpha D T.
\end{align*}

\subsection{Complete Version of Theorem \ref{thm:main-regret-bound-convex:dynamic} and Its Proof}\label{appendix:GAPS-dynamic-regret}
\begin{theorem}\label{thm:main-regret-bound-convex:dynamic:detailed}
Under the same assumptions as \Cref{thm:main-regret-bound-convex}, GAPS achieves the dynamic regret bound
{
\begin{align*}
    R_{\Upsilon}^D(T) \leq{}& 2\left(\frac{C_{L, f, \theta}^2}{(1 - \rho)^3} + \left(\frac{\hat{C}_0}{1 - \rho} + \frac{\hat{C}_1 + \hat{C}_2}{(1 - \rho)^2} + \frac{\hat{C}_2}{(1 - \rho)^3}\right)D\right)\eta T + \frac{3 D^2}{\eta} + \frac{2 D C_{L, f,\theta}}{1 - \rho} \cdot \rho^B T\\
    &+4\left(\frac{\hat{C}_0}{1 - \rho} + \frac{\hat{C}_1 + \hat{C}_2}{(1 - \rho)^2} + \frac{\hat{C}_2}{(1 - \rho)^3}\right)^2 D^2 \eta^3 T + \frac{4 C_{L, f, \theta}^2}{(1 - \rho)^2}\rho^{2B}\eta T\\
    &+ \left(\frac{2D}{\eta} + \frac{C_{L, f, \theta}}{(1 - \rho)^2}\right)\sum_{t=1}^{T-1} \norm{\tilde{\theta}_t - \tilde{\theta}_{t-1}}.
\end{align*}
}
\end{theorem}
\begin{proof}[Proof of \Cref{thm:main-regret-bound-convex:dynamic:detailed}]
By \Cref{thm:OGD-with-biased-gradient-regret} and \Cref{thm:OCO-with-parameter-update-gradient-bias}, we see that the sequence of policy parameters of the online policy satisfies that
\begin{align}\label{thm:main-regret-bound-convex:e3}
    \sum_{t=0}^{T-1}F_t(\theta_t) - \sum_{t=0}^{T-1}F_t(\tilde{\theta}_t)\leq \frac{3 D^2}{\eta} + \frac{2D}{\eta}\sum_{t=1}^{T-1} \norm{\tilde{\theta}_t - \tilde{\theta}_{t-1}} + \eta (W + \alpha)^2 T + 2 \alpha D T,
\end{align}
where $W = \frac{C_{L, f,\theta}}{1 - \rho}$ by \Cref{thm:OCO-with-parameter-update-trajectory-distance} and $\alpha = \left(\frac{\hat{C}_0}{1 - \rho} + \frac{\hat{C}_1 + \hat{C}_2}{(1 - \rho)^2} + \frac{\hat{C}_2}{(1 - \rho)^3}\right) \eta + \frac{C_{L, f, \theta}}{1 - \rho} \cdot \rho^B$.

Recall that by \Cref{thm:OCO-with-parameter-update-trajectory-distance}, we have that
\begin{align*}
    \abs{f_t(x_t, u_t) - F_t(\theta_t)} \leq{}& \frac{C_{L, f, \theta}^2 \rho}{(1 - \rho)^3}\cdot \eta,\\
    \abs{f_t(\tilde{x}_t, \tilde{u}_t) - F_t(\tilde{\theta}_t)} \leq{}& \frac{C_{L, f, \theta}}{1 - \rho} \sum_{\tau = 0}^{t-1} \rho^\tau \norm{\tilde{\theta}_{t-\tau} - \tilde{\theta}_{t-\tau-1}}.
\end{align*}
Substituting these into \eqref{thm:main-regret-bound-convex:e3} gives that
\begin{align*}
    &f_t(x_t, u_t) - f_t(\tilde{x}_t, \tilde{u}_t)\\
    \leq{}& \left(\sum_{t=0}^{T-1}F_t(\theta_t) - \sum_{t=0}^{T-1}F_t(\tilde{\theta}_t)\right) + \sum_{t=0}^{T-1} \abs{f_t(x_t, u_t) - F_t(\theta_t)} + \sum_{t=0}^{T-1} \abs{f_t(\tilde{x}_t, \tilde{u}_t) - F_t(\tilde{\theta}_t)}\\
    \leq{}& \frac{3 D^2}{\eta} + \frac{2D}{\eta}\sum_{t=1}^{T-1} \norm{\tilde{\theta}_t - \tilde{\theta}_{t-1}} + \eta (W + \alpha)^2 T + 2 \alpha D T + \frac{C_{L, f, \theta}^2 \rho}{(1 - \rho)^3}\cdot \eta T\\
    &+ \frac{C_{L, f, \theta}}{1 - \rho} \sum_{t = 0}^{T-1} \sum_{\tau = 0}^{t-1} \rho^\tau \norm{\tilde{\theta}_{t-\tau} - \tilde{\theta}_{t-\tau-1}}\\
    \leq{}& 2\left(\frac{C_{L, f, \theta}^2}{(1 - \rho)^3} + \left(\frac{\hat{C}_0}{1 - \rho} + \frac{\hat{C}_1 + \hat{C}_2}{(1 - \rho)^2} + \frac{\hat{C}_2}{(1 - \rho)^3}\right)D\right)\eta T + \frac{3 D^2}{\eta} + \frac{2 D C_{L, f,\theta}}{1 - \rho} \cdot \rho^B T\\
    &+4\left(\frac{\hat{C}_0}{1 - \rho} + \frac{\hat{C}_1 + \hat{C}_2}{(1 - \rho)^2} + \frac{\hat{C}_2}{(1 - \rho)^3}\right)^2 D^2 \eta^3 T + \frac{4 C_{L, f, \theta}^2}{(1 - \rho)^2}\rho^{2B}\eta T\\
    &+ \left(\frac{2D}{\eta} + \frac{C_{L, f, \theta}}{(1 - \rho)^2}\right)\sum_{t=1}^{T-1} \norm{\tilde{\theta}_t - \tilde{\theta}_{t-1}}.
\end{align*}
\end{proof}

\section{Regret of GAPS under nonconvex surrogate costs}\label{appendix:gaps-nonconvex}
To derive the local regret bound for GAPS in online policy selection, we first bound the local regret for OGD (with biased gradients) in online nonconvex optimization.

\begin{theorem}\label{thm:nonconvex-biased-OGD}
Suppose $\Theta = \mathbb{R}^d$. Consider the biased OGD update rule $\theta_{t+1} = \theta_t - \eta G_t$, where $G_t$ satisfies $\norm{G_t - \nabla F_t(\theta_t)} \leq \beta$. Suppose at every time $t$, $F_t$ is $L_F$-Lipschitz and $\ell_F$-smooth. If the learning rate $\eta < \frac{1}{\ell_F}$, we have that
\[\sum_{t = 0}^{T-1}\norm{\nabla F_t(\theta_t)}^2 \leq \frac{2}{\eta}\left(F_0(\theta_0) + L_F \sum_{t=1}^{T-1}\dist_s(F_t, F_{t-1})\right) + \left(2(1 - \ell_F \eta) L_F \beta + \ell_F \eta \cdot \beta^2\right)T,\]
where $\dist_s(F, F') \coloneqq \sup_{\theta \in \Theta} \abs{F(\theta) - F'(\theta)}$.
\end{theorem}

The proof of \Cref{thm:nonconvex-biased-OGD} is deferred to \Cref{appendix:thm:nonconvex-biased-OGD:proof}. Our proof is inspired by the analysis for (stochastic) gradient descent in offline nonconvex optimization (see \citep[e.g.][]{ghadimi2013stochastic}) with the additional step to handle the time-varying function sequence $F_{0:T-1}$ via the measure of variation $\dist_s(F_t, F_{t-1})$.

Our approximation error bound (\Cref{thm:bridge-GAPS-and-OGD}) guarantees that the bias $\beta = O(1/\sqrt{T})$ if we set the learning rate $\eta = O(1/\sqrt{T})$. Therefore, the remaining task is to bound the measure of variation in online nonconvex optimization $\sum_{t=1}^{T-1} \dist_s(F_t, F_{t-1})$ by the variation intensity in online policy selection $V$ (\Cref{def:total-variation-policy-selection}). To derive this bound, we need to show a convergence result on applying a fixed policy parameter in a time-invariant system. We begin with a definition that characterize this (imaginary) dynamical process. 

\begin{definition}\label{def:imaginary-state}
For fixed dynamics function $g$, policy function $\pi$, and policy parameter $\theta$, we define $\tilde{x}_\tau^{(g, \pi)}(\theta)$ recursively by the equation
\[\tilde{x}_{\tau+1}^{(g, \pi)}(\theta) = g\left(\tilde{x}_{\tau}^{(g, \pi)}(\theta), \pi\left(\tilde{x}_{\tau}^{(g, \pi)}(\theta), \theta\right)\right), \forall \tau \geq 0, \text{ where }\tilde{x}_{0}^{(g, \pi)}(\theta) = x_0.\]
\end{definition}

Compared with $\hat{x}_t(\theta)$ we defined before \Cref{def:surrogate-cost}, the state $\tilde{x}_{\tau+1}^{(g, \pi)}(\theta)$ is produced by a time-invariant dynamical system induced by $g$ and $\pi$, while $\hat{x}_t(\theta)$ is produced by the actual time-varying dynamics induced by $g_{0:t-1}$ and $\pi_{0:t-1}$.

We show the time-invariant evolution $\tilde{x}_\tau^{(g, \pi)}$ has a unique limitation point as $\tau$ tends to infinity. This limit is also a fixed point, and the states will converge to the limit exponentially fast with respect to $\tau$. We state this result formally in \Cref{lemma:time-invariant-dynamics-fixed-point}.

\begin{lemma}\label{lemma:time-invariant-dynamics-fixed-point}
Suppose Assumptions \ref{assump:Lipschitz-and-smoothness} and \ref{assump:contractive-and-stability} hold, and $(g, \pi) \in \mathcal{G}$. The limit $\lim_{\tau \to \infty} \tilde{x}_{\tau}^{(g, \pi)}(\theta)$ exists. Let $\tilde{x}_{\infty}^{(g, \pi)}(\theta) \coloneqq \lim_{\tau \to \infty} \tilde{x}_{\tau}^{(g, \pi)}(\theta)$. Further, we also have that
\[\norm{\tilde{x}_{\tau}^{(g, \pi)}(\theta) - \tilde{x}_{\infty}^{(g, \pi)}(\theta)} \leq C \rho^\tau \norm{x_0 - \tilde{x}_{\infty}^{(g, \pi)}(\theta)} \leq C \rho^\tau \cdot \diam(\mathcal{X}),\]
where $\diam(\mathcal{X}) = 2C(R_S + C\norm{x_0}) + 2R_S$.
\end{lemma}

The proof of \Cref{lemma:time-invariant-dynamics-fixed-point} is deferred to \Cref{appendix:lemma:time-invariant-dynamics-fixed-point:proof}. With the fixed point and convergence result in \Cref{lemma:time-invariant-dynamics-fixed-point}, we bound the measure of variation based on $F_{0:T-1}$ by the variation intensity $V$ based on $g_{0:T-1}, \pi_{0:T-1},$ and $f_{0:T-1}$ in \Cref{lemma:path-length-bound}.

\begin{lemma}\label{lemma:path-length-bound}
Suppose Assumptions \ref{assump:Lipschitz-and-smoothness} and \ref{assump:contractive-and-stability} hold. Then, we have
\[\sum_{t=1}^{T-1} \dist_s(F_t, F_{t-1}) \leq \frac{2C L_f (1 + L_{\pi, x}) (1 + L_{g, u})}{(1 - \rho)^2 \rho} \cdot V + \frac{2 C L_f (1 + L_{\pi, x})}{1 - \rho} \cdot \diam(\mathcal{X}),\]
where $\diam(\mathcal{X}) = 2C(R_S + C\norm{x_0}) + 2R_S$.
\end{lemma}

With these auxiliary results, we restate \Cref{thm:GAPS-nonconvex-regret} with complete expressions and present the proof.

\begin{theorem}\label{thm:GAPS-nonconvex-regret:detailed}
Under the same assumptions as \Cref{thm:bridge-GAPS-and-OGD}, if we additionally assume that $\Theta = \mathbb{R}^d$ for some integer $d$, then GAPS satisfies local regret
\begin{align}
    R^L(T) \leq{}& \frac{2}{\eta}\left(c_0 + \frac{2C L_F L_f (1 + L_{\pi, x})}{1 - \rho}\left(\frac{(1 + L_{g, u})V}{(1 - \rho)\rho} + 2C(R_S + C \norm{x_0}) + 2 R_S\right)\right) \nonumber\\
    &+ 2(1 - \ell_F \eta)L_F \left(\frac{\hat{C}_0}{1 - \decayfactor} + \frac{\hat{C}_1 + \hat{C}_2}{(1 - \decayfactor)^{2}} + \frac{\hat{C}_2}{(1 - \decayfactor)^{3}}\right) \eta T\nonumber\\
    &+ 2(1 - \ell_F \eta)L_F \frac{C_{L, f, \theta}}{1 - \decayfactor} \cdot \decayfactor^B T\nonumber\\
    &+ 2 \ell_F \left(\frac{\hat{C}_0}{1 - \decayfactor} + \frac{\hat{C}_1 + \hat{C}_2}{(1 - \decayfactor)^{2}} + \frac{\hat{C}_2}{(1 - \decayfactor)^{3}}\right)^2 \eta^3 T + 2 \ell_F \frac{C_{L, f, \theta}^2}{(1 - \decayfactor)^2} \cdot \decayfactor^{2B} \eta T,
\end{align}
where $\hat{C}_0, \hat{C}_1, \hat{C}_2$ are defined in \Cref{thm:OCO-with-parameter-update-gradient-bias:full}, $c_0 = f_0(x_0, \pi_0(x_0, \theta_0))$, and 
\[L_F = \frac{C_{L, f, \theta}}{1 - \rho}, \  \ell_F = \frac{C L_{g, u} L_{\pi, \theta} C_{\ell, f, (\theta, x)} + \rho C_{\ell, f, (x, \theta)} C_{L, g, \theta}}{(1 - \rho)^2}.\]
\end{theorem}
\begin{proof}[Proof of \Cref{thm:GAPS-nonconvex-regret:detailed}]
By \Cref{thm:nonconvex-biased-OGD} and \Cref{thm:bridge-GAPS-and-OGD}, we know the parameter sequence of GAPS satisfies that
\begin{align}\label{thm:GAPS-nonconvex-regret:detailed:e0}
    \sum_{t = 0}^{T-1}\norm{\nabla F_t(\theta_t)}^2 \leq \frac{2}{\eta}\left(F_0(\theta_0) + L_F \sum_{t=1}^{T-1}\dist_s(F_t, F_{t-1})\right) + \left(2(1 - \ell_F \eta) L_F \beta + \ell_F \eta \cdot \beta^2\right)T,
\end{align}
where $\beta = \left({\hat{C}_0}{(1 - \decayfactor)^{-1}} + {(\hat{C}_1 + \hat{C}_2)}{(1 - \decayfactor)^{-2}} + {\hat{C}_2}{(1 - \decayfactor)^{-3}}\right) \eta + {C_{L, f, \theta}}{(1 - \decayfactor)^{-1}} \cdot \decayfactor^B$. Here, $\hat{C}_0, \hat{C}_1, \hat{C}_2$ are defined in \Cref{thm:OCO-with-parameter-update-gradient-bias:full}.

Note that $L_F = \frac{C_{L, f, \theta}}{1 - \rho}$ by \Cref{thm:OCO-with-parameter-update-trajectory-distance}. Now we show that we can set 
\[\ell_F = \frac{C L_{g, u} L_{\pi, \theta} C_{\ell, f, (\theta, x)} + \rho C_{\ell, f, (x, \theta)} C_{L, g, \theta}}{(1 - \rho)^2}.\]
To see this, by \Cref{lemma:smooth-multi-step-dynamics}, we obtain that the following inequality holds for every time step $t$,
\begin{subequations}\label{thm:GAPS-nonconvex-regret:detailed:e1}
\begin{align}
    \norm{\nabla F_t(\theta) - \nabla F_t(\theta')} ={}& \norm{\sum_{\tau = 0}^t \left.\frac{\partial f_{t\mid \tau}}{\partial \theta_\tau}\right|_{\hat{x}_\tau(\theta), \theta_{\times (t - \tau + 1)}} - \sum_{\tau = 0}^t \left.\frac{\partial f_{t\mid \tau}}{\partial \theta_\tau}\right|_{\hat{x}_\tau(\theta'), \theta'_{\times (t - \tau + 1)}}} \label{thm:GAPS-nonconvex-regret:detailed:e1:s1}\\
    \leq{}& \sum_{\tau = 0}^t \norm{\left.\frac{\partial f_{t\mid \tau}}{\partial \theta_\tau}\right|_{\hat{x}_\tau(\theta), \theta_{\times (t - \tau + 1)}} - \left.\frac{\partial f_{t\mid \tau}}{\partial \theta_\tau}\right|_{\hat{x}_\tau(\theta'), \theta'_{\times (t - \tau + 1)}}}, \label{thm:GAPS-nonconvex-regret:detailed:e1:s2}
\end{align}
\end{subequations}
where we use the definition of surrogate cost functions in \eqref{thm:GAPS-nonconvex-regret:detailed:e1:s1}; we use the triangle inequality in \eqref{thm:GAPS-nonconvex-regret:detailed:e1:s2}.

Note that for each term in \eqref{thm:GAPS-nonconvex-regret:detailed:e1}, we can decompose it as
\begin{align}\label{thm:GAPS-nonconvex-regret:detailed:e2}
    &\norm{\left.\frac{\partial f_{t\mid \tau}}{\partial \theta_\tau}\right|_{\hat{x}_\tau(\theta), \theta_{\times (t - \tau + 1)}} - \left.\frac{\partial f_{t\mid \tau}}{\partial \theta_\tau}\right|_{\hat{x}_\tau(\theta'), \theta'_{\times (t - \tau + 1)}}}\nonumber\\
    \leq{}& \norm{\left.\frac{\partial f_{t\mid \tau}}{\partial \theta_\tau}\right|_{\hat{x}_\tau(\theta), \theta_{\times (t - \tau + 1)}} - \left.\frac{\partial f_{t\mid \tau}}{\partial \theta_\tau}\right|_{\hat{x}_\tau(\theta'), \theta_{\times (t - \tau + 1)}}}\nonumber\\
    &+ \sum_{j = \tau}^{t} \norm{\left.\frac{\partial f_{t\mid \tau}}{\partial \theta_\tau}\right|_{\hat{x}_\tau(\theta'), \theta'_{\times (j - \tau)}, \theta_{\times (t-j+1)}} - \left.\frac{\partial f_{t\mid \tau}}{\partial \theta_\tau}\right|_{\hat{x}_\tau(\theta'), \theta'_{\times (j - \tau + 1)}, \theta_{\times (t-j)}}}.
\end{align}
Note that for any time step $\tau$, by the triangle inequality, we have
\begin{align*}
    \norm{\hat{x}_\tau(\theta) - \hat{x}_\tau(\theta')} \leq{}& \sum_{j = 0}^{\tau - 1} \norm{g_{\tau\mid 0}(x_0, \theta_{\times j}, \theta'_{\times (\tau - j)}) - g_{\tau\mid 0}(x_0, \theta_{\times (j + 1)}, \theta'_{\times (\tau - j - 1)})}\\
    \leq{}& \sum_{j = 0}^{\tau - 1} C \rho^{\tau - j - 1} \cdot L_{g, u} L_{\pi, \theta}\norm{\theta - \theta'} \leq \frac{C L_{g, u} L_{\pi, \theta}}{1 - \rho} \cdot \norm{\theta - \theta'},
\end{align*}
where we apply the time-invariant contractive perturbation in the last inequality. Therefore, by \Cref{coro:smooth-multi-step-costs}, we obtain that
\begin{align}\label{thm:GAPS-nonconvex-regret:detailed:e3}
    &\norm{\left.\frac{\partial f_{t\mid \tau}}{\partial \theta_\tau}\right|_{\hat{x}_\tau(\theta), \theta_{\times (t - \tau + 1)}} - \left.\frac{\partial f_{t\mid \tau}}{\partial \theta_\tau}\right|_{\hat{x}_\tau(\theta'), \theta_{\times (t - \tau + 1)}}}\nonumber\\
    \leq{}& C_{\ell, f, (\theta, x)} \rho^{t - \tau} \norm{\hat{x}_\tau(\theta) - \hat{x}_\tau(\theta')} \leq \frac{C L_{g, u} L_{\pi, \theta} C_{\ell, f, (\theta, x)}}{1 - \rho} \cdot \rho^{t - \tau} \cdot \norm{\theta - \theta'}.
\end{align}
We also see that
\begin{align}\label{thm:GAPS-nonconvex-regret:detailed:e4}
    &\norm{\left.\frac{\partial f_{t\mid \tau}}{\partial \theta_\tau}\right|_{\hat{x}_\tau(\theta'), \theta'_{\times (j - \tau)}, \theta_{\times (t-j+1)}} - \left.\frac{\partial f_{t\mid \tau}}{\partial \theta_\tau}\right|_{\hat{x}_\tau(\theta'), \theta'_{\times (j - \tau + 1)}, \theta_{\times (t-j)}}}\nonumber\\
    ={}& \norm{\left(\left.\frac{\partial f_{t\mid j}}{\partial x_j}\right|_{\hat{x}_j(\theta'), \theta_{\times (t-j+1)}} - \left.\frac{\partial f_{t\mid j}}{\partial \theta_j}\right|_{\hat{x}_j(\theta'), \theta', \theta_{\times (t-j)}}\right)\left.\frac{\partial g_{j\mid \tau}}{\partial \theta_\tau}\right|_{\hat{x}_\tau(\theta'), \theta'_{\times (j - \tau)}}}\nonumber\\
    \leq{}& \norm{\left.\frac{\partial f_{t\mid j}}{\partial x_j}\right|_{\hat{x}_j(\theta'), \theta_{\times (t-j+1)}} - \left.\frac{\partial f_{t\mid j}}{\partial \theta_j}\right|_{\hat{x}_j(\theta'), \theta', \theta_{\times (t-j)}}}\cdot \norm{\left.\frac{\partial g_{j\mid \tau}}{\partial \theta_\tau}\right|_{\hat{x}_\tau(\theta'), \theta'_{\times (j - \tau)}}}\nonumber\\
    \leq{}& C_{\ell, f, (x, \theta)} \rho^{t - j} \norm{\theta - \theta'} \cdot C_{L, g, \theta} \rho^{j - \tau} = C_{\ell, f, (x, \theta)} C_{L, g, \theta} \rho^{t - \tau} \norm{\theta - \theta'}.
\end{align}
Substituting \eqref{thm:GAPS-nonconvex-regret:detailed:e3} and \eqref{thm:GAPS-nonconvex-regret:detailed:e4} into \eqref{thm:GAPS-nonconvex-regret:detailed:e2} gives that
\begin{align*}
    &\norm{\left.\frac{\partial f_{t\mid \tau}}{\partial \theta_\tau}\right|_{\hat{x}_\tau(\theta), \theta_{\times (t - \tau + 1)}} - \left.\frac{\partial f_{t\mid \tau}}{\partial \theta_\tau}\right|_{\hat{x}_\tau(\theta'), \theta'_{\times (t - \tau + 1)}}}\nonumber\\
    \leq{}& \left(\frac{C L_{g, u} L_{\pi, \theta} C_{\ell, f, (\theta, x)}}{1 - \rho} + (t - \tau) C_{\ell, f, (x, \theta)} C_{L, g, \theta}\right)\cdot \rho^{t - \tau} \norm{\theta - \theta'}.
\end{align*}
Substituting this inequality into \eqref{thm:GAPS-nonconvex-regret:detailed:e1} gives that
\begin{align*}
    \norm{\nabla F_t(\theta) - \nabla F_t(\theta')} \leq{}& \sum_{\tau = 0}^t \left(\frac{C L_{g, u} L_{\pi, \theta} C_{\ell, f, (\theta, x)}}{1 - \rho} + (t - \tau) C_{\ell, f, (x, \theta)} C_{L, g, \theta}\right)\cdot \rho^{t - \tau} \norm{\theta - \theta'}\\
    \leq{}& \frac{C L_{g, u} L_{\pi, \theta} C_{\ell, f, (\theta, x)} + \rho C_{\ell, f, (x, \theta)} C_{L, g, \theta}}{(1 - \rho)^2} \cdot \norm{\theta - \theta'}.
\end{align*}
Therefore, we can set $\ell_F = \frac{C L_{g, u} L_{\pi, \theta} C_{\ell, f, (\theta, x)} + \rho C_{\ell, f, (x, \theta)} C_{L, g, \theta}}{(1 - \rho)^2}$.

Recall that the notation $\dist_s$ is defined in \Cref{thm:nonconvex-biased-OGD}.
By \Cref{lemma:path-length-bound}, we know that
\begin{align*}
    &\sum_{t=1}^{T-1} \dist_s(F_t, F_{t-1})\\
    \leq{}& \frac{2C L_f (1 + L_{\pi, x}) (1 + L_{g, u})}{(1 - \rho)^2 \rho} \cdot V + \frac{4 C L_f (1 + L_{\pi, x}) \left(C(R_S + C\norm{x_0}) + R_S\right)}{1 - \rho}.
\end{align*}
Substituting this inequality and the expressions of $L_F, \ell_F$ into \eqref{thm:GAPS-nonconvex-regret:detailed:e0} finishes the proof.
\end{proof}

\subsection{Proof of Theorem \ref{thm:nonconvex-biased-OGD}}\label{appendix:thm:nonconvex-biased-OGD:proof}
By the smoothness of $F_t(\cdot)$, we see that
\begin{align}\label{thm:deterministic-OGD:e1}
    F_t(\theta_{t+1}) \leq{}& F_t(\theta_t) + \langle \nabla F_t(\theta_t), \theta_{t+1} - \theta_t\rangle + \frac{\ell_F}{2}\norm{\theta_{t+1} - \theta_t}^2\nonumber\\
    ={}& F_t(\theta_t) - \eta \langle \nabla F_t(\theta_t), G_t\rangle + \frac{\ell_F \eta^2}{2}\norm{G_t}^2\nonumber\\
    ={}& F_t(\theta_t) - \eta \langle \nabla F_t(\theta_t), \nabla F_t(\theta_t)\rangle + \frac{\ell_F \eta^2}{2}\norm{\nabla F_t(\theta_t)}^2\nonumber\\
    &- \eta \langle \nabla F_t(\theta_t), G_t - \nabla F_t(\theta_t)\rangle + \ell_F \eta^2 \langle \nabla F_t(\theta_t), G_t - \nabla F(\theta_t)\rangle\nonumber\\
    &+ \frac{\ell_F \eta^2}{2}\norm{\nabla F_t(\theta_t) - G_t}^2\nonumber\\
    \leq{}& F_t(\theta_t) - \eta \left(1 - \frac{\ell_F \eta}{2}\right)\norm{\nabla F_t(\theta_t)}^2 + \eta (1 - \ell_F \eta) L \beta + \frac{\ell_F \eta^2}{2}\cdot \beta^2.
\end{align}
Summing \eqref{thm:deterministic-OGD:e1} over $t = 0, 1, \ldots, T-1$ and rearranging the terms gives that
\begin{align}\label{thm:deterministic-OGD:e2}
    &\eta \left(1 - \frac{\ell_F \eta}{2}\right)\sum_{t = 0}^{T-1} \norm{\nabla F_t(\theta_t)}^2\nonumber\\
    \leq{}& \sum_{t = 0}^{T-1} \left(F_t(\theta_t) - F_t(\theta_{t+1})\right) + \left(\eta (1 - \ell_F \eta) L \beta + \frac{\ell_F \eta^2}{2}\cdot \beta^2\right)T\nonumber\\
    \leq{}& F_0(\theta_0) + \sum_{t=1}^{T-1} \left(F_t(\theta_t) - F_{t-1}(\theta_t)\right) + \left(\eta (1 - \ell_F \eta) L \beta + \frac{\ell_F \eta^2}{2}\cdot \beta^2\right)T\nonumber\\
    \leq{}& F_0(\theta_0) + L_F\sum_{t=1}^{T-1}\dist_s(F_t, F_{t-1}) + \left(\eta (1 - \ell_F \eta) L_F \beta + \frac{\ell_F \eta^2}{2}\cdot \beta^2\right)T.
\end{align}

\subsection{Proof of Lemma \ref{lemma:time-invariant-dynamics-fixed-point}}\label{appendix:lemma:time-invariant-dynamics-fixed-point:proof}
We first show the limit $\lim_{\tau \to \infty} \tilde{x}_\tau^{(g, \pi)}(\theta)$ exists. It suffices to show that $\{\tilde{x}_\tau^{(g, \pi)}(\theta)\}$ is a Cauchy sequence. Note that for $\tau' > \tau \geq 0$, we have
\begin{subequations}\label{lemma:time-invariant-dynamics-fixed-point:e1}
\begin{align}
    \norm{\tilde{x}_{\tau'}^{(g, \pi)}(\theta) - \tilde{x}_{\tau}^{(g, \pi)}(\theta)} \leq{}& \sum_{j = \tau}^{\tau'-1} \norm{\tilde{x}_{j+1}^{(g, \pi)}(\theta) - \tilde{x}_{j}^{(g, \pi)}(\theta)} \label{lemma:time-invariant-dynamics-fixed-point:e1:s1}\\
    \leq{}& \sum_{j = \tau}^{\tau'-1} C \rho^j \norm{\tilde{x}_1^{(g, \pi)}(\theta) - x_0} \label{lemma:time-invariant-dynamics-fixed-point:e1:s2}\\
    \leq{}& \frac{C \rho^\tau}{1 - \rho} \cdot \norm{\tilde{x}_1^{(g, \pi)}(\theta) - x_0}, \nonumber
\end{align}
\end{subequations}
where we use the triangle inequality in \eqref{lemma:time-invariant-dynamics-fixed-point:e1:s1} and \Cref{assump:contractive-and-stability} in \eqref{lemma:time-invariant-dynamics-fixed-point:e1:s2}. Therefore, we see the limit $\lim_{\tau \to \infty} \tilde{x}_\tau^{(g, \pi)}(\theta)$ exists because $\{\tilde{x}_\tau^{(g, \pi)}(\theta)\}$ is a Cauchy sequence and we denote $\tilde{x}_\infty^{(g, \pi)}(\theta) \coloneqq \lim_{\tau \to \infty} \tilde{x}_\tau^{(g, \pi)}(\theta)$.

Now, we show that $\tilde{x}_\infty^{(g, \pi)}(\theta)$ is a fixed point of the time-invariant closed-loop dynamics induced by $(g, \pi, \theta)$. To see this, note that
\begin{align*}
    g\left(\tilde{x}_\infty^{(g, \pi)}(\theta), \pi(\tilde{x}_\infty^{(g, \pi)}(\theta), \theta)\right) ={}& g\left(\lim_{\tau \to \infty}\tilde{x}_\tau^{(g, \pi)}(\theta), \pi(\lim_{\tau \to \infty}\tilde{x}_\tau^{(g, \pi)}(\theta), \theta)\right)\\
    ={}& \lim_{\tau \to \infty} g\left(\tilde{x}_\tau^{(g, \pi)}(\theta), \pi(\tilde{x}_\tau^{(g, \pi)}(\theta), \theta)\right)\\
    ={}& \lim_{\tau \to \infty} \tilde{x}_{\tau+1}^{(g, \pi)}(\theta) = \tilde{x}_\infty^{(g, \pi)}(\theta),
\end{align*}
where we can pull out $\lim_{\tau \to \infty}$ in the second equation because the right hand side is a continuous function of $\tilde{x}_\tau^{(g, \pi)}(\theta)$ at the point $\tilde{x}_\infty^{(g, \pi)}(\theta)$ by \Cref{assump:Lipschitz-and-smoothness}.

Therefore, applying the contractive perturbation property in \Cref{assump:contractive-and-stability} gives that
\[\norm{\tilde{x}_\tau^{(g, \pi)}(\theta) - \tilde{x}_\infty^{(g, \pi)}(\theta)} \leq C \rho^\tau \norm{x_0 - \tilde{x}_\infty^{(g, \pi)}(\theta)} \leq C \rho^{\tau} \diam(\mathcal{X}).\]

\subsection{Proof of Lemma \ref{lemma:path-length-bound}}
\label{appendix:lemma:path-length-bound:proof}
To simplify the notation, we introduce the notations
\begin{align*}
    \dist_d(g, g') &\coloneqq \sup_{x \in \mathcal{X}, u \in \mathcal{U}} \norm{g(x, u) - g'(x, u)},
    & \dist_p(\pi, \pi') &\coloneqq \sup_{x \in \mathcal{X}, \theta \in \Theta} \norm{\pi(x, \theta) - \pi'(x, \theta)},\\*
    \dist_c(f, f') &\coloneqq \sup_{x \in \mathcal{X}, u \in \mathcal{U}} \abs{f(x, u) - f'(x, u)}.
\end{align*}
For any $\theta \in \Theta$, we see that
\begin{subequations}\label{lemma:path-length-bound:e1}
\begin{align}
    &\abs{F_t(\theta) - F_{t-1}(\theta)}\nonumber\\
    ={}& \abs{f_t\left(\hat{x}_t(\theta), \pi_t(\hat{x}_t(\theta), \theta)\right) - f_{t-1}\left(\hat{x}_{t-1}(\theta), \pi_{t-1}(\hat{x}_{t-1}(\theta), \theta)\right)} \label{lemma:path-length-bound:e1:s1}\\
    \leq{}& \abs{f_t\left(\hat{x}_t(\theta), \pi_t(\hat{x}_t(\theta), \theta)\right) - f_t\left(\hat{x}_{t-1}(\theta), \pi_{t-1}(\hat{x}_{t-1}(\theta), \theta)\right)} + \dist_c(f_t, f_{t-1}) \label{lemma:path-length-bound:e1:s2}\\
    \leq{}& L_f \left(\norm{\hat{x}_t(\theta) - \hat{x}_{t-1}(\theta)} + \norm{\pi_t(\hat{x}_t(\theta), \theta) - \pi_{t-1}(\hat{x}_{t-1}(\theta), \theta)}\right) + \dist_c(f_t, f_{t-1}) \label{lemma:path-length-bound:e1:s3}\\
    \leq{}& L_f \left(\norm{\hat{x}_t(\theta) - \hat{x}_{t-1}(\theta)} + \norm{\pi_t(\hat{x}_t(\theta), \theta) - \pi_t(\hat{x}_{t-1}(\theta), \theta)}\right)\nonumber\\
    &+ (L_f \dist_p(\pi_t, \pi_{t-1}) + \dist_c(f_t, f_{t-1})) \label{lemma:path-length-bound:e1:s4}\\
    \leq{}& L_f (1 + L_{\pi, x}) \norm{\hat{x}_t(\theta) - \hat{x}_{t-1}(\theta)} + (L_f \dist_p(\pi_t, \pi_{t-1}) + \dist_c(f_t, f_{t-1})), \label{lemma:path-length-bound:e1:s5}
\end{align}
\end{subequations}
where we use the definition of the surrogate cost in \eqref{lemma:path-length-bound:e1:s1}; we use the definition of $\dist_c$ in \eqref{lemma:path-length-bound:e1:s2}; we use the assumption that $f_t$ is Lipschitz in \eqref{lemma:path-length-bound:e1:s3}; we use the definition of $\dist_\pi$ on \eqref{lemma:path-length-bound:e1:s4}; we use the assumption that $\pi_t$ is Lipschitz in \eqref{lemma:path-length-bound:e1:s5}.

To bound $\norm{\hat{x}_t(\theta) - \hat{x}_{t-1}(\theta)}$, we first bound the difference between $\hat{x}_t(\theta)$ and $\tilde{x}_t^{(g, \pi)}(\theta)$ for arbitrary $(g, \pi) \in \mathcal{G}$. Note that $\hat{x}_t(\theta) = g_{t\mid 0}\left(\tilde{x}_0^{(g, \pi)}, \theta_{\times t}\right)$, thus
\begin{subequations}\label{lemma:path-length-bound:e2}
\begin{align}
    \norm{\hat{x}_t(\theta) - \tilde{x}_t^{(g, \pi)}(\theta)} \leq{}& \sum_{\tau = 0}^{t-1} \norm{g_{t\mid \tau + 1}\left(\tilde{x}_{\tau + 1}^{(g, \pi)}(\theta), \theta_{\times (t - \tau - 1)}\right) - g_{t\mid \tau}\left(\tilde{x}_\tau^{(g, \pi)}(\theta), \theta_{\times (t-\tau)}\right)} \label{lemma:path-length-bound:e2:s1}\\
    \leq{}& C \sum_{\tau = 0}^{t-1} \rho^{t-\tau-1} \norm{\tilde{x}_{\tau + 1}^{(g, \pi)}(\theta) - g_{\tau+1\mid \tau}\left(\tilde{x}_{\tau + 1}^{(g, \pi)}(\theta), \theta\right)} \label{lemma:path-length-bound:e2:s2}\\
    \leq{}& C \sum_{\tau = 0}^{t-1} \rho^{t-\tau-1} \left(\dist_d(g_\tau, g) + L_{g, u} \dist_p(\pi_\tau, \pi)\right), \label{lemma:path-length-bound:e2:s3}
\end{align}
\end{subequations}
where we use the triangle inequality in \eqref{lemma:path-length-bound:e2:s1}; we use the contractive perturbation property in \eqref{lemma:path-length-bound:e2:s2}; we use the definitions of $\dist_d$ and $\dist_p$ in \eqref{lemma:path-length-bound:e2:s3}.

Therefore, we see that
\begin{subequations}\label{lemma:path-length-bound:e3}
\begin{align}
    \norm{\hat{x}_t(\theta) - \hat{x}_{t-1}(\theta)} \leq{}& \norm{\hat{x}_t(\theta) - \tilde{x}_t^{(g_t, \pi_t)}(\theta)} + \norm{\hat{x}_{t-1}(\theta) - \tilde{x}_{t-1}^{(g_t, \pi_t)}(\theta)}\nonumber\\
    &+ \norm{\tilde{x}_t^{(g_t, \pi_t)}(\theta) - \tilde{x}_\infty^{(g_t, \pi_t)}(\theta)} + \norm{\tilde{x}_{t-1}^{(g_t, \pi_t)}(\theta) - \tilde{x}_\infty^{(g_t, \pi_t)}(\theta)} \label{lemma:path-length-bound:e3:s1}\\
    \leq{}& C \sum_{\tau = 0}^{t-1} \rho^{t-\tau-1} \left(\dist_d(g_\tau, g_t) + L_{g, u} \dist_p(\pi_\tau, \pi_t)\right) \nonumber\\
    &+ C \sum_{\tau = 0}^{t-2} \rho^{t-\tau-2} \left(\dist_d(g_\tau, g_t) + L_{g, u} \dist_p(\pi_\tau, \pi_t)\right) \nonumber\\
    &+ C \rho^t \diam(\mathcal{X}) + C \rho^{t-1} \diam(\mathcal{X}) \label{lemma:path-length-bound:e3:s2}\\
    \leq{}& \frac{2C}{1 - \rho} \cdot \sum_{\tau = 0}^{t-1} \rho^{t - \tau - 2} \left(\dist_d(g_\tau, g_{\tau + 1}) + L_{g, u} \dist_p(\pi_\tau, \pi_{\tau + 1})\right) \nonumber\\
    &+ 2 C \rho^{t-1} \diam(\mathcal{X}), \label{lemma:path-length-bound:e3:s3}
\end{align}
\end{subequations}
where we use the triangle inequality in \eqref{lemma:path-length-bound:e3:s1}; in \eqref{lemma:path-length-bound:e3:s2}, we use the bound we derived in \eqref{lemma:path-length-bound:e2}; in \eqref{lemma:path-length-bound:e3:s3}, we use the triangle inequality decomposition
\[\dist_d(g_\tau, g_t) \leq \sum_{j = \tau}^{t-1} \dist_d(g_{j+1}, g_j) \text{, and } \dist_p(\pi_\tau, \pi_t) \leq \sum_{j = \tau}^{t-1} \dist_p(\pi_{j+1}, \pi_j).\]
Substituting this into \eqref{lemma:path-length-bound:e1} gives that
\begin{align}\label{lemma:path-length-bound:e4}
    \dist_s(F_t, F_{t-1}) \leq \frac{2C L_f (1 + L_{\pi, x})}{1 - \rho} \cdot \sum_{\tau = 0}^{t-1} \rho^{t - \tau - 2} \left(\dist_d(g_\tau, g_{\tau + 1}) + L_{g, u} \dist_p(\pi_\tau, \pi_{\tau + 1})\right) \nonumber\\
    + 2 C L_f (1 + L_{\pi, x}) \rho^{t-1} \diam(\mathcal{X}) + (L_f \dist_p(\pi_t, \pi_{t-1}) + \dist_c(f_t, f_{t-1}))
\end{align}
because \eqref{lemma:path-length-bound:e3} holds for arbitrary $\theta \in \Theta$.

Summing \eqref{lemma:path-length-bound:e4} over $t = 1, \ldots, T-1$ and rearranging the terms give that
\begin{align*}
    &\sum_{t=1}^{T-1} \dist_s(F_t, F_{t-1})\\
    \leq{}& \frac{2C L_f (1 + L_{\pi, x})}{1 - \rho} \cdot \sum_{t=1}^{T-1} \sum_{\tau = 0}^{t-1} \rho^{t - \tau - 2} \left(\dist_d(g_\tau, g_{\tau + 1}) + L_{g, u} \dist_p(\pi_\tau, \pi_{\tau + 1})\right)\\
    &+ 2 C L_f (1 + L_{\pi, x}) \sum_{t=1}^{T-1} \rho^{t-1} \diam(\mathcal{X}) + \sum_{t=1}^{T-1} (L_f \dist_p(\pi_t, \pi_{t-1}) + \dist_c(f_t, f_{t-1}))\\
    \leq{}& \frac{2C L_f (1 + L_{\pi, x}) (1 + L_{g, u})}{(1 - \rho)^2 \rho} \cdot P_{\mathcal{T}} + \frac{2 C L_f (1 + L_{\pi, x})}{1 - \rho} \cdot \diam(\mathcal{X}).
\end{align*}

\section{Learning Finite Policy Parameters}\label{appendix:finite}
In this appendix, we consider a different setting of online policy selection where the policy parameters come from a finite set $\Theta = \{1, 2, \ldots, \controlnum\}$ rather than a convex subset of $\mathbb{R}^d$. In this setting, we can also adopt a proof approach based on the contractive perturbation property to achieve a sublinear (static) policy regret bound. Specifically, we establish a regret bound in expectation under the same assumptions as the continuous parameter case \footnote{Since step size constraints on policy parameters are not well-defined in the finite setting, by ``under the same assumptions'', we mean \Cref{assump:contractive-and-stability} holds with $\varepsilon = 0$.}, though it requires a different algorithm. Our results for the finite setting also requires less information: here we do not assume knowledge of the partial derivatives of the dynamics and cost.

\subsection{Algorithm}

We propose an algorithm, Bandit-based Adaptive Policy Selection (BAPS, \Cref{alg:batched-exp3-policy-selection}), that deploys the same policy parameter for a batch of~$b$ time steps before switching.
Intuitively, when we switch the policy parameter, we wait until the state ``forgets'' the impact of the parameter used in the previous batch before we evaluate the current parameter's performance.
With a sufficiently large batch size~$b$, we can guarantee that the difference between the mean actual cost over a batch and the mean surrogate cost is negligible.

\begin{algorithm}
\caption{Bandit-based Adaptive Policy Selection (BAPS)}\label{alg:batched-exp3-policy-selection}
\begin{algorithmic}[1]
\REQUIRE policy count $\controlnum$, batch size $b$, learning rate $\eta$.
\STATE Set the initial distribution to uniform: $s_0 = (1/\controlnum)_{\times \controlnum}$.
\FOR{$m = 0, \ldots, T/b-1$}
    \STATE  Sample $j_m \sim s_m$ and for all $t \in [mb: (m+1)b-1]$, pick $u_t = \ALG_t(x_t, j_m)$. \label{alg:baps:sampling}
    \STATE Compute the batched loss vector $\hat{\ell}_m^b \in \mathbb{R}^\controlnum$:
    \begin{align*}
        \hat{\ell}_m^b(j) = \begin{cases}
            \frac{1}{s_m(j_m)} \sum_{t = mb}^{(m+1)b-\!1} f_t(x_t, u_t), & j = j_m\\
            0, &\text{$j\not = j_m$}.
        \end{cases}
    \end{align*}
    \vspace{-2mm}
    \STATE Set $s_{m+1}(j) = s_m(j) e^{- \eta \hat{\ell}_m^b(j)}, \forall j \in \{1, \ldots, \controlnum\}$ and do the normalization $s_{m+1} \gets {s_{m+1}}/{\norm{s_{m+1}}_1}$.
\ENDFOR
\end{algorithmic}
\end{algorithm}

\subsection{Analysis}
We present our main result for BAPS in \Cref{thm:batched-exp3-policy-selection}.
\begin{theorem}\label{thm:batched-exp3-policy-selection}
For the finite parameter setting, suppose \Cref{assump:Lipschitz-and-smoothness} holds and \Cref{assump:contractive-and-stability} holds with $\varepsilon = 0$. Then, BAPS (Algorithm \ref{alg:batched-exp3-policy-selection}) with batch size $b \geq \tau_0$ and learning step size $\eta > 0$
attains the following expected policy regret,
where the expectation is taken over the randomness of the sampling mechanism in line 3 of \Cref{alg:batched-exp3-policy-selection}:
\[\mathbb{E}\left[R^S(T)\right] \leq {C D_0}\cdot {(1 - \decayfactor)^{-1} b^{-1}} T + \controlnum D_0 \eta b T + {\log \controlnum}\cdot {\eta^{-1}},\]
where $R^S(T) = \sum_{t = 0}^{T-1} f_t(x_t, u_t) - \min_{\theta \in \Theta} \sum_{t=0}^{T-1} F_t(\theta)$. Here, we define
\begin{align*}
    &R_0 = \min\{\frac{C\norm{x_0} + R_S}{\rho}, R_C\}, \tau_0 = \log\left(\frac{C R_0}{R_0 - R_S}\right)/\log\left(\frac{1}{\rho}\right), \text{ and }\\
    &D_0 = \max_{t\in \mathcal{T}}\sup_{x \in \mathcal{X}, u \in \mathcal{U}}\abs{f_t(x, u)}.
\end{align*}
\end{theorem}
We discuss how to choose the learning rate and the regret it achieves in the following corollary.

\begin{corollary}\label{thm:batched-exp3-policy-selection:simplified}
Under the same assumptions as \Cref{thm:batched-exp3-policy-selection}, suppose $T \gg \tau_0$, if we set the parameters as $b = \left(\frac{C_0^2 D_0 T}{(1 - \rho)^2 k \log k}\right)^{\frac{1}{3}}$ and $\eta = \left(\frac{(1 - \rho)(\log k)^2}{C_0 D_0^2 k T^2}\right)^{\frac{1}{3}}$, we have $\mathbb{E}\left[R^S(T)\right] = 3 \left(\frac{C D_0^2 \cdot k \log k\cdot T^2}{1 - \rho}\right)^{\frac{1}{3}}$.
\end{corollary}

Note that the $O(T^{2/3})$ static policy regret bound in expectation is worse than the $O(\sqrt T)$ bound for GAPS in \Cref{thm:main-regret-bound-convex} in the dependence on the horizon $T$.
However, it matches the regret lower bound for MAB with switching costs \cite{dekel2014bandits}, which we conjecture to be the parallel online learning setting for online policy selection. The intuition is that the possible extra cost incurred when we wait for the trajectory to converge to $\hat{x}_t(\theta)$ after switching to a new policy parameter $\theta$ can be modeled as an unknown switching cost with a constant upper bound. An interesting open question is to verify this intuition and see whether the $O(T^{2/3})$ static policy regret is tight for finite policy selection under the assumptions of BAPS.

Now we present the proof of \Cref{thm:batched-exp3-policy-selection}.

\begin{proof}[Proof of \Cref{thm:batched-exp3-policy-selection}]
We first show by induction that the trajectory achieved by a ``batched'' policy parameter with batch size $b \geq \tau_0$ will always stay within $\mathcal{X} \times \mathcal{U}$, so that the cost is uniformly bounded by $D_0$.

To see this, we show the state $x_{mb}$ at first time step in each batch satisfies $\norm{x_{mb}} \leq R_0$. Note that the statement holds for $m = 0$. Suppose it holds for $m-1$. By \Cref{assump:contractive-and-stability}, we see that
\[\norm{x_{mb}} \leq C \rho^{b}\norm{x_{(m-1)b}} + R_S \leq C \rho^{\tau_0} \cdot R_0 + R_S \leq R_0,\]
where we use the definition of $\tau_0$ in the last inequality. So $\norm{x_{mb}} \leq R_0$ holds for all $m$ by induction. Then, for any time step $t$ that satisfies $mb < t < (m+1)b$, we have
\[\norm{x_t} \leq C \rho^{t - mb}\norm{x_{mb}} + R_S \leq C (C \norm{x_0} + R_S) + R_S.\]
Thus, we see that $(x_t, u_t) \in \mathcal{X} \times \mathcal{U}$ and $\abs{f_t(x_t, u_t)} \leq D_0$.

The rest of the proof is inspired by the proof of Lemma 6.3 in \cite{hazan2016introduction}.
By Theorem 1.5 in \cite{hazan2016introduction}, the following inequality holds for any $j \in \Theta$:
    \begin{align}\label{thm:batched-exp3-regret-bound:e1}
        \sum_{m = 0}^{T/b - 1} s_m^\top \hat{\ell}_m^b \leq \sum_{m=0}^{T/b - 1} \hat{\ell}_m^b(j) + \eta \sum_{m=0}^{T/b - 1} s_m^\top (\hat{\ell}_m^b)^2 + \frac{\log \controlnum}{\eta}.
    \end{align}
    Taking the expectation on both sides gives that
    \begin{align}\label{thm:batched-exp3-regret-bound:e2}
        &\sum_{m = 0}^{T/b - 1} \mathbb{E}\left[s_m^\top \hat{\ell}_m^b\right]\nonumber\\*
        \leq{}& \sum_{m=0}^{T/b - 1} \mathbb{E}\left[\hat{\ell}_m^b(j)\right] + \eta \sum_{m=0}^{T/b - 1} \mathbb{E}\left[s_m^\top (\hat{\ell}_m^b)^2\right] + \frac{\log \controlnum}{\eta}\nonumber\\
        =& \sum_{m=0}^{T/b - 1} \mathbb{E}\left[\mathbb{E}\left[\hat{\ell}_m^b(j)\mid j_{0:m-1}\right]\right] + \eta \sum_{m=0}^{T/b - 1} \mathbb{E}\left[\mathbb{E}\left[s_m^\top (\hat{\ell}_m^b)^2\mid j_{0:m-1}\right]\right] + \frac{\log \controlnum}{\eta}\nonumber\\
        =& \sum_{m=0}^{T/b - 1} \mathbb{E}\left[\sum_{t = mb}^{(m+1)b-1} f_{t\mid mb}(x_{mb}, j_{\times (t - mb)})\right]\nonumber\\
        &+ \eta \sum_{m=0}^{T/b - 1} \mathbb{E}\left[\sum_{q \in \Theta} \left(\sum_{t = mb}^{(m+1)b - 1} f_{t\mid mb}(x_{mb}, q_{\times (t - mb)})\right)^2\right] + \frac{\log \controlnum}{\eta}\nonumber\\
        \leq& \left(\sum_{t = 0}^{T-1} F_t(j) + \frac{2 C_0 L_0 R_x}{1 - \rho} \cdot \frac{T}{b}\right) + \eta b D_0^2 \controlnum T + \frac{\log \controlnum}{\eta}.
    \end{align}
    Note that the left hand side of \eqref{thm:batched-exp3-regret-bound:e2} is equal to the expected cost incurred by Algorithm \ref{alg:batched-exp3-policy-selection}, i.e., $\mathbb{E}\left[\sum_{t=0}^{T-1} f_t(x_t, u_t)\right]$. Therefore, we see that for any $j \in \Theta$, we have
    \[\mathbb{E}\left[\sum_{t=0}^{T - 1} f_t(x_t, u_t)\right] \leq \sum_{t = 0}^{T-1} F_t(j) + \frac{2 C_0 L_0 R_x}{1 - \rho} \cdot \frac{T}{b} + \eta b D_0^2 \controlnum T + \frac{\log \controlnum}{\eta}.\]
    Since the cost functions and dynamics are oblivious, this finishes the proof.
\end{proof}

\section{Discussion of Examples}
\label{appendix:examples}
In this section, we discuss why Examples \ref{example:MPC-confidence} and \ref{example:nonlinear-control} satisfy the time-varying contractive perturbation and stability in \Cref{assump:contractive-and-stability} as well as \Cref{assump:Lipschitz-and-smoothness} about Lipschitzness and smoothness. Before presenting these two examples, we first explain why contractive perturbation generalizes the notion of the disturbance-action controller (DAC), which is a policy class that has been studied intensively in the literature of online control \cite{agarwal2019online, hazan2020nonstochastic, chen2021black}.

\subsection{Disturbance-action controller}
The example about disturbance-action controller (DAC) considers a linear time-varying (LTV) system ${g_t(x_t, u_t) = A_t x_t + B_t u_t + w_t}$, where $w_t$ is a bounded adversarial disturbance at time~$t$. The DAC class is given by
\[
    \ALG_t(x_t, \theta_t) = - K_t^{stb} x_t + \textstyle \sum_{i = 1}^{h} M_t^{[i]} w_{t-i},
\]
where
$K_t^{stb}$ is a given matrix sequence such that for some positive constants $C$ and $\decayfactor < 1$
\[\norm{\prod_{t = t_1}^{t_2 - 1} (A_t - B_t K_t^{stb})} \leq C \decayfactor^{t_2 - t_1}, \forall t_2 > t_1.\]
In this example, $\Theta \coloneqq \{(M^{[i]})_{i \in [1:h]}\mid \sum_{i=1}^{h} \norm{M^{[i]}} \leq R_M\}$ for some $R_M > 0$.

DAC has been widely used in recent online control literature \cite{agarwal2019online, hazan2020nonstochastic, chen2021black}.
When applied to linear dynamics and convex cost functions, the DAC class renders the multi-step cost functions $f_{t\mid \tau}$ convex in the policy parameters, which enables the reduction to OCO with Memory \citep{agarwal2019online, hazan2020nonstochastic}.
It is also shown to be sufficient to approximate any strongly stable state feedback policy when the memory length $h$ is large enough.
\label{sec:drc-contractive-details}

Note that the parameter set $\Theta$ is a convex compact subset of $\mathbb{R}^{hnm}$ with diameter $D \leq n R_M$. For this example to satisfy \Cref{assump:Lipschitz-and-smoothness}, we additionally assume that there exists positive constants $a, b, \bar{k}, \bar{w}$ such that
$\norm{A_t} \leq a$,
$\norm{B_t} \leq b$,
$\norm{K_t^{stb}} \leq \bar{k}$,
$\norm{w_t} \leq \bar{w}$
holds for all $t$.
The cost functions $f_t$ is $(\ell_{f, x}, \ell_{f, u})$-smooth in $(x, u)$ and $\argmin f_t = (0, 0)$.

DAC is special because it satisfies \Cref{assump:contractive-and-stability} with arbitrarily large $\varepsilon$. To see this, we define the state transition matrix as
\begin{align*}
    \Psi_{t\mid \tau} \coloneqq \begin{cases}
        \prod_{\tau' = 1}^{t - \tau} (A_{t-\tau'} - B_{t-\tau'} K_{t-\tau'}^{stb}) & \text{ if } t > \tau,\\
        I & \text{ otherwise.}
    \end{cases}
\end{align*}
By assumption, we know that $\norm{\psi_{t\mid \tau}} \leq C \decayfactor^{t - \tau}$.

We first show the $\varepsilon$-time-varying contractive perturbation property (\Cref{def:epsilon-exp-decay-perturbation-property}) holds for $R_C = +\infty, \varepsilon = +\infty$. To see this, note that for any two time steps $t \geq \tau$, states $x_\tau, x_\tau' \in \mathbb{R}^n$, and $\theta_{\tau:t-1} \in \Theta^{t-\tau}$, we have that
\begin{align*}
    g_{t\mid \tau}(x_\tau, \theta_{\tau:t-1}) - g_{t\mid \tau}(x_\tau', \theta_{\tau:t-1}) = \Psi_{t\mid \tau} (x_\tau - x_\tau'),
\end{align*}
which implies that
\[\norm{g_{t\mid \tau}(x_\tau, \theta_{\tau:t-1}) - g_{t\mid \tau}(x_\tau', \theta_{\tau:t-1})} \leq \norm{\Psi_{t\mid \tau}} \cdot \norm{x_\tau - x_\tau'} \leq C \rho^{t-\tau}\norm{x_\tau - x_\tau'}.\]

To see that the $\varepsilon$-time-varying stability (\Cref{def:epsilon-time-varying-stability}) holds with $R_S = \frac{C R_M \bar{w}}{1 - \rho}$ and $\varepsilon = +\infty$, note that for arbitrary $\theta_{\tau:t-1} \in \Theta^{t-\tau}$, we have
\begin{align}\label{example-proof:e1}
    x_t = g_{t\mid \tau}(x_\tau, \theta_{\tau: t-1}) = \psi_{t\mid \tau} x_\tau + \sum_{\tau' = \tau}^t \psi_{t\mid \tau'} \left(\sum_{i = 1}^h M_{\tau'}^{[i]} w_{\tau' - i}\right).
\end{align}
Taking norms on both side of the equation and applying the triangle inequality gives that
\begin{align*}
    \norm{x_t} \leq{}& \norm{\psi_{t\mid \tau}} \cdot \norm{x_\tau} + \sum_{\tau' = \tau}^t \norm{\psi_{t\mid \tau'}} \left(\sum_{i = 1}^h \norm{M_{\tau'}^{[i]}}\cdot \norm{w_{\tau' - i}}\right)\\
    \leq{}& C \decayfactor^{t-\tau} \norm{x_\tau} + \sum_{\tau' = \tau}^t C \decayfactor^{t - \tau'} R_M \bar{w}\\
    \leq{}& C \decayfactor^{t-\tau} \norm{x_0} + \frac{C R_M \bar{w}}{1 - \rho},
\end{align*}
where we use the assumption that $\norm{\psi_{t\mid \tau}} \leq C \decayfactor^{t - \tau}$, the definition of the policy class, and the bound on disturbances in the second inequality. Therefore, $\varepsilon$-time-varying stability holds with $R_S = \frac{C R_M \bar{w}}{1 - \rho}$.

\subsection{MPC with confidence coefficients}
\label{sec:mpc-contractive-details}
Our example about Model Predictive Control (MPC) with Confidence Coefficients generalizes the $\lambda$-confident policy proposed by \citet{li2021robustness}.
In this setting, some policy parameter $\theta^{(c)} \in \Theta$ can achieve near-optimal performance when the predictions of the future are accurate (consistency), and another policy parameter $\theta^{(r)} \in \Theta$  has a worst-case guarantee even when the predictions are unreliable (robustness).
Minimizing regret in this setting implies that the online policy is both robust and consistent.

Recall that in this example, MPC selects the current control action by solving the optimization problem
\begin{align}
\label{eq:mpc-opt-confidence:appendix}
     \argmin_{u_{t:t+k-1\mid t}}& \sum_{\tau = t}^{t + k-1} f_\tau(x_{\tau|t}, u_{\tau|t}) + \quadratic(x_{t+k\mid t}, \tilde{Q}) \nonumber\\*
    \operatorname{s.t.}& \; x_{t\mid t} = x_t,\\*
    &\ x_{\tau+1\mid t} = A_\tau x_{\tau\mid t} + B_\tau u_{\tau\mid t} +\lambda_t^{[\tau - t]} w_{\tau\mid t} : \; t \leq \tau < t\!+\!k,
    \nonumber
\end{align}
where $\theta_t = \big(\lambda_t^{[0]}, \lambda_t^{[1]}, \ldots, \lambda_t^{[k-1]}\big)$, $\Theta \subseteq [0, 1]^k$. Thus, $\Theta$ is a convex compact set with diameter $\sqrt{k}$. 

For this example to satisfy Assumptions \ref{assump:Lipschitz-and-smoothness} and \ref{assump:contractive-and-stability}, we make two standard assumptions that are also required by prior works on online control with MPC \cite{lin2021perturbation,lin2022bounded}. The first assumption is about the uniform bounds on the dynamical matrices, cost function matrices, and disturbances.

\begin{assumption}\label{assump:MPC-example-bounds}
For any time step $t \in \mathcal{T}$, we have $\norm{A_t} \leq a, \norm{B_t} \leq b, \norm{w_t} \leq \bar{w}$, and
\[\mu I_n \preceq Q_t \preceq \ell I_n, \mu I_m \preceq R_t \preceq \ell I_m, \mu I_n \preceq \tilde{Q} \preceq \ell I_n.\]
We also assume that $\norm{w_{\tau\mid t}} \leq \bar{w}$ for predicted disturbances.
\end{assumption}

The second assumption is about the uniform controllability of the LTV system:

\begin{assumption}\label{assump:MPC-example-controllability}
We define the transition matrix in this LTV system as
\begin{align*}
    \Phi(t_2, t_1) \coloneqq \begin{cases}
        A_{t_2-1} \cdots A_{t_1}, &\text{ if } t_2 > t_1,\\
        I, & \text{ otherwise.}
    \end{cases}
\end{align*}
For any two time steps $t' > t$, we define
\begin{align*}
    \Xi_{t, t'} = \left[\Phi(t', t+1)B_t, \Phi(t', t+2)B_{t+1}, \ldots, \Phi(t', t')B_{t'}\right].
\end{align*}
We assume that there exists a positive integer $d_0$ such that the smallest singular value of the matrix $\Xi_{t, t'}$ is uniformly lower bounded by some positive constant $\sigma$ when $t' \geq t + d_0$, i.e., $\sigma_{min}\left(\Xi_{t, t'}\right) \geq \sigma$ holds for any $t' \geq t + d_0$, where $\sigma_{\min}(\cdot)$ denotes the smallest singular value of a matrix.
\end{assumption}

Before we proceed to show that Assumptions \ref{assump:Lipschitz-and-smoothness} and \ref{assump:contractive-and-stability}, we first define an auxiliary parameterized optimal problem control problem that will be used in our analysis: For any two time steps $t < t'$, let $\psi_t^{t'}(y_t, v_{t:t'}; P_{t'})$ denote the optimal trajectory planned according to initial state $y_t$, disturbances $v_{t:t'}$, and terminal cost matrix $P_{t'}$, i.e.,
\begin{align*}
\psi_t^{t'}(y_t, v_{t:t'}; P_{t'}) = \argmin_{x_{t:t'}, u_{t:(t'-1)}}& \sum_{\tau = t}^{t' - 1} f_\tau\left(x_{\tau}, u_{\tau}\right) + \frac{1}{2} x_{t'}^\top P_{t'} x_{t'}\\
\text{s.t. }& x_{\tau+1} = A_\tau x_{\tau} + B_\tau u_{\tau} + v_\tau, \forall \tau \in [t:t'-1];\\
& x_{t} = y_t.
\end{align*}
Using this notation, we can express MPC with confidence coefficients as
\[\pi_t(x_t, \theta_t) = \psi_t^{t+k}(x_t, \{\lambda_t^{[\tau - t]} w_{\tau\mid t}\}_{\tau \in [t:t+k-1]}; \tilde{Q})_{u_{t}},\]
where the index $u_{t}$ denotes the corresponding entry in the predictive optimal solution. The perturbation bound in \cite[Theorem 3.3]{lin2021perturbation} states that for any $\tau \in [t, t']$, we have
\begin{align}\label{MPC-example:e3}
    \norm{\psi_t^{t'}(y_t, v_{t:t'}; \tilde{Q})_{u_\tau} - \psi_t^{t'}(y_t', v_{t:t'}'; \tilde{Q})_{u_\tau}} \leq C_0\left(\rho_0^{\tau - t}\norm{y_t - y_t'} + \sum_{j = t}^{t'} \rho_0^{\abs{\tau - j}} \norm{v_j - v_j'}\right),
\end{align}
where $C_0 > 0, \rho_0 \in (0, 1)$ are constants that depends on the system parameters including $a, b, \mu, \ell, \sigma,$ and $d_0$. And the inequality \eqref{MPC-example:e3} still holds if we replace the index $u_\tau$ on the left hand side by $x_\tau$. Therefore, we know that $\psi_t^{t+k}(y_t, v_{t:t+k}; \tilde{Q})_{u_{t}}$ is bounded Lipschitz in $y_t$ and $v_{t:t+k}$. Note that $\pi_t(x_t, \theta_t)$ is an affine function in its inputs $(x_t, \theta_t)$, i.e., $\ALG_t(x_t, \theta_t)$ can be expressed equivalently as
\begin{align}\label{example-proof:e2}
    \ALG_t(x_t, \theta_t) = - \bar{K}_t^{(k)} x_t - \sum_{\tau = t}^{t + k - 1} \lambda_t^{[\tau - t]} \bar{K}_t^{(k, \tau)} w_{\tau \mid t},
\end{align}
where the matrices $\bar{K}_t^{(k)}, \bar{K}_t^{(k, \tau)}$ only depends on $\{(A_t, B_t, Q_t, R_t)\}_{t \in \mathcal{T}}$ and $\tilde{Q}$. The superscript $k$ denotes the prediction horizon of the MPC we adopt, thus $k = T$ will give MPC future predictions all the way to the end of the online policy selection game. So the smoothness constants $\ell_{\pi, x} = \ell_{\pi, \theta} = 0$. Thus, we see that \Cref{assump:Lipschitz-and-smoothness} holds. We also see that the surrogate cost function $F_t$ is convex.

Next, we show a lemma about contractive perturbation and stability.

\begin{lemma}\label{lemma:MPC-properties}
Suppose Assumptions \ref{assump:MPC-example-bounds} and \ref{assump:MPC-example-controllability} hold. Recall that $C_0$ and $\rho_0$ are given in \eqref{MPC-example:e3}. Then, for any $\rho \in (\rho_0, 1)$, if the prediction horizon satisfies
\[k \geq \frac{1}{2}\log\left(C_0^3 ab \decayfactor_0 / (\decayfactor - \decayfactor_0)\right)/\log(1/\decayfactor_0),\]
the MPC with confidence coefficients policy class satisfies $\varepsilon$-time-varying contractive perturbation with $\varepsilon = +\infty$, $R_C = +\infty$, $C = C_0$ and decay factor $\rho$. It also satisfies $\varepsilon$-time-varying stability with $\varepsilon = +\infty$ and $R_S = \frac{C_0 (1 - \rho_0 + C_0) \bar{w}}{(1 - \rho_0)(1 - \rho)}$.
\end{lemma}

\begin{proof}[Proof of \Cref{lemma:MPC-properties}]
By the perturbation bound in \eqref{MPC-example:e3}, we see that for any $t' > t$,
\[\norm{\psi_t^T(y_t, 0_{\times (T-t)}; \tilde{Q})_{x_{t'}} - \psi_t^T(y_t', 0_{\times (T-t)}; \tilde{Q})_{x_{t'}}} \leq C_0 \rho_0^{t' - t} \norm{y_t - y_t'}.\]
Therefore, we obtain that for any $t' > t$,
\begin{align}\label{MPC-example:e4}
    \norm{(A_{t'-1} - B_{t'-1} \bar{K}_{t'-1}^{(T)})(A_{t'-2} - B_{t'-2} \bar{K}_{t'-2}^{(T)})\cdots (A_{t} - B_{t} \bar{K}_{t}^{(T)})} \leq C_0 \rho_0^{t' - t}.
\end{align}
Now we show that
\begin{align}\label{MPC-example:e5}
    \norm{(A_{t'-1} - B_{t'-1} \bar{K}_{t'-1}^{(k)})(A_{t'-2} - B_{t'-2} \bar{K}_{t'-2}^{(k)})\cdots (A_{t} - B_{t} \bar{K}_{t}^{(k)})} \leq C_0 \rho^{t' - t}.
\end{align}
To see this, we construct a sequence $\hat{v}_{t:T-1}$ such that $\hat{v}_{t+k} = - A_{t+k} \psi_t^{t+k}(x_t, 0_{\times k}; \tilde{Q})_{x_{t+k}}$ and $\hat{v}_\tau = 0$ for $\tau \not = t+k, \tau \in [t:T-1]$. We observe that
\begin{align*}
    &\norm{\psi_t^{t+k}(x_t, 0_{\times k}; \tilde{Q}) - \psi_t^T(x_t, 0_{\times (T-t)}; \tilde{Q})}\\
    ={}&\norm{\psi_t^T(x_t, \hat{v}_{t:T-1}; \tilde{Q}) - \psi_t^T(x_t, 0_{\times (T-t)}; \tilde{Q})}\\
    \leq{}& C_0 \rho_0^k \norm{\hat{v}_{t+k}} \leq C_0^2 a \rho_0^{2k} \norm{x_t},
\end{align*}
where we use \eqref{MPC-example:e3} in the last line. To simplify the notation, we define $M_\tau^{(p)} \coloneqq A_\tau - B_\tau \bar{K}_\tau^{(p)}$ and $\alpha \coloneqq C_0^2 \rho_0^{2k} a b$. By the above inequality, we see that
\begin{align}\label{MPC-example:e6}
    \norm{M_\tau^{(k)} - M_\tau^{(T)}} \leq \alpha, \forall \tau \in \mathcal{T}.
\end{align}
Therefore, we obtain that
\begin{subequations}\label{MPC-example:e7}
\begin{align}
    \norm{M_{t'-1}^{(k)}M_{t'-2}^{(k)} \cdots M_{t}^{(k)}} \leq{}& \sum_{j = 0}^{t' - t} \binom{t'-t}{j} C_0^{j+1} \rho_0^{t' - t} \alpha^j \label{MPC-example:e7:s1}\\
    \leq{}& C_0 \rho_0^{t' - t}\left(1 + C_0 \alpha\right)^{t' - t} \leq C_0 \rho^{t' - t}, \label{MPC-example:e7:s2}
\end{align}
\end{subequations}
where we use the decomposition $M_\tau^{(k)} = M_\tau^{(T)} + (M_\tau^{(k)} - M_\tau^{(T)})$, the triangle inequality, and \eqref{MPC-example:e6} in \eqref{MPC-example:e7:s1}; we use the condition that $k \geq \frac{1}{2}\log\left(C_0^3 ab \decayfactor_0 / (\decayfactor - \decayfactor_0)\right)/\log(1/\decayfactor_0)$ in \eqref{MPC-example:e7:s2}. This finishes the proof of \eqref{MPC-example:e5}.

Now we consider two trajectories that apply the same policy parameter sequence but start from different states $x_\tau$ and $x_\tau'$. For arbitrary $\theta_{\tau:t-1} = \Theta^{t-\tau}$, we see that
\begin{subequations}\label{MPC-example:e8}
\begin{align}
    &\norm{g_{t\mid \tau}(x_\tau, \theta_{\tau:t-1}) - g_{t\mid \tau}(x_\tau', \theta_{\tau:t-1})}\nonumber\\
    ={}& \norm{(A_{t-1} - B_{t-1} \bar{K}_{t-1}^{(k)})(A_{t-2} - B_{t-2} \bar{K}_{t-2}^{(k)})\cdots (A_{\tau} - B_{\tau} \bar{K}_{\tau}^{(k)})(x_\tau - x_\tau')} \label{MPC-example:e8:s1}\\
    \leq{}& C_0 \rho^{t - \tau} \norm{x_\tau - x_\tau'}, \label{MPC-example:e8:s2}
\end{align}
\end{subequations}
where we use the affine expression of $\pi_t$ \eqref{example-proof:e2} and the fact that these two trajectories experience the same sequence of disturbances and predictions. This finishes the proof of $\varepsilon$-time-varying contractive perturbation with $\varepsilon = +\infty$ and $R_C = +\infty$.

By the perturbation bound in \eqref{MPC-example:e3}, we also see that
\[\norm{\psi_t^{t+k}(0, v_{t:t+k-1}; \tilde{Q})_{u_t} - \psi_t^{t+k}(0, v_{t:t+k-1}'; \tilde{Q})_{u_t}} \leq C_0 \sum_{\tau = t}^{t+k-1} \rho_0^{\tau - t} \norm{v_\tau - v_\tau'}.\]
Combining this inequality with the affine relationship in \eqref{MPC-example:e4}, we see that
\begin{align}\label{MPC-example:e9}
    \norm{\bar{K}_t^{(k, \tau)}} \leq C_0 \rho_0^{\tau - t}.
\end{align}
Therefore, we obtain that
\begin{subequations}\label{MPC-example:e10}
\begin{align}
    &\norm{g_{t\mid \tau}(0, \theta_{\tau:t-1})}\nonumber\\
    ={}& \norm{\sum_{i = \tau}^{t - 1} (A_{t-1} - B_{t-1} \bar{K}_{t-1}^{(k)})\cdots (A_{i+1} - B_{i+1} \bar{K}_{i+1}^{(k)}) \left(w_i - \sum_{j = i}^{i+k-1} \lambda_i^{[j - i]} \bar{K}_i^{(k, j)} w_{j\mid i}\right)}\nonumber\\
    \leq{}& \sum_{i = \tau}^{t - 1} C_0 \rho^{t-1-i} \norm{w_i - \sum_{j = i}^{i+k-1} \lambda_i^{[j - i]} \bar{K}_i^{(k, j)} w_{j\mid i}} \label{MPC-example:e10:s1}\\
    \leq{}& \sum_{i = \tau}^{t - 1} C_0 \rho^{t-1-i}\left(\bar{w} + \sum_{j = i}^{i+k-1} C_0 \rho_0^{j-i} \bar{w}\right) \label{MPC-example:e10:s2}\\
    \leq{}& \frac{C_0 (1 - \rho_0 + C_0) \bar{w}}{(1 - \rho_0)(1 - \rho)}, \nonumber
\end{align}
\end{subequations}
where we use the triangle inequality and \eqref{MPC-example:e5} in \eqref{MPC-example:e10:s1}; we use the triangle inequality and \eqref{MPC-example:e9} in \eqref{MPC-example:e10:s2}. This finishes the proof of $\varepsilon$-time-varying stability.
\end{proof}

\subsection{Linear Feedback Control in Nonlinear System}\label{sec:nonlinear-control-details}

The example we present below is an extension of the nonlinear control model in \cite{li2022certifying}.

We consider the dynamical system $g_t(x_t, u_t) = A x_t + B u_t + \delta_t(x_t, u_t)$, where $(A, B)$ is stabilizable, i.e., there exists $K \in \mathbb{R}^{n\times m}$ such that $\rho(A - BK) < 1$. The cost functions are given by $f_t(x_t, u_t) = q(x_t, Q) + q(u_t, R)$, where $Q \succ 0$ and $R \succeq 0$. Under these assumptions, we know the discrete algebraic Riccati equation (DARE) has a unique solution $P \in \mathbb{S}^n$ such that $P \succeq 0$ and
\[P = Q + A^\top P A - A^\top P B (R + B^\top P B)^{-1} B^\top P A.\]
We let $\bar{K} \coloneqq (R + B^\top P B)^{-1} B^\top P A$ and introduce the notation $M \coloneqq A - B \bar{K}$. We know that $\rho(M) < 1$. Thus, there exists constants $C_0 > 0, \rho_0 \in (0, 1)$ such that the inequality $\norm{M^t} \leq C_0 \rho_0^t$ holds. To simplify the presentation, we assume $x_0 = 0$, $\mathcal{X} = B(0, R_x)$, and $\mathcal{U} = B(0, R_u)$ with $R_u > \norm{\bar{K}} R_x$. Our main assumption bounds the magnitude and Lipschitzness of the nonlinear residual $\delta_t$:

\begin{assumption}\label{assump:linear-feedback-nonlinear-control}
Function $\delta_t$ is Lipschitz continuous on $\mathcal{X} \times \mathcal{U}$ and satisfies that
\[\norm{\delta_t(0, 0)} \leq \bar{\delta}, \text{ and } \norm{\delta_t(x, u) - \delta_t(x', u')} \leq L_{\delta, x} \norm{x - x'} + L_{\delta, u} \norm{u - u'}\]
for any $x, x' \in \mathcal{X}$ and $u, u' \in \mathcal{U}$.
\end{assumption}

Compared with Assumption 1 in \cite{li2022certifying}, \Cref{assump:linear-feedback-nonlinear-control} does not require $\delta_t(0, 0) = 0$ or assume the Lipschitzness holds globally.

Define the set $\mathcal{K}_\sigma \coloneqq \{\hat{K} \in \mathbb{R}^{n \times m}\mid \norm{\hat{K} - \bar{K}} \leq \sigma\}$. The policy class in this example is $\pi_t(x_t, \theta_t) = K(\theta_t) x_t$, where $K: \Theta \to \mathcal{K}_\sigma$ is a $L_K$-Lipschitz function. Here, we do not directly use $\mathcal{K}_\sigma$ as our parameter set because it may be difficult to do projection and we want to leave the freedom to set $\Theta$ to the whole Euclidean space.

Now we provide an example where $\Theta = \mathbb{R}^d$. Let $K_1, K_2, \ldots, K_d$ be $d$ different linear feedback controllers in the set $\mathcal{K}_\sigma$. For any $\theta \in \Theta$, we define
\[K(\theta) \coloneqq \sum_{i=1}^d \textsf{softmax}(\theta)_i K_i, \text{ where } \textsf{softmax}(\theta)_i = \frac{e^{\theta_i}}{\sum_{j=1}^d e^{\theta_j}}, \text{ for } i = 1, \ldots, d.\]
Note that $K$ is a Lipschitz and differentiable function in this example.

\begin{lemma}\label{lemma:linear-feedback-nonlinear-control}
Under \Cref{assump:linear-feedback-nonlinear-control}, suppose $\sigma$ and the Lipschitzness constants satisfies that
\[0 < \sigma \leq \frac{R_u - \norm{\bar{K}} R_x}{R_x} \text{, and } \rho_0 + C_0 \left(\sigma \norm{B} + L_{\delta, x} + L_{\delta, u} (\norm{\bar{K}} + \sigma)\right) < 1.\]
Let $\rho \coloneqq \rho_0 + C_0 \left(\sigma \norm{B} + L_{\delta, x} + L_{\delta, u} (\norm{\bar{K}} + \sigma)\right)$. We also assume that $\bar{\delta} < (1 - \rho) R_x/C_0$. Then, the joint system of $(g_t, f_t, \pi_t)$ satisfies time-invariant contractive perturbation with 
\[C = C_0, \rho = \rho, R_C = \frac{(1 - \rho)R_x - C_0 \bar{\delta}}{(1 - \rho)C_0}.\]
It also satisfies time-invariant stability with $R_S = \frac{C_0 \bar{\delta}}{1 - \rho}$.
\end{lemma}

The assumptions of \Cref{lemma:linear-feedback-nonlinear-control} requires the nonlinear residual to be sufficiently small and Lipschitz, and the radius $\sigma$ of the set $\mathcal{K}_\sigma$ should also be sufficiently small. Such assumptions are also required by prior works \cite{li2022certifying,qu2021exploiting}.

One can use \Cref{lemma:from-static-to-time-varying} to extend the time-invariant contractive perturbation/stability to the $\varepsilon$-time-varying version. And to satisfy the assumptions of \Cref{lemma:from-static-to-time-varying}, we need to additionally assume $\bar{\delta}$ is sufficiently small such that
\[\bar{\delta} \leq \frac{(1 - \rho) R_x}{C_0 (C_0 + 1)^2}.\]

\begin{proof}[Proof of \Cref{lemma:linear-feedback-nonlinear-control}]
We first study stability. When the policy parameter is fixed to be $\theta$, the control action is $u_t = - K(\theta) x_t$ for all time step $t$. We see that
\begin{align}\label{lemma:linear-feedback-nonlinear-control:e1}
    x_{t+1} &= A x_t + B u_t + \delta_t(x_t, u_t) = (A - B K(\theta)) x_t + \delta_t(x_t, - K(\theta) x_t)\nonumber\\
    &= M x_t + \underbrace{\left(B (\bar{K} - K(\theta)) x_t + \delta_t(x_t, - K(\theta) x_t)\right)}_{\coloneqq r_t(x_t)}.
\end{align}
Using \Cref{assump:linear-feedback-nonlinear-control}, we know that if $x_t \in \mathcal{X}$,
\begin{align}\label{lemma:linear-feedback-nonlinear-control:e2}
    \norm{r_t(x_t)} \leq \bar{\delta} + \left(\sigma \norm{B} + L_{\delta, x} + L_{\delta, u} (\norm{\bar{K}} + \sigma)\right) \norm{x_t}.
\end{align}
Therefore, for any $0 \leq \tau < t$, if $x_\tau, \ldots, x_{t-1} \in \mathcal{X}$, we have
\begin{subequations}\label{lemma:linear-feedback-nonlinear-control:e3}
\begin{align}
    \norm{x_t} \leq{}& \norm{M^{t-\tau} x_\tau} + \norm{\sum_{j = \tau}^{t-1} M^{t-1-j} r_j(x_j)} \label{lemma:linear-feedback-nonlinear-control:e3:s1}\\
    \leq{}& C_0 \rho_0^{t-\tau} \left(\norm{x_\tau} + \sum_{j = \tau}^{t-1} \rho_0^{\tau - 1 - j} \norm{r_j(x_j)}\right) \label{lemma:linear-feedback-nonlinear-control:e3:s2}\\
    \leq{}& C_0 \rho_0^{t-\tau} \left(\norm{x_\tau} + \left(\sigma \norm{B} + L_{\delta, x} + L_{\delta, u} (\norm{\bar{K}} + \sigma)\right) \sum_{j = \tau}^{t-1} \rho_0^{\tau - 1 - j} \norm{x_j}\right) + \frac{C_0 \bar{\delta}}{1 - \rho_0}, \label{lemma:linear-feedback-nonlinear-control:e3:s3}
\end{align}
\end{subequations}
where we use \eqref{lemma:linear-feedback-nonlinear-control:e1} and the triangle inequality in \eqref{lemma:linear-feedback-nonlinear-control:e3:s1}; we use the property of matrix $M$ in \eqref{lemma:linear-feedback-nonlinear-control:e3:s2}; we use \eqref{lemma:linear-feedback-nonlinear-control:e2} in \eqref{lemma:linear-feedback-nonlinear-control:e3:s3}.

Let $\rho \coloneqq \rho_0 + C_0 \left(\sigma \norm{B} + L_{\delta, x} + L_{\delta, u} (\norm{\bar{K}} + \sigma)\right)$ and $\beta \coloneqq \frac{C_0 \bar{\delta}}{1 - \rho}$. Now we use induction to show that if $\norm{x_\tau} \leq \frac{R_x - \beta}{C_0}$, then for every $t \geq \tau$, $x_t \in \mathcal{X}$ and $\norm{x_t} \leq C_0 \rho^{t-\tau} \norm{x_\tau} + \beta$.

Note that the statement holds for $\tau$. Suppose it holds for $\tau, \ldots, t-1$, then for $t$, by \eqref{lemma:linear-feedback-nonlinear-control:e3}, we see that
\begin{subequations}\label{lemma:linear-feedback-nonlinear-control:e4}
\begin{align}
    \norm{x_t} \leq{}& C_0 \rho_0^{t-\tau} \norm{x_\tau} + (\rho - \rho_0) \rho_0^{t-\tau} \sum_{j = \tau}^{t-1} \rho_0^{\tau - 1 - j} \norm{x_j} + \frac{C_0 \bar{\delta}}{1 - \rho_0} \label{lemma:linear-feedback-nonlinear-control:e4:s1}\\
    \leq{}& C_0 \rho_0^{t-\tau} \norm{x_\tau} + (\rho - \rho_0) \rho_0^{t-\tau} \sum_{j = \tau}^{t-1} \rho_0^{\tau - 1 - j} \left(C_0 \rho^{j-\tau} \norm{x_\tau} + \beta\right) + \frac{C_0 \bar{\delta}}{1 - \rho_0} \label{lemma:linear-feedback-nonlinear-control:e4:s2}\\
    \leq{}& C_0 \rho_0^{t-\tau} \norm{x_\tau} + C_0 (\rho^{t-\tau} - \rho_0^{t-\tau}) \norm{x_\tau} + \frac{(\rho - \rho_0)\beta}{1 - \rho_0} + \frac{C_0 \bar{\delta}}{1 - \rho_0} \label{lemma:linear-feedback-nonlinear-control:e4:s3}\\
    ={}& C_0 \rho^{t-\tau} \norm{x_\tau} + \beta,\nonumber
\end{align}
\end{subequations}
where we use \eqref{lemma:linear-feedback-nonlinear-control:e3} in \eqref{lemma:linear-feedback-nonlinear-control:e4:s1}; we use the induction assumption in \eqref{lemma:linear-feedback-nonlinear-control:e4:s2} and rearrange the terms in \eqref{lemma:linear-feedback-nonlinear-control:e4:s3}. Therefore, $x_t$ is also in $\mathcal{X}$, and we have shown that $x_t \in \mathcal{X}$ and $\norm{x_t} \leq C_0 \rho^{t-\tau} \norm{x_\tau} + \beta$ hold for all $t$ by inductions. This finishes the proof for stability.

Now, we show that contractive perturbation also holds. Suppose $x_\tau', \ldots, x_t'$ is another trajectory that starts from a different state $x_\tau'$, which also satisfies $\norm{x_\tau'} \leq \frac{R_x - \beta}{C_0}$. Note that this trajectory evolves according to $x_{j+1}' = M x_j' + r_j(x_j'), \forall j \geq \tau$. By the Lipschitzness of the nonlinear residual $\delta_j$, we see that
\[\norm{r_t(x_t) - r_t(x_t')} \leq \left(\sigma \norm{B} + L_{\delta, x} + L_{\delta, u} (\norm{K} + \sigma)\right) \norm{x_t - x_t'}.\]
Therefore, we obtain that
\begin{align*}
    \norm{x_t - x_t'} \leq{}& \norm{M^{t-\tau} (x_\tau - x_\tau')} + \norm{\sum_{j=\tau}^{t-1} M^{t-1-j} \left(r_j(x_j) - r_j(x_j')\right)}\\
    \leq{}& C_0 \rho_0^{t-\tau} \left(\norm{x_\tau - x_\tau'} + \sum_{j = \tau}^{t-1} \rho_0^{\tau - 1 - j} \norm{r_j(x_j) - r_j(x_j')}\right)\\
    \leq{}& C_0 \rho_0^{t-\tau} \left(\norm{x_\tau - x_\tau'} + \left(\sigma \norm{B} + L_{\delta, x} + L_{\delta, u} (\norm{K} + \sigma)\right) \sum_{j = \tau}^{t-1} \rho_0^{\tau - 1 - j} \norm{x_j - x_j'}\right).
\end{align*}
Using the same technique as \eqref{lemma:linear-feedback-nonlinear-control:e4}, we can show that
\[\norm{x_t - x_t'} \leq C_0 \rho^{t-\tau}\norm{x_\tau - x_\tau'}\]
holds if $\norm{x_\tau} \leq \frac{R_x - \beta}{C_0}$ and $\norm{x_\tau'} \leq \frac{R_x - \beta}{C_0}$.
\end{proof}

\section{Numerical experiment details}
\label{sec:numerical-appendix}
In this appendix we provide details on the numerical experiments of \Cref{sec:Numerics}.
We also present a third experiment that compares GAPS to our bandit-based algorithm (BAPS, \Cref{alg:batched-exp3-policy-selection}) in a setting where the policy class for GAPS is a strict superset of the policy class for BAPS.

\subsection{Adapting a scalar confidence parameter for MPC}
\label{appendix:ours-vs-lambda-confident}
This section provides details on the baseline algorithm \citet{li2021robustness}
to which GAPS was compared in \Cref{sec:Numerics}.
The main result was shown in \Cref{fig:ours-vs-lambda-confident}.

The method of Li et al.\ \citet{li2021robustness} tunes a scalar confidence parameter for model-predictive control in a LTI system with imperfect disturbance predictions.
The setting is similar to \Cref{example:MPC-confidence} but has several restrictions:
1) the dynamics matrices $A, B$ and cost matrices $Q, R$ are time-invariant,
2) the disturbance predictions $\{ \widehat w_t \}_{t=0}^{T-1}$ are static and known before the game begins,
and 3) the policy is parameterized by a single confidence parameter $\lambda \in [0, 1]$ instead of separate confidence parameters $(\lambda^{[0]}, \dots, \lambda^{[k-1]}) \in [0, 1]^k$ for each prediction in the planning horizon.

Li et al. \citet{li2021robustness} propose a ``follow-the-leader''-style algorithm where the entire trajectory history is used to select $\lambda$ at each time step.
Their method commits the $u_{t\mid t}$ entry of the optimal solution
\begin{align*}
    \argmin_{u_{t:t+k-1\mid t}} & \quad \sum_{\tau = t}^{t + k-1}\left(\quadratic\left(x_{\tau\mid t}, Q\right) + \quadratic\left(u_{\tau\mid t}, R\right)\right) + \quadratic\left(x_{t+k\mid t}, P\right) \\*
    \text{s.t.} & \quad x_{\tau+1\mid t} = A x_{\tau\mid t} + B u_{\tau\mid t} + \lambda_t \widehat w_\tau, \text{ for }\tau = [t:t+k-1],\\*
    & \quad x_{t\mid t} = x_t,
\end{align*}
where $P$ solves the discrete-time algebraic Riccati equation (DARE) for $A, B, Q, R$.
The confidence parameter $\lambda_t$ is determined according to
\begin{align*}
  \lambda_t = & \frac{\sum_{s=0}^{t-1}\left(\eta(w;s,t-1)\right)^{\top} H \left(\eta(\widehat{w};s,t-1)\right)}{\sum_{s=0}^{t-1}\left(\eta(\widehat{w};s,t-1)\right)^{\top} H \left(\eta(\widehat{w};s,t-1)\right)},
  \quad \text{where} \quad  \eta(w;s,t)\coloneqq \sum_{\tau=s}^{t}\left(F^\top\right)^{\tau-s} P w_\tau,
\end{align*}
where $H = B(R + B^\top P B)^{-1} B^\top$
and
\[
F = A - BK^\star = A - B(R + B^\top P B)^{-1} B^\top P A
\]
is the closed-loop linear dynamics matrix induced by the LQR-optimal infinite-horizon linear policy.

We compare the ability of each algorithm to quickly adapt to a change in prediction accuracy.
We consider the unstable scalar system
\(
    x_{t+1} = 2 x_t + u_t + w_t
\)
with the LQR costs $Q = R = 1$.
We construct sinusoidal disturbances $w_t$
and noise-corrupted disturbance predictions $\widehat w_t$ following
\[
    w_t = \sin(2\pi f t + p), \quad \quad
    \widehat w_t = w_t + n_t, \quad \quad
    n_t \sim
    \begin{cases}
        \operatorname{Uniform}([-2, 2]) &: t \leq T/4 \\
        \operatorname{Uniform}([-0.02, 0.02]) &: \mathrm{otherwise},
    \end{cases}
\]
where $f > 0$ and $p \in [0, 2\pi)$ are frequency and phase constants. 
In other words, the predictions are initially noisy but become accurate after $t = T/4$.
We set $T = 400$ and perform $100$ random trials.
In each trial, the frequency $f$ is sampled from a log-uniform distribution on $[0.01, 0.1]$
and the phase $p$ is sampled uniformly from $[0, 2\pi)$.
Results and discussion are found in \Cref{sec:Numerics}.

\subsection{Adapting a linear controller for a nonlinear time-varying system}
\label{appendix:pendulum}
This section provides details on the experimental setting and baseline controller for the second experiment in \Cref{sec:Numerics},
in which GAPS is compared to a LQR baseline for tuning feedback control gains in the nonlinear inverted pendulum.
The main result was shown in \Cref{fig:pendulum}.

\begin{figure}[H]
    \centering
    \usetikzlibrary{calc}

\begin{tikzpicture}[every node/.style={outer sep=0pt}, scale=1.2]
	\tikzstyle{ground}=[fill, pattern=north east lines, draw=none]
	\tikzset{>=latex}
	\begin{scope}[rotate=-15]
		\draw (0, 0) -- ++(0, 15mm) coordinate (masspt);
		\node[circle, radius=0.5mm, inner sep=0.3mm, fill] at (0, 0) {};
		\node[circle, inner sep=1mm, fill] (mass) at (masspt) {};
		\node[right=0mm of mass, anchor=west] {$m$};

		\draw[|<->|] (-7mm, 0) -- ($(masspt) + (-7mm, 0)$)
			node [midway, inner sep=1mm, fill=white] {$\ell$};
	\end{scope}

	\draw (0, 0) -- ++(0, 7mm) coordinate (angleline) -- ++(0, 1mm);
	\draw [rotate=90] (angleline) arc (00:-15:7mm) node[pos=0.65, above]{$\phi$};

	\draw [->] (mass) -- ++(0, -8mm) node[midway, right] {$G$};
	\begin{scope}[rotate=40]
	    \draw [->] (0, 3mm) arc (90:-20:3mm) node[pos=0.85, right]{$u$};
	\end{scope}

	\node[ground, minimum height=3mm, minimum width=10mm, anchor=north] (g) at (0, 0) {};
	\draw (g.north east) -- (g.north west);
\end{tikzpicture}
    \caption{The inverted pendulum system.}
    \label{fig:pendulum-diagram}
\end{figure}

We consider the inverted pendulum illustrated in \Cref{fig:pendulum-diagram}.
The system is a point mass on a massless rod with a stationary and frictionless pivot.
The input is a torque about the pivot.
The continuous-time nonlinear dynamics are given by
\[
    \ddot \phi = \frac{G}{\ell} \sin \phi + \frac{u}{m \ell^2},
\]
where $\phi$ is the pivot angle,
$u$ is the applied torque,
$m$ is the mass,
$\ell$ is the rod length,
and $G$ is the gravitational constant.
In first-order state-space form, the state is $x = (\phi, \dot \phi)$.
Near the unstable equilibrium of $x = \zero$, the system is well-approximated by its state-space linearization:
\begin{equation}
\label{eq:pendulum-linear}
    \frac{d}{dt}
    x
    \approx
    \begin{pmatrix}
        0 & 1 \\
        \nicefrac{G}{\ell} & 0 \\
    \end{pmatrix}
    x
    +
    \begin{pmatrix}0 \\ \nicefrac{1}{m \ell^2} \end{pmatrix} u.
\end{equation}
A linear feedback policy takes the form
\(
    u = -k_p \phi - k_d \dot \phi
\)
for the positive constants $k_p,\ k_d$,
also known as a proportional-derivative (PD) policy.
We apply GAPS to tune the parameters $\theta = (k_p, k_d)$.
We simulate time-varying pendulum mass $m_t$ with regularly-spaced step changes.

We measure policy performance by the LQR cost 
\begin{equation}
\label{eq:pendulum-lqr-cost}
    f_t(x_t, u_t) = \Delta_t \cdot (\phi_t^2 + \dot \phi_t^2 + 0.1 u_t^2),
\end{equation}
where $\Delta_t$ is the discretization interval.
We compare GAPS to a natural baseline 
that deploys the policy
$u_t = K^\star (m_t) x_t$,
where $K^\star$ is the optimal PD policy for the infinite-horizon summation of the LQR cost \eqref{eq:pendulum-lqr-cost}
when the pendulum dynamics for mass $m_t$ are linearized as in \eqref{eq:pendulum-linear} and discretized with a zero-order hold.

Although this is not the exact offline optimal for the finite-horizon time-varying problem, it is a good approximation when mass $m_t$ changes infrequently,
and is practical and realistic.

We model disturbances via a Ornstein-Uhlenbeck-type random walk with state $s_t \in \bR$
and dynamics $s_{t+1} = \gamma s_t + \delta_t$,
where
$\gamma \in [0, 1)$
and $\delta_t$ is Gaussian-distributed with zero mean and variance $\sigma^2$. 
When $\gamma = 0$, the process becomes i.i.d. uniform noises.
The disturbance enters as a perturbation of the velocity, i.e. the system is driven by the true input
$u_t + m \ell^2 s_t$.
We test two scenarios: the i.i.d. case $(\gamma = 0, \sigma = 8)$
and the random walk $(\gamma = 0.95, \sigma = 0.5)$.
These values lead to similar closed-loop state magnitudes under optimal LQR control.
We deploy GAPS with buffer size $B = 400$ and learning rate $\eta = 0.03$.

This setting is intended to match the spirit
of the nonlinear system in \Cref{example:nonlinear-control}
and the local regret bound for nonconvex surrogate cost in \Cref{thm:GAPS-nonconvex-regret}.
A simple randomized search yields two controller parameters that certify the nonconvexity of the surrogate cost $F_1$.
However, we relax some requirements:
1) we have unbounded (but concentrated) noise,
2) we do not parameterize the controller to allow $\Theta$ to be a full Euclidean space,
and
3) we do not restrict the controller set or noise magnitude to ensure that the conditions of \Cref{lemma:linear-feedback-nonlinear-control} are satisfied.
Nonetheless, we observe strong performance from GAPS.

\subsection{Selecting the horizon in MPC: discrete and continuous approaches}
\label{appendix:integrator}

In this section, we compare GAPS to the bandit-based algorithm (BAPS, \Cref{alg:batched-exp3-policy-selection}) in a setting where both finite and infinite policy classes are well motivated. 
We consider MPC with disturbance predictions that are accurate in the near future but inaccurate in the far future.
The inaccurate far-future predictions mean it is counterproductive to plan over a long horizon.
One natural solution is to optimize the MPC planning horizon, which yields a small finite policy class.

On the other hand, we can also apply MPC with confidence coefficients (\Cref{example:MPC-confidence}) and optimize the confidence coefficients using GAPS.

The continuous policy class is a strict superset of the finite class.
To see this, first let $\controlnum$ denote the maximum MPC horizon we wish to consider.
Let $\MPC_\controlidx(\theta)$ denote the MPC algorithm \eqref{eq:mpc-opt-confidence} with horizon $\controlidx \leq \controlnum$ and confidences $\theta \in [0, 1]^\controlidx$.
When the MPC terminal cost matrix $\tilde Q$ is the solution of the discrete-time algebraic Riccati equation for the associated infinite-horizon LQR problem, the Bellman optimality condition and the convexity of the MPC optimization problem imply that selecting from
$\{ \MPC_0(\one), \dots,  MPC_\controlnum(\one) \}$
is equivalent to selecting from the parameter set
\[
        \Theta_\dis~=~\{
        [0\ \cdots\ 0],
        [1\ 0\ \cdots\ 0],
        [1\ 1\ 0\ \cdots\ 0],
        \dots,
        [1\ \cdots\ 1]
        \}
\]
for $\MPC_\controlnum$.
For the continuous parameter set $\Theta_\cts = [0, 1]^\controlnum$ we have
$\Theta_\dis \subsetneq \Theta_\cts$,
Therefore, in the long term GAPS should outperform BAPS,
but it is not clear which will adapt faster.

We consider a two-dimensional discretized double integrator system subject to disturbances in both velocity and position, with quadratic cost.
The continuous-time system is discretized with forward Euler integration over the interval $\Delta_t$, resulting in the dynamics and cost matrices
\[
    A = \begin{pmatrix}
        I_2 & \Delta_t I_2 \\
        \zero & I_2
    \end{pmatrix} \ ,
    \quad \quad
    B = \begin{pmatrix}
        \zero \\
        \Delta_t I_2
    \end{pmatrix}\ ,
    \quad \quad
    Q = s_n \begin{pmatrix}
        I_2 & \zero \\
        \zero & s_Q I_2
    \end{pmatrix} \ ,
    \quad \quad
    R = s_n s_R I_2,
\]
with positive cost scale factors $s_Q, s_R, s_n$.
The cost scale factors $s_Q, s_R$ balance the costs of position error, velocity error, and control effort.
The overall scale factor $s_n$ is set to ensure compatibility with the cost bound assumptions of \Cref{alg:batched-exp3-policy-selection}. Note that the optimal policy is invariant to changes in $s_n$.

The true disturbances $\{w_t\}_{t=0}^{H-1}$ are sampled i.i.d. from a uniform distribution on $[-\Delta_t, \Delta_t]^4$.
To model the idea that predictions of the near future are more accurate than predictions of the distant future,
we construct noisy predictions using the process
\[
    w_{t+\tau \mid t} \coloneqq w_{t+\tau} + s_\epsilon \Delta_t \sum_{i = t}^{t+\tau} \epsilon_i,
\]
where $s_\epsilon$ is a scale factor and $\{\epsilon_t\}_{t=0}^{H-1}$ are sampled i.i.d. from a uniform distribution on $[-1, 1]^4$.

In this experiment, we select the discretization interval $\Delta_t = 0.1$
and the scale factors
$s_Q = 0.1$,
$s_R = 0.01$,
and $s_\epsilon = 0.25$.
These choices ensure that the optimal cost is strongly affected by the MPC horizon (or confidence weights).
We use the GAPS/BAPS learning rates and batch sizes suggested by our regret bounds.

    Our regret bounds are realized by particular choices of learning rate and EXP3 batch size that depend on the decay constants $(C, \rho)$ induced by the MPC policies.
    The conservative bounds on $(C, \rho)$ yielded by 
    \Cref{lemma:MPC-properties}
    will lead to very slow learning rates and long batches.
    Therefore, we appeal to the fact that MPC is solving the same optimization as the infinite-horizon LQR in this setting, and use the constants of the optimal linear policy $K^\star$.
    We compute $\rho$ empirically from the spectrum of $A-BK^\star$ and let
    $C = \max \left\{ \rho^{-n} \norm{(A - BK)^n} \right\}_{n=0}^N$
    for sufficiently large $N$.

\begin{figure*}[ht]
   
    \begin{subfigure}[t]{0.52\textwidth}
        \centering
        \includegraphics[height=4.5cm]{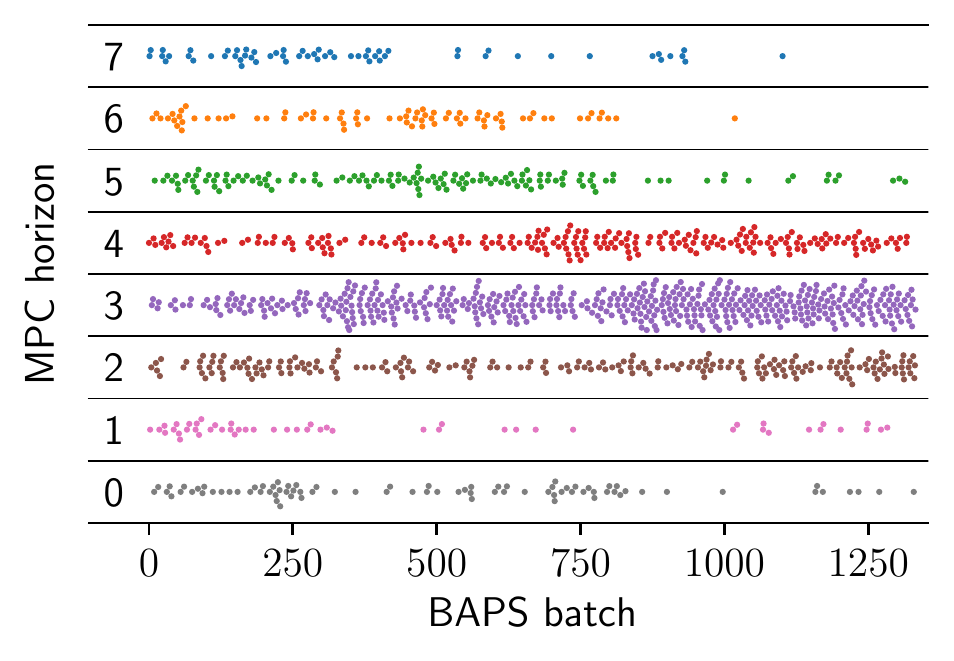}
        \caption{
            BAPS: Each dot represents the MPC horizon selected for a single batch.
        }
        \label{fig:mpc-horizon-exp3-scatter}
    \end{subfigure}
    \hfill
    \begin{subfigure}[t]{0.42\textwidth}
        \centering
        \includegraphics[trim={2.5mm 0 0 0}, clip, height=4.5cm]{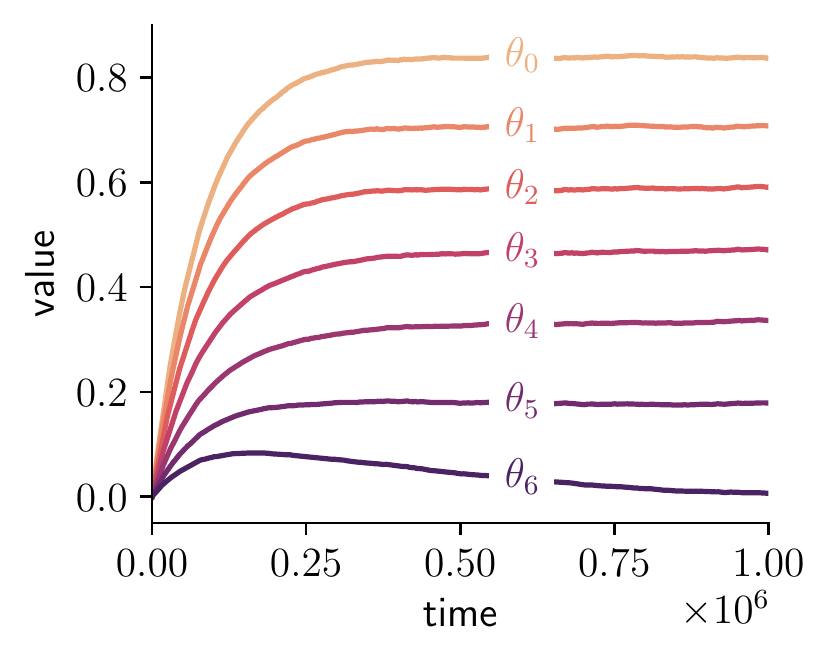}
        \caption{
            GAPS: Evolution of prediction confidence parameters $\theta_i$.
        }
    \label{fig:oco-parameter-evolution}
    \end{subfigure}
    \hfill
    \caption{
        Application of GAPS and BAPS to model-predictive control with noisy disturbance predictions.
        GAPS tunes per-step confidence weights $\theta_i \in [0, 1]$,
        while BAPS selects between different horizon lengths.
        Both methods learn to disregard far-future predictions.
    }
    \label{fig:exp3-horizon-selection}
\end{figure*}

\Cref{fig:mpc-horizon-exp3-scatter} visualizes
the behavior of BAPS on the policies
$\{ \MPC_0(\one), \dots, \MPC_\controlnum (\one) \}$.
A dot at the coordinate $(\beta, \controlidx)$ indicates that BAPS selected $\MPC_\controlidx(\one)$ during the $\beta^\mathrm{th}$ batch.
The selections eventually concentrate on the optimal $\MPC_3(\one)$,
with some continued exploration of the nearly-optimal $\MPC_2(\one)$ and $\MPC_4(\one)$. For GAPS, the policy is $\MPC_\controlnum$ with initial trust parameters $\theta_0 = \zero$.
\Cref{fig:oco-parameter-evolution} shows the evolution of the trust parameter $\theta_t$ under GAPS.
It quickly converges to values that are (roughly) inversely proportional to the prediction errors.

\begin{figure}[ht]
    \centering
    \includegraphics[height=4.0cm]{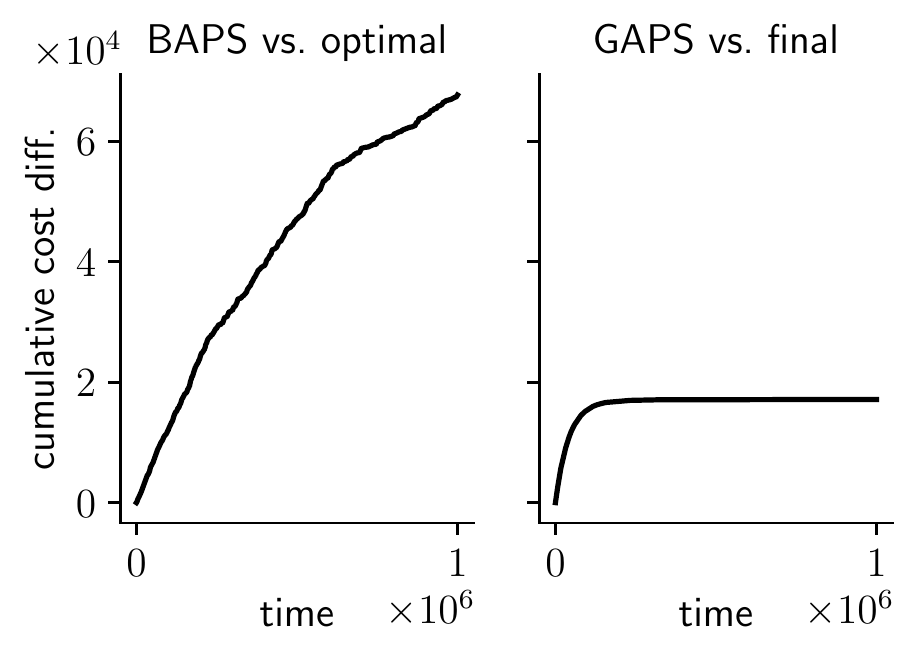}
    \caption{
        True regret of BAPS;
        ``Regret'' of GAPS compared to the final $\theta_{T-1}$ GAPS selects.
    }
    \label{fig:regret-curves}
\end{figure}
Due to the gap between $\Theta_\cts$ and $\Theta_\dis$,
we cannot evaluate the regret of each algorithm with respect to the same optimum.
In this example, we have
\(
    J(\theta_{T-1})
    =
    0.89
    \min_{\theta \in \Theta_\dis} J(\theta)
\),
where $\theta_{T-1}$ is the final parameter value selected by GAPS.
Therefore, we plot separate regret curves in \Cref{fig:regret-curves}.
Thanks to the LTI dynamics, convex costs, and time-invariant noise properties,
we have
$\theta_{T-1} \approx \min_{\theta \in \Theta_\cts} J(\theta)$.
Therefore, we use $\theta_{T-1}$ as the regret baseline for GAPS.
Over the simulation timeframe, BAPS exhibits characteristic $T^{2/3}$ regret while GAPS reaches approximately constant regret. Since GAPS also competes against a stronger baseline, we see the significant benefit of exploiting the continuous problem structure.

\end{document}